\documentclass[12pt,oneside]{amsart}

\usepackage{rotating}

\usepackage{amsmath}
\usepackage{amssymb, latexsym, amsfonts, amscd, amsthm, mathrsfs, enumerate, esint}
\usepackage[usenames,dvipsnames]{color}
\usepackage[shortlabels]{enumitem}
\usepackage{graphicx}
\usepackage{mathtools} 
\usepackage[colorlinks=true,allcolors=blue]{hyperref}
\usepackage{url}
\usepackage{stmaryrd}
\usepackage[normalem]{ulem}

\usepackage{cite}

\usepackage{geometry}
\usepackage{chngcntr}
\counterwithin{equation}{section}
\usepackage{setspace}
\newcommand*{\doi}[1]{\href{http://dx.doi.org/\detokenize{#1}}{doi}}
\allowdisplaybreaks
\newcommand{\qm}{?/1}
\newcommand{\key}{B}
\newcommand{\visited}{{ \text {Visited}}}
\newcommand{\setzero}{{ \text {Set}_0}}
\newcommand{\setone}{{ \text {Set}_1}}
\newcommand{\settwo}{{ \text {Set}_2}}

\newcommand{\stagei}{{ \text {Stage}_i}}
\newcommand{\stageinf}{{ \text {Stage}_\infty}}
\newcommand{\stageminusone}{{ \text {Stage}_{-1}}}
\newcommand{\djokovich}{{  {\daleth}_i}}

\newcommand{\stageiminusone}{{ \text {Stage}_{i-1}}}
\newcommand{\eventone}{{ \text {Event}_{1,i}}}
\newcommand{\eventtwo}{{ \text {Event}_{2,i}}}
\newcommand{\eventthree}{{ \text {Event}_{3,i}}}

\newcommand{\eventfourplus}{{ \text {Event}^+_{4,i}}}
\newcommand{\eventfourminus}{{ \text {Event}^-_{4,i}}}
\newcommand{\eventfourpm}{{ \text {Event}^\pm_{4,i}}}

\newcommand{\intervali}{{ \text {Interval}_i}}
\newcommand{\intervalleft}{{ \text {Interval}_{\text{left}}}}
\newcommand{\intervalright}{{ \text {Interval}_{\text{right}}}}
\newcommand{\intervalminus}{{ \text {Interval}_{i-1}}}
\newcommand{\intervalzero}{{ \text {Interval}_0}}
\newcommand{\intervalminusone}{{ \text {Interval}_{-1}}}
\newcommand{\setthreek}{{ \text {Set}_{3,r}}}

\newcommand{\setthree}{{ \text {Set}_3}}

\newcommand{\setqm}{{ \text {Set}_{\qm}}}
\newcommand{\setqmit}{{ \text {Set}_{\,\qm}}}

\newcommand{\frozen}{{ \text {Frozen}}}
\newcommand{\ratio}{{\gamma}}

\newcommand{\range}{{ \text {Range}}}

\newcommand{\local}{{ \text {Local}}}
\newcommand{\extreme}{{ \text {Extreme}}}
\newcommand{\hidden}{{ \text {Hidden}}}

\newcommand{\X}{{ \text {Parity}_Q}}
\newcommand{\parity}{{ \text {Parity}}}
\newcommand{\Xt}{{ \text {Parity}_{Q^t}}}
\newcommand{\Xs}{{ \text {Parity}_{Q^s}}}
\newcommand{\Xszero}{{ \text {Parity}_{Q^{s_0}}}}

\newcommand{\E}{\mbox{$\bf E$}}

\newcommand{\reveal}{{\textbf{Reveal}}}
\newcommand{\revealtit}{{\textbf{Reveal}}^{\,t}}
\newcommand{\last}{{\text{Last}}}
\newcommand{\rightzeros}{{\text{Right Zeros}}}

\renewcommand{\P}{{\bf P}}

\newenvironment{pfofthm}[1]
{\par\vskip2\parsep\noindent{\em Proof of Theorem\ #1. }}{{\hfill
$\square$}

\par\vskip2\parsep}

\newenvironment{pfofprop}[1]
{\par\vskip2\parsep\noindent{\em Proof of Proposition\ #1. }}{{\hfill
$\square$}
\par\vskip2\parsep}

\usepackage{amsmath}
\usepackage{amssymb,amsbsy,amsthm}
\usepackage{graphicx}

\usepackage{enumerate}
\usepackage{cases}

\usepackage{hyperref}

\allowdisplaybreaks[1]

\DeclareMathOperator{\N}{\mathbb{N}} % Natural numbers
\DeclareMathOperator{\Z}{\mathbb{Z}} % Integers

\newcommand{\p}{\mathbb{P}}

\renewcommand{\P}{\mathbb{P}}
\renewcommand{\leq}{\leqslant}
\renewcommand{\geq}{\geqslant}

\newtheorem{theorem}[equation]{Theorem}
\newtheorem{lemma}[equation]{Lemma}

\newtheorem{proposition}[equation]{Proposition}

\newtheorem{corollary}[equation]{Corollary}

\newtheorem{definition}[equation]{Definition}

\mathtoolsset{showonlyrefs}

\thanks{We would like to thank Lionel Levine for helpful conversations. CH and YH thank the NSF for their support through grant DMS-1954059. DR thanks the NSF for their support through grant DMS-1855568.}
\begin{document}

\title{Active Phase for the Stochastic Sandpile on $\mathbb{Z}$}
\author{Christopher Hoffman \and Yiping Hu \and Jacob Richey \and Douglas Rizzolo}
%\date{\today}

\keywords{stochastic sandpile, Abelian property, self-organized criticality}

%%%%%%%%%%%%%%%%
%%% Abstract %%%
%%%%%%%%%%%%%%%%

\begin{abstract}
We prove that the critical value of the one-dimensional Stochastic Sandpile Model is less than one. This verifies a conjecture of Rolla and Sidoravicius \cite{RolSid12}.
\end{abstract} 
%\begin{abstract}
%We consider activated random walks on $\mathbb{Z}$ and show that the system stays active for any sleep rate $\lambda < \infty$, as long as the density $ \zeta $ is close enough to $1$.
%In this model, each particle performs a continuous-time simple symmetric random walk, and falls asleep at rate $\lambda$.
%A sleeping particle is reactivated upon meeting another particle.
%The proof uses a block argument introduced by Basu, Ganguly and the first author, further developed in the third author's monograph, with extra elements to handle dependence across blocks.
%\end{abstract} 

\maketitle

%%%%%%%%%%%%%%%%%%%%%%%%%%%%%%%%%%
\section{Introduction}
\label{sec:intro}
%%%%%%%%%%%%%%%%%%%%%%%%%%%%%%%%%%

Sandpile models have a long history in the statistical mechanics literature as paradigms of self-organized criticality \cite{BakTangWiesenfeld87, manna1991two, dhar1999abelian}. Of particular importance is the Stochastic Sandpile Model (SSM).  This abelian variant of Manna's model  \cite{manna1991two} is widely believed to exhibit universality \cite{biham}. Since the SSM gained popularity in the mathematics community the prime challenge has been to prove that the critical density for the one-dimensional Stochastic Sandpile Model is less than one \cite{RolSid12}. In this paper we prove this conjecture.  

%Following \cite{RolSid12}, t
The SSM is an interacting particle system on $\Z$ defined as follows. At each time, the state of the system is given by a function $\eta_t:\Z\to \{0,1,2,\dots\}$, where $\eta_t(x)$ represents the number of particles at site $x$ at time $t$.  Sites $x$ such that $\eta_t(x)\leq 1$ are considered stable while sites with $\eta_t(x)\geq 2$ are unstable.  Unstable sites topple independently at exponential rate $1$ and when an unstable site topples two particles from the site independently move to neighboring sites.

Criticality of the SSM is defined with respect to whether or not the system remains active.  We say the stochastic sandpile \textit{locally fixates} if for each $x$ the function $t\mapsto \eta_t(x)$ is eventually constant and \textit{stays active} otherwise.  The main result of \cite{RolSid12} about the SSM is that if the initial distribution $\nu$ is given by i.i.d.\! Poisson random variables with parameter $\mu$, then there exists a  critical value $\mu_c\in [1/4,1]$ such that the system locally fixates almost surely if $\mu<\mu_c$ and stays active almost surely in $\mu>\mu_c$, see \cite[Theorem 1]{RolSid12}.  The argument that $\mu_c\leq 1$ is essentially trivial and comes down to showing that for trivial reasons if $\mu>1$ then the site $0$ topples infinitely many times, while the argument that $\mu_c\geq 1/4$ is subtle.  Since \cite{RolSid12} improvements have been made on the lower bound, see \cite{PodderRolla20}, but the problem posed in 
\cite[Section 7]{RolSid12} of finding a non-trivial upper bound has remained open. 
%with simulations showing that $\mu_c \approx 0.9489$ \cite{DickmanAlavaMunozPeltolaVespignaniZapperi01}.  
Our main result is the following theorem, which gives the first non-trivial upper bound on $\mu_c$.

\begin{theorem}\label{thm_main} For any independent starting configuration with multiple particles at some sites a.s., the critical value $\mu_c$ for the SSM is strictly less than $1$.  In fact, $\mu_c\leq 1- e^{-2\times10^5}$. \end{theorem}

%Talk about universality

In this paper we consider an infinite version of the SSM  but there are also finite versions (e.g.\ driven dissipative and fixed energy \cite{MunozDickmanPastor-SatorrasVespignaniZapperi01, FeyMeester15}). Universality would imply that all the critical values and exponents should be the same across these models. Some sandpile models with deterministic toppling rules have been shown to not exhibit universality \cite{FeyLevineWilson10a}, while there is numerical evidence for universality in the SSM \cite{biham}. Our result should transfer to finite versions of this model \cite{BasuGangulyHoffmanRichey19}.

%ARW was introduced as a more mathematically tractable model that also is believed to exhibit universality.

The methods used in this paper are based on those recently developed to study Activated Random Walk (ARW). ARW is a related abelian network that is also believed to exhibit self-organized criticality and universality \cite{HoffmanSidoravicius06, RolSid12}. See \cite{rolla2020activated} for a nice survey on ARW. While there are many similarities between the SSM and activated random walk, since \cite{RolSid12} activated random walk was introduced as it was believed to be a more tractable model to study \cite{HoffmanSidoravicius06, BasuGangulyHoffman18}. For example, the analogue of Theorem \ref{thm_main} was established in \cite{HRR20} and, indeed, has recently been extended to higher dimensions \cite{hu2022active, asselah2022critical}. There is also evidence that ARW exhibits universality \cite{RollaSidoraviciusZindy19}. Our approach in this paper is similar to the one taken in \cite{HRR20}. However the analysis is much more delicate for the SSM than what was required in previous papers on ARW. The added difficulties arise because in ARW particles move or sleep independently, while in the SSM the particle moves are correlated in pairs. This correlation makes the SSM much more difficult to rigorously analyze.

An essential component of this analysis is to use a half-toppling scheme, defined in Section \ref{sec:setup}. Half-topplings allow us to follow similar (albeit significantly more complicated) approaches to those that have been used to study ARW. Half-toppling schemes have also been used in the previous lower bound on the critical density for the SSM \cite{RolSid12} and the recent  improvement that proves $\mu_c \geq 1/2$ \cite{PodderRolla20}.  In the next section we introduce some notation, detail the half-toppling scheme and then outline the rest of the paper.

%%%%%%%%%%%%%%%%%%%%%%%%%%%%%%%%%%
\section{Setup and outline} \label{sec:setup}
%%%%%%%%%%%%%%%%%%%%%%%%%%%%%%%%%%

\subsection{Sandpile dynamics} 
\label{sec:sandpile_dynamics}

We start by defining the `site-wise' representation for SSM. We identify configurations of particles as functions $\eta: \Z \to \N$, i.e. $\eta(x)$ counts the number of particles at site $x$. To run the dynamics on a finite interval $I \subset \Z$, every site $x \in I$ is assigned an infinite sequence of iid \textbf{instructions}, $(\xi^x_j: j \in \N)$ each taking one of the two possible values $\xi_{-,x}$ or $\xi_{+,x}$ with equal probability, where $\xi_{-,x}$ and $\xi_{+,x}$ are operators on the space of particle configurations that act via
$$\xi_{-,x}(\eta)(y) = \begin{cases} \eta(y), & y \notin \{x, x-1\} \\ \eta(y) \pm 1, & y = x - \frac{1}{2} \mp \frac{1}{2} \end{cases} $$
and 
$$\xi_{+,x}(\eta)(y) = \begin{cases} \eta(y), & y \notin \{x, x+1\} \\ \eta(y) \pm 1, & y = x + \frac{1}{2} \pm \frac{1}{2}. \end{cases} $$

In the usual \textit{discrete-time} setup, the system evolves one step by choosing a site $x \in I$ with $\eta(x) \geq 2$, and applying the \emph{two} stack instructions $\xi^x_j, \xi^x_{j+1}$ where $j$ is such that all instructions $\xi^x_{j'}$ for $j' < j$ have already been used, but $\xi^x_j$ has not. In this version, particles always topple in pairs. For an initial particle configuration $\eta$ and a sequence of sites $\alpha = (\alpha_1, \alpha_2, \ldots, \alpha_k), \alpha_i \in I$ to be toppled, let 
$$\xi_\alpha =  \xi_{j_{2k}}^{\alpha_k} \circ \xi_{j_{2k-1}}^{\alpha_k} \circ \cdots \circ \xi_{j_4}^{\alpha_2} \circ \xi_{j_3}^{\alpha_2} \circ \xi_{j_2}^{\alpha_1} \circ \xi_{j_1}^{\alpha_1}$$
be the operator obtained by performing all topplings of $\alpha$ in order.  The \textbf{odometer} of the pair $(\eta, \alpha)$ is the function which records the total number of times each site was toppled, i.e.
$$m_I(\eta, \alpha, x) := \#\{j: \alpha_j = x\}. $$
Such a sequence $\alpha$ is called \textbf{legal} for $\eta$ if $\xi_{(\alpha_1, \ldots,\alpha_{\ell-1})}\eta(\alpha_\ell) \geq 2$ for every $\ell\in\{1,\ldots,k\}.$

The odometer function $m_I$ is universal in the sense that if $\alpha, \alpha'$ are any two legal toppling sequences for a configuration $\eta$ on interval $I$ such that $\xi_{\alpha} \eta$ and $\xi_{\alpha'} \eta$ have no sites with at least two particles, then $\alpha$ is a permutation of $\alpha'$ and thus $m_I(\eta, \alpha, \cdot) = m_I(\eta, \alpha', \cdot)$. For our argument, it is only necessary that the half-toppling version of the dynamics gives a lower bound on this odometer; see Lemma \ref{lem:halftop_abelian}. 

We need the following result from \cite{RolSid12} which relates the site-wise representation to the continuous-time process $(\eta_t)_{t\geq0}$ in Section \ref{sec:intro}. Define $$m(\eta, x) := \sup_{I, \alpha} m_I(\eta, \alpha, x)$$
where the supremum is over all finite intervals $I$ and all legal toppling sequences $\alpha$ for the configuration $\eta$ on $I$.

\begin{lemma}[{\cite[Lemma 4]{RolSid12}}]\label{0/1law}
	Suppose the initial state $\eta_0$ is a translation-invariant, ergodic distribution on $\Z$ with finite density $\E(\eta_0(0))$, then
$$\p\left((\eta_t)_{t\geq0} \text{ stays active}\right) = \p(m(\eta_0, 0) = \infty) \in \{0,1\}.$$
\end{lemma}

The above lemma says that in order to establish the system staying active a.s., it suffices to show that some site is toppled infinitely many times with positive probability.

\subsection{Half-topplings}

To allow more freedom in our toppling procedure (and retain some independence between particles), we work with an equivalent version of the dynamics. Instead of toppling two particles at a time, we topple single particles, i.e. perform \ \textbf{half-topplings}. In the usual dynamics, the number of particles toppled at a given site is always even, so there are restrictions on which half-topplings are allowed. Namely, we must keep track of how many times each site has been half-toppled, and only allow another half-toppling if that number is odd, or if that site has at least two particles. 

Thus, in the half-toppling scheme, a configuration of particles consists of two functions, $\eta: \mathbb{Z} \to \mathbb{N}$, the number of particles per site, and a parity sequence $\omega: \mathbb{Z} \to \{0,1\}$.
%where $\omega(x) = 0$ if $x$ has been half-toppled an odd number of times, or $\omega(x) = 1$ if $x$ has been half-toppled an even number of times; or, more concisely, $\omega(x)$ is the odometer at $x$ mod 2. Formally,
A site $x$ is \textit{legal} for a pair of configurations $(\eta, \omega)$ if either $\eta(x) \geq 2$ or $\eta(x) = \omega(x) = 1$. At each step of the half-toppling scheme, a legal site $x$ is chosen for the current pair of configurations, and the first unused instruction $\xi_j^x$ acts upon the pair by sending a single particle at $x$ to a uniform random neighbor and adding 1 (mod 2) to $\omega(x)$. A half-toppling sequence of sites $\alpha=(\alpha_1, \alpha_2, \ldots, \alpha_k)$ is \textbf{legal} for $(\eta,\omega)$ if the chosen site is legal for every step, that is, $\alpha_\ell$ is legal for
$$\xi_{j_{\ell-1}}^{\alpha_{\ell-1}} \circ \cdots \circ \xi_{j_2}^{\alpha_2} \circ \xi_{j_1}^{\alpha_1}(\eta,\omega)$$
%$$\xi_{(\alpha_1,\ldots,\alpha_{\ell-1})}(\eta,\omega)$$
for every $\ell\in\{1,\ldots,k\}.$

One key fact about this version of the dynamics is that it provides a lower bound on the total odometer. Recall the odometer function $m_I$ for the classical dynamics, and for any sequence $\alpha$ of half-topplings, let $m_I(\eta, \omega, \alpha, x)$ denote the corresponding half-toppling odometer function, after performing all half-topplings in $\alpha$ started from the configuration $(\eta, \omega)$. We have:

\begin{lemma} \label{lem:halftop_abelian} (Abelian lemma for half-topplings) Fix a particle configuration $\eta$ on a finite interval $I$. Let $\alpha$ be any legal half-toppling sequence for $(\eta, \vec{0})$, and let $\overline{\alpha}$ be any legal toppling sequence for $\eta$ (in the sense of Section \ref{sec:sandpile_dynamics}) such that $\xi_{\overline{\alpha}} \eta$ has no site with at least two particles. Then for any $x\in I$,
$$ m_I(\eta, \vec{0}, \alpha, x) \leq 2  m_I(\eta, \overline{\alpha}, x).$$ \end{lemma}

We omit the proof, which follows identically to that of \cite[Lemma 6]{RolSid12}. 

\subsection{Outline}

Our strategy is to alternate between two types of legal half-toppling sequences, taking place on a sequence of nested intervals in $\mathbb{Z}$.

The first of these is called the `carpet/hole toppling procedure' and is essentially a renormalization scheme of the SSM, detailed in Section \ref{sec:toppling}. The procedure starts from and remains in a very particular set of configurations, and runs until it partially stabilizes the interval -- when there are no more legal topplings among a small class of marked (`free') particles. We will show that each iteration is well behaved -- that many particles reach the boundary -- with exponentially high probability. 
The main technical hurdle in analyzing the carpet/hole procedure is the single block estimate Lemma \ref{lem:sbe}, which we deal with in Sections \ref{sec:carpets}, \ref{sec:ugly} and \ref{sec:renewal}. In Section \ref{sec:flows_sbe} we use the single block estimate to prove a version of the above statement, via an energy-entropy calculation similar to \cite{HRR20}.

The carpet/hole procedure needs to start with relatively nice configurations. Thus when partial stabilization is achieved on an interval, before restarting the carpet/hole procedure on the larger interval, it is necessary to perform the other `bootstrap' toppling procedure to restore the configuration.
%This `bootstrap' phase, detailed in section \ref{sec:boot}, involves comparatively simple topplings, namely a few IDLA steps.
In Section \ref{sec:boot} we analyze the alternation of both procedures and prove Theorem \ref{thm_main}.

In the remaining sections, we focus on the carpet/hole dynamics inside a single block. In Section \ref{sec:carpets}, to exploit the independence among random walks, we introduce an auxiliary process that only reveals partial information about the parity configuration. In Section \ref{sec:ugly}, we turn to the intricate correlation between particles in our toppling procedure, by carefully controlling the typical number of zeros in the auxiliary process. Finally, in Section \ref{sec:renewal} we combine these results and prove Lemma \ref{lem:sbe}. 

%\begin{lemma}\label{lem:halftop_abelian} (Abelian lemma for half-toppling) Let $\alpha$ be any sequence of legal half-topplings for a finite interval $V$ that stabilizes the initial configuration $(\eta, \omega = \vec{0})$. Then $m_\alpha(\eta) = 2 m_V(\eta)$. \end{lemma} 

%To keep track of which sites are legal, configurations are represeneted as binary sequences. Use [RS '11] corollary 7 for the abelian lemma for half topplings. Upshot is: if we use legal half topplings to stabilize a configuration, the odometer counts match what we would have with the usual dynamics. Legal means that the site is unstable, either because there are at least two particles there \emph{or} the site has been half-toppled an odd number of times. I think it's worth doing a better job of explaining this here, it's not explained very clearly in the second paragraph of [RS '11] section 6. 
%
%Here is also a good place to define the configurations we're using: bit strings on $\mathbb{Z}$, which represent stable/unstable sites. 

%%%%%%%%%%%%%%%%%%%%%%%%%%%%%%%%%%
\section{Carpet/hole toppling procedure}
\label{sec:toppling}
%%%%%%%%%%%%%%%%%%%%%%%%%%%%%%%%%%

The carpet/hole toppling procedure divides the space into \textit{blocks} as well as long \textit{transit regions} between blocks. Each particle is designated as one of two types: there are \emph{carpet} particles which do not move and occupy all sites except for one location per block called the \textit{hole},
%(though their parity values $\omega$ change over time, and are sometimes turned into free particles)
and \textit{free} particles which are toppled repeatedly and legally (thanks to the carpet particles).
%(and are sometimes turned into carpet particles).
At some point, a free particle may turn into a carpet particle and vice versa, but the total number of free particles will be preserved throughout. In view of Lemma \ref{lem:halftop_abelian}, we will make sure the half-topplings performed during the procedure are always legal.
%The number of free particles will be preserved throughout -- whenever a free particle turns into a carpet particle, some carpet particle will be turned into a free particle -- and (half) topplings performed during this procedure will always be legal.
After the definition, we state in Proposition \ref{prop:kick} the main fact that will be used regarding the carpet/hole toppling procedure.

\subsection{Valid configurations}
\label{sec:valid}

Fix integers $a > 0$ and $K = a^4$. The procedure described in this section is well-defined for any positive integer $a$, although later in Proposition \ref{prop:kick} and Lemma \ref{lem:sbe} we will need to choose a large enough $a$.
%(The procedure is well defined for any positive integer $a$, and our main results hold for all integers $a > a_0$ for some integer $a_0$ to be specified later.)
The toppling procedure needs to start from relatively nice particle configurations on a finite interval, which we call valid configurations.

\begin{definition} \label{def:valid} For any $n \in \N$, a pair of particle configuration and parity sequence $(\eta, \omega)$ on $D_n := (-K+a, nK)$ is \textbf{valid} if $$\eta = \sum_{i = -K+a+1}^{nK-1} \delta_i - \sum_{u \in U} \delta_{u} + \sum_{i=0}^{n-1} (b_{i,0}\delta_{iK} + b_{i,1}\delta_{iK+a})$$ for some $b_{i,0}, b_{i,1} \geq 0$ and some $U \subset \cup_{i=0}^{n-1} [iK, iK+a]$ satisfying 
\begin{equation} |U \cap [iK, iK+a]| \leq 1,\quad \text{for } i = 0, 1, \ldots, n-1. \end{equation}
\end{definition}

%For any $n \in \N$, a \textbf{carpet configuration} on $D_n = (-K+a, nK)$ has initial configuration of particles of the form

%\begin{equation} \eta^{carpet}(U,n) = \sum_{i = -K+a}^{nK} \delta_i - \sum_{u \in U} \delta_{u}, \end{equation}

%for some $U \subset \cup_{i=0}^{n-1} [iK, iK+a]$ satisfying 

%\begin{equation} \label{valid_h} |U \cap [iK, iK+a]| \leq 1, i = 0, 1, \ldots, n-1, \end{equation}

%i.e. there are single particles everywhere except for some possible empty spots, with at most one hole per block.

The subinterval $[iK, iK+a]$ is called the $i$th \textbf{block} for $i = 0, 1, \ldots, n-1$, and the complement of the blocks are called \textbf{transit regions}. In other words, a valid configuration has single particles everywhere, except for at most one empty spot per block and possibly some additional particles at the boundary of every block.
%A valid configuration has particle density at least $1-\frac{|U|}{nK}$ (this is our $\zeta_c$).

Fix a valid configuration $(\eta, \omega)$. In order to perform the carpet/hole toppling procedure starting from $(\eta, \omega)$, we need to identify the special site \textit{hole} in each block based on $(\eta, \omega)$, as well as classify the particles of $\eta$ into several categories. Note that the definitions given in the current Section \ref{sec:valid} only apply to the starting configuration $(\eta, \omega)$. In Section \ref{sec:carpet_rules}, we will provide rules regarding how the holes move and particles switch types during the procedure.

Every block has exactly one \textbf{hole}, uniquely determined by the pair $(\eta, \omega)$. For each $i = 0, \ldots, n-1$, if block $i$ has a position with no particle, then the empty site is the hole of block $i$. If block $i$ has a particle in every position, the leftmost site $x$ in that block with $\omega(x) = 1$ is declared the hole. If there is no such $x$, then declare $iK+a$ to be the position of the hole.

With the definition of the holes, we allocate the particles in $\eta$ into several types. By definition, each site in $D_n$ (containing both blocks and transit regions) that does not have a hole will have at least one particle. We declare one particle at every such site to be a \textbf{carpet} particle. (If there are multiple particles at a site, we simply pick an arbitrary one and declare it a carpet particle.) All the remaining particles in $D_n$ are \textbf{free} particles.

In the configuration $\eta$, free particles are either in a hole or at an endpoint of a block. Free particles are further divided into two groups. For each $i=0,\ldots,n-1$, if block $i$ has $\eta(x) \geq 1$ and $\omega(x) = 0$ for all $x \in [iK, iK+a]$, then the hole was defined to be at $iK+a$ and we declare an arbitrary (free) particle at $iK+a$ to be a \textbf{frozen} free particle. All the other free particles are \textbf{thawed}. There is always one special thawed particle, the \textbf{hot} particle, that actually gets toppled. We will explain the rule of determining the hot particle in Section \ref{sec:carpet_rules}.

\subsection{Carpet/hole dynamics} \label{sec:carpet_rules} 

We now formally define the carpet/hole toppling procedure inside a given interval $D_n = (-K+a, nK)$. Fix a valid configuration $(\eta, \omega)$. Suppose we've picked the starting locations of the holes and labelled the particles as `carpet', `free', `frozen' and `thawed' according to Section \ref{sec:valid}. The toppling procedure works as follows:
%Recall that the $i$th `block' is the subinterval $[iK, iK + a]$ for $i = 0, 1, 2, \ldots, n-1$. We have two types of particles: free particles and carpet particles. Each block has a single site designated the \textbf{hole} which can be empty or contain a free particle. There is always one special free particle, the \textbf{hot} particle, which is the leftmost unfrozen free particle. Some free particles can become \textbf{frozen} during the procedure. 

%Fix a starting configuration $\eta^{carpet}(U,n)$ for some $U$ satisfying \ref{valid_h} and any parity configuration $\omega$. For each $i = 0, \ldots, n-1$, if block $i$ has a particle in every position, the particle at the leftmost legal site $x$ in that block with $\omega(x) = 1$ is declared an unfrozen free particle, and site $x$ is declared the hole of block $i$. If there is no such $x$, then declare the particle at $iK+a$ to be a frozen free particle, and also the position of the hole. If block $i$ has a position with no particle, then it contains no free particle, and the empty site is the hole. 

%The procedure works left to right within $D_n$, always choosing the leftmost block containing an unfrozen free particle, and toppling one of the free particles in that block according to the following loop: 

\begin{enumerate}[(L1)] \item \label{loop:hot} We follow a \textit{leftmost priority policy} for choosing the hot particle, which implies Lemma \ref{lem:coarse_freeze} below. Find the leftmost block $i$ that contains a thawed particle. If no such block exists, the procedure ends and we say the procedure reaches \textit{partial stabilization}. Among the thawed particles in block $i$, we choose the one inside the hole to be the hot particle if there is one. Otherwise, the thawed particles are at the boundary $iK$ or $iK+a$, and we declare an arbitrary thawed particle the hot one. This particular order of choice is to ensure \ref{pro:freewhere_afterhot} below. Once we choose the hot particle, we proceed with \ref{loop:main} or \ref{loop:unfreeze} depending on whether block $i$ contains a frozen particle.
%Denote by $x$ the position of the hole in block $i$. 

\item \label{loop:main} Suppose block $i$ does not contain a frozen particle. If the hot particle is at the hole but the hole is not the leftmost position $y$ with parity value $\omega(y)=1$, then go directly to \ref{loop:exc} or \ref{loop:froze}. Otherwise, repeatedly topple the hot particle until it returns to the hole of block $i$, or hits either $(i-1)K+a$ or $(i+1)K$, which could be a boundary point of a neighboring block or a boundary point of the interval $D_n$. There are three cases:

\begin{enumerate}[(L2a)] 
\item \label{loop:exc} Suppose the hot particle returned to the hole of block $i$ and there exists some position in the block $i$ with odd parity. Find the leftmost position $y$ in block $i$ with parity value $\omega(y)=1$. Declare the hot particle at the hole a carpet particle, move the hole in block $i$ to position $y$, and declare the carpet particle at $y$ a thawed and hot particle. Return to \ref{loop:main}. 

\item\label{loop:froze} Suppose the hot particle returned to the hole of block $i$, but there is no site in block $i$ with odd parity. Turn the hot particle at the hole into a carpet particle, move the hole to position $iK+a$, and declare the carpet particle at $iK+a$ a frozen free particle. Return to \ref{loop:hot}.

\item\label{loop:emit} If the hot particle reached $(i-1)K+a$ or $(i+1)K$, declare it no longer hot but still free and thawed. Return to \ref{loop:hot}.
\end{enumerate}

\item \label{loop:unfreeze} If block $i$ does contain a frozen particle, then every site in block $i$ has a particle which is not the hot particle. Repeatedly topple the hot particle until it reaches $(i-1)K+a$ or $(i+1)K$. Then declare it no longer hot but still free and thawed. If block $i$ now has the leftmost site $y$ with parity value $\omega(y) = 1$, then turn the frozen free particle at $iK+a$ into a carpet particle, move the hole to $y$, and declare the carpet particle at $x$ free and thawed. Return to \ref{loop:hot}. 

\end{enumerate} 

The block size $a$ is large, so the typical occurrence is \ref{loop:exc}: a hot particle makes an excursion inside its block. Our aim is to show that step \ref{loop:froze} occurs much \textit{less} often than \ref{loop:emit}.

More broadly, what we will show is that the carpet/hole toppling procedure, as a renormalization scheme of the Stochastic Sandpile Model, `converges' to the Activated Random Walk model (ARW) with a small sleep rate when the block size goes to infinity. Each block and every free particle in our procedure correspond to a single site and a normal particle in the ARW model respectively. Step \ref{loop:emit} mimics an ARW particle moving to a neighboring site, whereas the rare \ref{loop:froze} step represents an ARW particle becoming sleepy with a small sleep rate. If \ref{loop:froze} does happen and a free particle becomes frozen, it can become thawed through \ref{loop:unfreeze}, just as a sleepy ARW particle gets reactivated by the presence of another particle.

Once we make the `convergence' rigorous in certain sense, the analysis of a 1D ARW-like system with a small sleep rate yields to a classical energy-entropy argument \cite{hu2022active}. The overall outcome of these is Proposition \ref{prop:kick}, stated at the end of Section \ref{sec:toppling}.
%and proved in Sections \ref{sec:flows_sbe}, \ref{sec:carpets}, \ref{sec:ugly} and \ref{sec:renewal}.

%Suppose this is shown, then the  -- when a particle is frozen -- which will imply that \ref{loop:emit} -- when a particle is emitted to a neighbor block -- occur many times, and thus that many particles reach the boundary of $D_n$. 

%Note that the (half) topplings away from the hole are always legal because there will be two particles there (the `free' particle and one carpet particle); half topplings from the hole are only legal when the parity bit is 1 there. 

%A block can become frozen if the hole makes it to the edge of the block. Once a particle from a neighboring block is emitted and its excursion covers the whole frozen block, we declare the block unfrozen, and stop toppling that particle. 

The following properties hold trivially for the initial configuration $(\eta, \omega)$ and are preserved by the half-toppling sequence, which can be checked by induction.
\begin{enumerate}[(P1)]

\item
%\label{prop:onehole}
Each block $ i $ has exactly one hole which is located at some site $ x \in [iK,iK+a]$.

\item
%\label{prop:holewhere}
Whenever the hot particle is inside the hole of a block without a frozen particle, we always make sure this hole is located at the leftmost site $y$ in that block with parity value $\omega(y)=1$. If there is no site in block $i$ with odd parity, then we make sure the hole is at $iK+a$ and contains a frozen particle.

%\item
%\label{prop:holewhere_frozen}
%If there is no site in block $i$ with odd parity, then the hole is at $iK+a$ and contains a frozen particle.

\item
%\label{prop:carpetwhere}
There is exactly one carpet particle at every site in $D_n$ except for the holes.

\item
%\label{prop:freewhere}
All free particles except the hot particle are in a hole, or at a site $iK$ or $iK+a$ for some $i$.

\item
%\label{prop:frozen_block}
We call a block which contains a frozen particle a \textbf{frozen block}. In a frozen block, both the frozen particle and the hole are at $iK+a$. Thus every site in this block has a particle that is not hot.
\end{enumerate}
We also collect some facts that will be useful throughout the paper.

\begin{enumerate}[(F1)]
\item 
\label{pro:legal}
All half-topplings performed during this procedure are legal. 

%\item 
%\label{prop:leftzeros}
%Except for the hole in the block containing the hot particle, at all times during the procedure, each hole is located at the leftmost legal -- i.e. site with parity value $\omega = 1$ -- site of its block. If there is no such site, the block contains a frozen particle at site $iK+a$ and has uniform parity value $0$.
%\item
%\label{prop:hascarpet}
%Every site except the holes or sites $iK$ or $iK+a$ contains exactly one carpet particle.
%\item
%\label{prop:freewhere}
%All free particles except the hot particle are in a hole, or at a site $ iK $ or $ iK + a $ for some $i$.  

\item
\label{pro:freewhere_afterhot}
After the first free particle has been designated hot in block $i$, all free particles in block $i$ except the hot particle (if there is one) are at $iK$ or $iK+a$. 

\item
\label{pro:frozen_run}
A free particle designated hot in a block with a frozen particle reaches a neighboring block before being declared not hot.

\item
\label{pro:conserved}
The number of free particles is conserved during the procedure. 

\item 
\label{pro:end}
When the procedure ends, all free particles are either frozen, with at most one frozen particle per block, or at the boundary points of $D_n$. Also, there is at most one particle (either carpet or frozen free particle) at every site inside $D_n$.

\end{enumerate}

\subsection{Bounds on frozen particles}

We prove the following regarding the carpet/hole toppling procedure. Roughly speaking, it says that the dynamics is likely to sustain `enough activity' so that at the end of the procedure, there are few frozen particles left inside every subinterval.

Let $(\eta, \omega)$ be a valid configuration on $D_n=(-K+a, nK)$. Run carpet/hole dynamics with initial condition $(\eta, \omega)$ until \textit{partial stabilization} when there are no more thawed particles inside $D_n$. Denote its law by $\p_{(\eta, \omega)}$.
For $0 \leq n_0 \leq n_1 \leq n-1$, define $\frozen(n_0, n_1)$ to be the number of frozen free particles remaining in the blocks $n_0, n_0+1, \dots, n_1$ at the end of the carpet/hole dynamics on $D_n$.
Also let $m_n(\eta, \omega, x)$ be the (half-toppling) odometer function at site $x \in D_n$ resulting from the carpet/hole toppling procedure on $D_n$, and define
%and let $\p_{(\eta, \omega)}$ denote the joint law of the random initial configuration $(\eta, \omega)$ and of the carpet/hole dynamics started from initial configuration $(\eta, \omega)$.
\begin{equation} H_i := \{m_n(\eta, \omega, iK) \geq \beta n\},\, H(n_0, n_1) := H_{n_0-1}^c \bigcap \left(\bigcap_{i = n_0}^{n_1} H_i \right) \end{equation}
for some constant $\beta>0$. 
%, where the convention is that $H_{-1} = \emptyset$
The purpose of the event $H(n_0, n_1)$ is to ensure each block starting from a nice configuration and to facilitate a bootstrap argument later in the proof of Lemma \ref{jayhawk}. Write $H:=H(0,n-1)$ and $\frozen := \frozen(0,n-1)$.

\begin{proposition} \label{prop:kick} For $\delta= \beta = 4\times10^{-4}$ and any $L^\ast>0$, there exist $c_1>0$ and $K = e^{1.8\times10^5}$ such that the following hold for large enough $n' = n_1-n_0+1$.
If $(\eta, \omega)$ is valid and the total number of particles $\sum_{x \in D_n} \eta(x) \leq L^\ast |D_n|$,
%Run carpet/hole dynamics with initial condition $(\eta, \omega)$ until the partial stabilization when there are no more legal topplings of free particles inside $D_n$, and let $\frozen$ denote the number of frozen free particles remaining in the interval.
%Consider the events
%\begin{equation} E^L = \text{ some free particle reaches } -K+a \text{ during the toppling procedure} , \end{equation}
%and $E^R$ similarly, for the right endpoint $nK$ of $D_n$. Let $E = E^R \cap E^L$. Also let 
%Let \begin{equation} H_i = \{m_n(\eta, \omega, iK) \geq \beta n\}, H = \bigcap_{i = 0}^{n-1} H_i. \end{equation}
then we have
\begin{equation} \p_{(\eta, \omega)}(\{\frozen(n_0, n_1) > \delta n'\} \cap H(n_0, n_1)) \leq \exp(-c_1 n'). \end{equation}
In particular,
\begin{equation} \p_{(\eta, \omega)}(\{\frozen > \delta n\} \cap H) \leq \exp(-c_1 n). \end{equation}
\end{proposition}

%\begin{proposition} \label{prop:kick} There exist $K, c_1, c_2 , \beta > 0$ such that the following hold for all $n$. Let $\epsilon = 1/200$ and let $(\eta, \omega)$ denote a configuration on $D_n = (-K+a, nK)$. Run carpet/hole dynamics with initial condition $(\eta, \omega)$ until there are no legal topplings of free particles inside $D_n$, and let $\frozen$ denote the number of frozen free particles remaining in the interval. Consider the events
%\begin{equation} E^L = \text{ some free particle reaches } -K+a \text{ during the toppling procedure} , \end{equation}
%and $E^R$ similarly, for the right endpoint $nK$ of $D_n$. Let $E = E^R \cap E^L$. Also let 
%\begin{equation} H_i = \{m_n(\eta, \omega, iK) \geq \beta n\}, H = \bigcap_{i = 0}^{n-1} H_i. \end{equation}

%We have:

%\begin{itemize} \item[A1.] If $\eta$ is valid, then 
%\begin{equation} \p_{(\eta, \omega)}((\{\frozen > \epsilon n\} \cup E^c )\cap H) \leq \exp(-c_1 n). \end{equation} 

%\item[A2.] If $(\eta, \omega)$ is fully valid, then
%\begin{equation} \p_{(\eta, \omega)}(\{\frozen > \epsilon n\} \cup E^c) \leq \exp(-c_2 n^{1/4}). \end{equation}

%\end{itemize}

%\end{proposition}

Proposition \ref{prop:kick} will be proved in the next section.
%One part of the proof, namely using Lemma \ref{lem:sbe} to bound $\frozen$, is essentially the same as the corresponding proof in \cite{HRR20}.
We will rely on Proposition \ref{prop:kick} to do the analysis in Section \ref{sec:boot}.

%Note that an emitted particle may, with small probability, fail to unfreeze the block, i.e. complete its excursion and leave all the sites in the block with parity $0$. Here it will help to have large transit regions, so the emitted particle will have high probability of fully covering the frozen block. Even still, we will have to use one of our parity lemmas for simple random walk to ensure we return to a good configuration with high probability.

%%%%%%%%%%%%%%%%%%%%%%%%%%%%%%%%%%
\section{Filtrations and coarse-grained particle flows}
\label{sec:flows_sbe}
%%%%%%%%%%%%%%%%%%%%%%%%%%%%%%%%%%

Fix any initial configuration $(\eta, \omega)$ on $D_n$. Consider one iteration of the carpet/hole dynamics on an interval $D_n = (-K+a, nK)$ with $n$ blocks. Recall our aim is to prove Proposition \ref{prop:kick}.
%show that at the end of the procedure, there are few frozen particles.
The `single block estimate' Lemma \ref{lem:sbe} roughly says that %for a given block, conditionally on all instructions in blocks to its left,
the probabilities of a block being frozen at certain stopping times are small. To combine these estimates for different blocks, it is necessary to introduce some independence between the blocks, and enumerate the possible trajectories of the system. 
%This is achieved by: 1) indexing instructions a free particles uses according to the block that free particle is in, and 2) enumerating all possible trajectories by counting the flow of particles between blocks.
Section \ref{sec:flows_sbe} is devoted to these goals, and then giving a proof of Proposition \ref{prop:kick} by using Lemma \ref{lem:sbe}. Though more technical, the argument in this section follows a similar approach as the one in \cite{HRR20}. 

%Note that in this section, and for the purpose of proving Lemma \ref{lem:sbe}, we are working with a fixed initial configuration $(\eta, \omega)$.

\subsection{Decorated stacks and filtration} 

Whenever a free particle is designated as the hot particle in step \ref{loop:hot}, we designate the block $i$ it is in as the \textbf{hot block}. For each $y \in D_n$ in a transit region, we split the stack instructions for $y$ into two independent stacks of independent instructions, $\xi^y  = (\xi^{y,L}, \xi^{y,R})$. The hot particle toppled at site $y$ uses instructions from the $L$ stack if the nearest block to the left of $y$ is the hot block, and from the $R$ stack if the nearest block to the right of $y$ is the hot block. Sites $y$ in a block have just a single instruction stack, $\xi^y$. We work with the filtration $\mathcal{F}_i$ of all the stack instructions decorated by blocks $j \leq i$, i.e. 

\begin{equation} \mathcal{F}_i = \sigma[\{\xi^y: y \in (-K+a, iK+a]\} \cup \{\xi^{y, L}: y \in [iK+a, (i+1)K)\}]. \end{equation}

\subsection{Coarse-grained particle flows and mass balance equations}

To index possible `trajectories' of the carpet/hole dynamics, we count the flow of particles between blocks. For each $i \in \{0, \ldots, n-1\}$ and $s \geq 0$ let 
\begin{equation} \eta^+(s, i) = \eta + s \delta_{iK + a} \end{equation} 
be the configuration $\eta$ with $s$ additional particles at the right edge of block $i$. We define the following `coarse-grained counters': 

\begin{definition} For each $j \in \{0, \ldots, n-1\}$, run carpet/hole dynamics started from configuration $(\eta^+(s, j), \omega)$ until there are no legal topplings of free particles in $(-K+a, jK)$, the first $j$ blocks. Define

\begin{itemize}
\item for any $i \leq j$, the total number of times $L_i^j(s) $ during this toppling procedure that a free particle goes left using an instruction from the stack $\xi^{(i-1)K+a+1, R}$;

\item for any $i \leq j$, the number of frozen particles $F_i^j(s)$ in block $i$ at the end of the toppling procedure (either 0 or 1 by \ref{pro:end});

\item $F_i := F_i^{n-1}(0)$ and $L_i := L_i^{n-1}(0)$. Set $L_n:=0$.

\end{itemize} \end{definition} 

Note that $F_i^j(s), L_i^j(s) \in \mathcal{F}_j$. Our leftmost priority policy for choosing the hot particle guarantees the next lemma.

\begin{lemma} \label{lem:coarse_freeze} We have, for each $i \in \{0, \ldots, n-1\}$,

\begin{equation}  F_i = F_i^i(L_{i+1})
\quad\text{and}\quad L_i = L_i^i(L_{i+1}) . \end{equation} \end{lemma}

See the corresponding lemma in \cite{HRR20} for a proof. This allows us to study the number of frozen free particles in block $i$ at the end of the procedure by running the carpet/hole dynamics up to block $i$, plus some random number of extra particles input from the right.

To avoid dependence, we make a uniform argument over all possible values of the input from the right.
% in the forthcoming `single block estimate', Lemma \ref{lem:sbe}. 
%Roughly speaking, the blocks only communicate through these flows, and the flow must be consistent between blocks.
A vector of integers $\boldsymbol{s} = (s_0, s_1, \cdots, s_n=0)$ is said to satisfy the \emph{mass balance equations} if 
$$L_i^i(s_{i+1}) = s_i  \text{ for } i = 0, \ldots, n-1.$$
Note that the \emph{random} vector $(L_0, L_1, \ldots, L_{n} = 0)$ satisfies the mass balance equations by definition.
%\red{Fix the indexing here, it's a little fucky}
In what follows, we will sum over the \textit{random} set of possible vectors $\boldsymbol{s}$ satisfying these random mass balance equations.
%(Aside: this exact form of the `mass balance equations' isn't crucial for the argument, there are many ways to express the same `conservation of flow' fact, which could equivalently be used to index the space of trajectories.) 

\subsection{Single block estimate}

The following Lemma is an upper bound on the number of frozen particles $F_i$ in a fixed block. Sections \ref{sec:carpets}, \ref{sec:ugly} and \ref{sec:renewal} are devoted to its proof.
%We need two versions corresponding to assumptions $A1$ and $A2$ in Proposition \ref{prop:kick}.

%This part can be carbon copied from the ARW paper, through to the single block estimate and the corresponding proof, which will be our goal. Even the statement of our single block estimate can probably be identical. Recall that $S_i(m)$ is the indicator random variable that block $i$ is frozen after stabilizing blocks $1, 2, \ldots i$ according to the carpet/hole toppling procedure, started from the carpet configuration plus $m$ particles at the leftmost point of block $i$; and $L_i(m)$ is the number of particles emitted to the left from block $i$ under the same conditions; and $\mathcal{F}_{i}$ is the $\sigma$-field consisting of all topplings up to stabilizing the first $i$ blocks. 

\begin{lemma}\label{lem:sbe} For $\delta=4\times10^{-4}$, there exist constant $K = e^{1.8\times10^5}$ and sufficiently large $c, c'$ such that $\log c' <\delta c$, and the following holds for any $i \leq n$. If the initial configuration $(\eta, \omega)$ on $D_n = (-K+a, nK)$ is valid, then
\begin{equation} \sup_{l \geq 0} \sum_{s \geq 2} \E_{(\eta, \omega)}[e^{c F_i^i(s)} 1\{L_i^i(s) = l\} | \mathcal{F}_{i-1}] < c'.
\end{equation}
Additionally, for any $k\ge0$,
\begin{equation}
\sup_{l\geq 0}\p_{(\eta, \omega)}\left( \sum_{s\geq0}1\{L_i^i(s)=l\} > k\bigg| \mathcal F_{i-1}\right)	\leq \theta^k
\end{equation}
for some $\theta\in(0,1)$ independent of $k$.
\end{lemma}

%\begin{itemize} \item[A1.] If $(\eta, \omega)$ is valid, then
%\begin{equation} \sup_{l \geq 0} \sum_{s \geq 2} \mathbb{E}_{(\eta, \omega)}[e^{c F_i(s)} 1\{L_i^i(s) = l\} | \mathcal{F}_{i-1}] < c'. \end{equation}

%\item[A2.] If $(\eta, \omega)$ is fully valid, then 
%\begin{equation} \sup_{l \geq 0} \sum_{s \geq 0} \mathbb{E}_{(\eta, \omega)}[e^{c F_i(s)} 1\{L_i^i(s) = l\} | \mathcal{F}_{i-1}] < c'. \end{equation} \end{itemize}
Note the lower limit of the summation over $s$ in the first statement. Since we only assume the starting configuration to be valid, we need at least two input particles to `fix' the configuration -- see the proof of Lemma \ref{lem:sbe} in Section \ref{sec:renewal}.

In this section we use the single block estimate to prove Proposition \ref{prop:kick}.
%Though more technical, the proof here is in the same spirit as the corresponding argument from [HRR '19]. 
\\
%The difference between the two versions is that if the starting configuration is only assumed to be valid, need at least two input particles to `fix' the configuration -- see the proof of Lemma \ref{lem:sbe} in Section \ref{sec:renewal}. First we use the single block estimate to prove Proposition \ref{prop:kick}.

\begin{pfofprop}{\ref{prop:kick}}
We start by treating the special case where $n_0=0$ and $n_1 = n-1$.
Let $\eta$ be any valid carpet, and recall the events $H_i$ that the odometer of site $iK$ is at least $\beta n$ and $H$ that all $H_i$'s occur (note that $H_{-1}=\emptyset$). Also let $J_i$ be the event that among the first $\beta n$ many times a hot particle is toppled at site $iK$, at least twice it takes a step left and then reaches site $(i-1)K+a$ before returning to site $iK$. Note that $J_i \in \mathcal{F}_i$ and
 \begin{equation} \label{omakase}
 	\p(J_i^c\mid \mathcal{F}_{i-1}) \leq \beta n \left(1-\frac{1}{2K}\right)^{\beta n-1} \leq \exp(-c_2n)
 \end{equation}
for some $c_2>0$ depending only on $K$ and $\beta$.

To bound $\frozen$, the number of frozen free particles remaining in $D_n$ after the carpet/hole dynamics, we use the fact that $H_i \subset \{L_i\geq2\} \cup J_i^c$. Since $\frozen = \sum_i F_i$,
\begin{align} \E[e^{c\frozen} 1_H] &\leq e^c \cdot \E\left[\prod_{i=0}^{n-2} e^{cF_i}1_{H_{i+1}} \right] \\
&\leq e^c \cdot \E\left[\prod_{i=0}^{n-2}\left(e^{cF_i}1\{L_{i+1}\geq2\} + e^c1_{J_{i+1}^c}1\{L_{i+1} < 2\}\right)	\right].
\end{align}
Recall there are at most $L^\ast |D_n|$ particles inside $D_n$ at the beginning. Let $c_4=L^\ast(K+1).$ Using Lemma \ref{lem:coarse_freeze}, we may rewrite the expectation as
\begin{align}
&\E\left[\sum_{s_0=0}^{c_4n}\sum_{s_1} \cdots \sum_{s_{n-1}} \prod_{i=0}^{n-2} \left(e^{cF_i^i(s_{i+1})}1\{L_{i+1}\geq2\} + e^c1_{J_{i+1}^c}1\{L_{i+1} < 2\}\right) 1\{L_i^i(s_{i+1}) = s_i\} 1\{L_i = s_i\}\right] \\
%&\leq\E\left[\sum_{s_0=0}^n\sum_{s_1} \cdots \sum_{s_n} \prod_{i=0}^{n-1} \left(e^{cF_i^i(s_{i+1})}1\{L_i^i(s_{i+1}) = s_i\}1\{s_{i+1}\geq2\} + e^c 1_{J_{i+1}^c} 1\{s_{i+1}<2\} \right)\right]\\
&\leq \E\left[  \sum_{s_0=0}^{c_4n} \sum_{s_1} \cdots \sum_{s_{n-1}} \pi_{n-1}\right] = \mathcal S(n-1),
\end{align}
where \begin{equation}
 \pi_k := \prod_{i=0}^{k-1} \left(e^{cF_i^i(s_{i+1})}1\{L_i^i(s_{i+1}) = s_i\}1\{s_{i+1}\geq2\} + e^c 1_{J_{i+1}^c} 1\{s_{i+1}<2\} \right) \in \mathcal F_k
 \end{equation} and \begin{equation}
 \mathcal S(k) := \E\left[  \sum_{s_0=0}^{c_4n} \sum_{s_1} \cdots \sum_{s_k} \pi_k\right].
 \end{equation}

We will inductively show that $\mathcal S(k) \leq (c_4n+1)(c'+4e^{2c-c_2\wedge c_3n})^k$ for constants $c_2,c_3,c_4$.
The base case $k=0$ is trivial and the case $k=1$ follows from Lemma \ref{lem:sbe} and the estimate \eqref{omakase}. Suppose that the inequality is true for $k-2$ and $k-1$. Let \begin{equation}
 U_i(s_{i+1},s_i) := e^{cF_i^i(s_{i+1})}1\{L_i^i(s_{i+1}) = s_i\} \in \mathcal F_i \quad\text{and}\quad	 V_i := e^c1_{J_i^c} \in \mathcal F_i.
 \end{equation}
Decomposing the last two sums depending on whether $S_{k-1}\ge2$ and $S_k\ge2$, we get
\begin{align}
\mathcal S(k) =& \,\E\left[\sum_{s_0=0}^{c_4n} \sum_{s_1} \cdots \sum_{s_{k-1}} \pi_{k-1} \cdot \sum_{s_k=0,1} V_k \right]\\
 &+\E\left[\sum_{s_0=0}^ {c_4n}\sum_{s_1} \cdots \sum_{s_{k-2}} \pi_{k-2} \cdot \sum_{s_{k-1}\geq2} U_{k-2} (s_{k-1},s_{k-2})\sum_{s_k\geq2} U_{k-1}(s_k,s_{k-1}) \right]\\
 &+\E\left[\sum_{s_0=0}^{c_4n} \sum_{s_1} \cdots \sum_{s_{k-2}} \pi_{k-2} \cdot \sum_{s_{k-1}=0,1}\sum_{s_k\geq2}  V_{k-1}U_{k-1}(s_k,s_{k-1}) \right].
\end{align}
By conditioning on $\mathcal F_{k-1}$ and using \eqref{omakase}, the first sum is bounded above by $\mathcal S(k-1)\cdot2e^{c-c_2n}$. Similarly, the second sum is at most $\mathcal S(k-1)\cdot c'$ by conditioning on $\mathcal F_{k-2}$ and using the first part of Lemma \ref{lem:sbe}. For the last sum, we have \begin{align}
 	&\E\left[\sum_{s_{k-1}=0,1}\sum_{s_k\geq2}  V_{k-1}U_{k-1}(s_k,s_{k-1})\bigg| \mathcal F_{k-2}\right]\\
 	&\leq 2 e^{2c} \E\left[ \left(\sum_{s_k} 1\{L_{k-1}^{k-1}(s_k)=s_{k-1}\}\right)\cdot 1_{J_{k-1}^c}\bigg|\mathcal F_{k-2}\right].
 \end{align}
It follows from \eqref{omakase} and the second part of Lemma \ref{lem:sbe} that the third sum is upper bounded $\mathcal S(k-2)\cdot 2e^{2c-c_3n}$ for some $c_3>0$. Combining all three bounds finishes the inductive step. 

To bound $\frozen$, we apply Markov's inequality together with the bound on $\mathcal S(n-1)$ to get
\begin{equation} \p(\{\frozen > \delta n\} \cap H) \leq  (c_4n+1)\exp(c + (\log(c'+4e^{2c-c_2\wedge c_3n}) - \delta c)n). \end{equation} 
By the assumption $\log c' < \delta c$, the event in question is exponentially unlikely in $n$. This completes the proof of the case $n_0=0$ and $n_1=n-1$.

The proof of the general case is almost the same as the above case. The main difference is that when $n_0 \neq 0$, the counter $L_{n_0}$ no longer enjoys the deterministic upper bound as $L_0 \leq c_4 n$. Instead it suffices to argue that $H_{n_0-1}^c$ and $L_{n_0} > 4\beta/(1-\theta)$ occur simultaneously with probability exponentially small in $n$. By the second part of Lemma \ref{lem:sbe}, out of every two added particles at $(n_0-1)K+a$, with probability at least $1-\theta$ some particle reaches $(n_0-2)K+a$ and thus visits $(n_0-1)K$. Then a standard concentration bound gives the claim.
\end{pfofprop}

\section{Independent starting configurations}\label{sec:boot}
%%%%%%%%%%%%%%%%%%%%%%%%%%%%%%

In this section we shall use Proposition \ref{prop:kick} to prove Theorem \ref{thm_main}. The main challenge is that Proposition \ref{prop:kick} requires a valid initial configuration, as well as sufficient activity in every block as stipulated in the event $H$, whereas in Theorem \ref{thm_main} we directly start from an independent configuration. We bridge the gap by giving an explicit toppling procedure, alternating between legal IDLA steps and the carpet/hole toppling procedure run on a nested sequence of intervals. The particles collected at the boundary of a smaller interval help restore a valid configuration and ensure sufficient activity on a larger interval, which, in turn, guarantees that enough particles reach the endpoints of the larger interval. We will show that this procedure runs forever with positive probability. 

%Let $X_i$ are i.i.d.\ random variables with distribution $\mu$.

Recall that $K=a^4$ is the period of the blocks in the carpet/hole procedure. Throughout Section \ref{sec:boot}, we use the same $K=e^{1.8\times10^5}$ from Proposition \ref{prop:kick} and Lemma \ref{lem:sbe}. We will prove the following re-statement of Theorem \ref{thm_main}.

\begin{theorem} \label{painting}
Let $\mu$ be a probability distribution supported on finitely many non-negative integers with mean $p \in (1-\frac{1}{3K},1)$ and $\sum_{j \geq 2} \mu (j) >0$. Let $\{X(i)\}_{i \in \Z}$ be i.i.d.\ random variables with distribution $\mu$. Then the system with starting configuration $\{X(i)\}_{i \in \Z}$ stays active a.s.
\end{theorem}

\begin{pfofthm}{\ref{thm_main}}
Theorem \ref{painting} is a quantitative version of Theorem \ref{thm_main} with $\mu_c \leq 1-\frac{1}{3K} \leq 1-\exp(-2\times10^5)$. It causes no loss of generality to assume that $\mu$ is supported on finitely many integers and $\E(\mu)<1$. To see this, note that for any probability distribution $\mu$ on $\N$ satisfying $\E(\mu) > 1-\frac{1}{3K}$ and $\sum_{j\geq2}\mu(j) > 0$, one can find another distribution $\tilde\mu$ stochastically dominated by $\mu$, such that $\tilde\mu$ is supported on finitely many integers, $\E(\tilde\mu) \in (1-\frac{1}{3K},1)$ and $\sum_{j\geq2} \tilde\mu(j) > 0$.
%$$\sum_{j \leq J} j \mu(j)>1-\frac{1}{100K}$$ and then consider $\tilde \mu$ such that $\tilde \mu(j)=\mu(j) $ for all $j=1,\dots, J$ and $\tilde\mu(0) = \mu(0) + \sum_{j>J}\mu(j)$. 
By monotonicity, if the stochastic sandpile with independent initial distributions $\tilde \mu$ does not fixate, then the model with distribution $\mu$ also does not fixate.
\end{pfofthm}
\vspace{.1in}

%It causes no loss of generality to assume that $\mu$ is supported on finitely many integers. To see this we find $J$ such that $\sum_{j=2}^J \mu(j) > 0$ and $$\sum_{j \leq J} j \mu(j)>1-\frac{1}{100K}$$ and then consider $\tilde \mu$ such that $\tilde \mu(j)=\mu(j) $ for all $j=1,\dots, J$ and $\tilde\mu(0) = \mu(0) + \sum_{j>J}\mu(j)$. 
%By monotonicity, if the stochastic sandpile with independent initial distributions $\tilde \mu$ does not fixate, then the model with distribution $\mu$ also does not fixate.
Throughout the rest of this section, we assume $\mu$ and $\{X(i)\}_{i\in\Z}$ satisfy all the conditions in Theorem \ref{painting}.

%Our main tool is the following lemma from the previous section.
%Remember that in the carpet/hole dynamics the odometer of a block is the number particles sent from a block through a transit region to a neighboring block.
%
%\begin{lemma} \label{need for bootstrap}
%Fix $1 \leq m\leq n$, $\beta>0$ and sequence $A=\{a_j\}_{0}^{m} \in\{1,\dots,a\}$. Let $\eta$ be any configuration that dominates 
%$$\eta_{A}=\sum_{i=-K+a}^{(m+1)K}\delta_i-\sum_{0 \leq j \leq m}\delta_{Kj+a_j}.$$ Then the carpet/hole dynamics with particles frozen at $-K+a$ and $(m+1)K$ is well defined.
%Let the event $E_{\beta,n}$ be that the stabilization of  $\eta$ according to the carpet/hole dynamics has  
%$\frozen>(\beta/2)n$ and
%the odometer of every block is at least $\beta n/10$. 
%
%There exists $\beta,c>0$ such that for any $n$, $m \in \{1, \dots, n\}$, sequence $A=\{a_j\}_{0}^{m} \in\{1,\dots,a\}$ and any  $\eta$  that dominates $\eta_A$
%%$$\sum_{i=-K+a+1}^{K(m+1)}\delta_i-\sum_{0 \leq j \leq m}\delta_{Kj+a_j}$$ 
%we have that
%$$\P(E_{\beta,n} \ | \ \eta)<e^{-cn}.$$
%\end{lemma}
%
%\begin{lemma} \label{also need for bootstrap}
%Normal argument but a particle in every position in all blocks.
%\end{lemma}
%

\subsection{Partial stabilization on nested intervals}

Define a sequence $\{ \tilde M_i\}_{i \geq 0}$ with $\tilde M_{i+1}=\lfloor (\tilde M_i) (1+ \ratio) \rfloor $ for some large 
$\tilde M_0$ with  $\gamma=.02$. Let $M_i=(\tilde M_i+1)K-a/2$. Consider a sequence of intervals of integers $$\intervali :=\{-M_i, \dots ,M_i\},$$ which is a shifted version of the interval $D_{2\tilde M_i+1} = (-K+a,(2\tilde M_i+1) K)$ with $2\tilde M_i+1$ many blocks. For $i\geq1$, write $$\intervalleft = \{-M_{i},\dots,-M_{i-1}-1\} $$ and  $$\intervalright = \{M_{i-1}+1,\dots, M_i\}. $$
Also define
$S^+_i:=\sum_{j \in \intervalright} jX(j)$ and 
$S^-_i:=\sum_{j \in \intervalleft} jX(j)$. Then
$$\E(S^\pm_i) =\pm E(\mu)(M_i-M_{i-1})(M_i+M_{i-1}+1)/2.$$
%and $$\E(S^-_i)= -\E(\mu)(M_i-M_{i-1})(M_i+M_{i-1}+1)/2.$$
Finally, let $\eventfourpm$ be the event that
$$|S^\pm_i -\E(S^\pm_i)| \leq .01 \gamma \tilde M_i M_i.$$
%$$\sum_{j \in \intervali} X(j) > (1 - 1/(3K)) (2M_i+1) \quad\text{and}\quad|S^\pm_i -\E(S^\pm_i)|>.01 \gamma M_i \tilde M_i.$$
%$$\eventfourplus :=\left\{ |S^+_i -\E(S^+_i)|>.005 M_i \tilde M_i \right\}$$
%$$\eventfourminus := \left\{ |S^-_i -\E(S^-_i)|>.005 M_i \tilde M_i\right\}.$$

The above notation $\intervali$ works for all $i\geq0$. For $i=-1$, we use the convention that $\intervalminusone:= \{-a/2\}$. Here $-a/2$ is the left endpoint of the center block containing site zero.
For $i=0$, we write $\intervalleft = \{-M_0,\dots,-a/2-1\}$ and $\intervalright = \{-a/2+1,\dots,M_0\}$.

We will inductively define partial stabilization procedures on the nested sequence of intervals $\{\intervali\}_{i\geq0}$ and the resulting configurations $\{Y_i(z)\}_{z \in \intervali}$. Set $Y_{-1}(-a/2)= X(-a/2)$. Suppose $\{Y_{i-1}(z)\}_{z \in \intervalminus}$ is defined, we may extend the definition $Y_{i-1}(z) := X(z)$ for all $z \notin \intervalminus$, and then define $\{Y_i(z)\}_{z \in \intervali}$ to be the configuration after the partial stabilization of 
$\intervali$ with initial configuration $\{Y_{i-1}(z)\}_{z \in \intervali}$ 
and particles frozen when they get to the boundary of the interval.
%By the abelian 
%property this is equivalent to the stabilization of the initial configuration $Y_{i-1}(z)$ in the interval $[-M_i,M_i]$. 
Let $u_i(z)$ be the site odometer at $z$ in the partial stabilization of $\intervali$ from $Y_{i-1}$.

For $i\geq1$, to go from $Y_{i-1}$ to $Y_i$ we do the partial stabilization in the following order:
\begin{itemize}
\item Run IDLA on the particles in the two intervals %of 
$$\intervalleft\setminus\{-M_i\} \text{ and } \intervalright\setminus\{M_i\}$$
% \end{eqnarray*}
%\begin{eqnarray*}
%\intervali \setminus \intervalminus%&=&\{-M_{i},\dots,-M_{i-1}-1\} \cup \{M_{i-1}+1,\dots, M_i\}\\
%&=&\intervalleft \cup \intervalright \end{eqnarray*}
and freeze particles at the boundaries $-M_i, -M_{i-1},M_{i-1}$ or $M_i$. In other words, keep toppling every particle inside both intervals until it either becomes alone at its site or reaches the boundary.
\item Run IDLA on the particles at $-M_{i-1}$ and $M_{i-1}$, freezing particles at $-M_i$ and $M_i$, until
$\intervalleft\setminus\{-M_i\}$ and $\intervalright\setminus\{M_i\}$ are completely filled or we run out of particles at $-M_{i-1}$ and $M_{i-1}$.
%Then run until one additional particle goes to each of $-M_i$ and $M_i$.
\item If the configuration inside $\intervali$ becomes valid after the previous two steps, stabilize $\intervali$ according to the carpet/hole dynamics, freezing particles at $-M_i$ and $M_i$.
\end{itemize}

For $i=0$, we perform a similar procedure by replacing the interval endpoints $-M_{i-1}$ and $M_{i-1}$ in the first and second steps by $-a/2$.

Let $Z_{i,1}$, $Z_{i,2}$ and $Z_{i,3}=Y_i$ the configurations after each of these three steps at stage $i$ respectively. Each of these configurations has at most one particle per site inside $\intervali$ except at $-M_{i-1}$ and $M_{i-1}$ (or $-a/2$ when $i=0$).
%Each of these distributions have at most one particle for at every site. The randomness is whether the particles have another move to make or not.

Let $L^\ast := \max\{j:\ \mu(j)>0\} \geq 2$, and let $\stageminusone$ be the event that
\begin{itemize}
\item every site in $\intervalzero$ contains $L^\ast$ particles initially,
\end{itemize}
which occurs with small but positive probability. Suppose $\stageiminusone$ happens, we will define the event $\stagei$ inductively. Conditioned on $\stageiminusone$, we will observe the following typical behavior during the partial stabilization on $\intervali$:
\begin{itemize}
\item After we run IDLA on the particles in  
$$\intervalleft\setminus\{-M_i\} \text{ and } \intervalright\setminus\{M_i\},$$ 
the density of sites inside each of those intervals that is covered by a particle is between $1-1/(3K)$ and 1. 
We also expect that $\eventfourplus \cap \eventfourminus$ occurs when $i\geq1$.
Call the intersection of these three events to be $\eventone$.
\item Then we run IDLA on the particles at $-M_{i-1}$ and $M_{i-1}$ (or $-a/2$ when $i=0$) until $\intervalleft\setminus\{-M_i\}$ and $\intervalright\setminus\{M_i\}$ are completely filled.
%and one particle has made it to each of $-M_i$ and $M_i$.
We expect that there are at least $.2 \tilde M_i$ particles left at both $-M_{i-1}$ and $M_{i-1}$ (or $-a/2$ when $i=0$) at the end of this step. Call this event $\eventtwo$.
\item If the typical events occur up to this point, by definition the configuration inside $\intervali$ will be valid and thus we may carry out the carpet/hole toppling procedure. After we stabilize according to the carpet/hole dynamics, there will be less than $\delta (2\tilde M_i+1) = .0004 (2\tilde M_i+1)$ blocks without a hole in $\intervali$. At both $-M_i$ and $M_i$ there will be at least $\tilde M_i/4$  particles. Moreover, the odometer $u_i(z)$ at the left endpoint $z=-a/2+m'K$ of every block $m' \in \{-\tilde M_i, \dots, \tilde M_i\}$ during this carpet/hole procedure is at least $\beta (2 \tilde M_i+1) = .0004(2\tilde M_i+1)$. Call this event $\eventthree$.
\end{itemize}

When all these three events happen, we call the procedure successful at stage $i$ and define inductively 
$$\stagei := \stageiminusone \cap \eventone \cap \eventtwo \cap \eventthree.$$
We will show that the event $ \stageinf :=\cap_{i\geq-1}\stagei$ happens with positive probability. Once this is proved, we can show that the odometer lower bound in the definition of $\eventthree$ implies the system stays active almost surely, thus proving Theorem \ref{painting}. 
The goal of the rest of Section \ref{sec:boot} is to show that $\p(\stageinf) > 0$.
The first two steps of each stage are straightforward IDLA processes.
%which restore the configuration into a valid one.
We will give lower bounds on the probabilities of $\eventone$ and $\eventtwo$ in Lemmata \ref{rock} and \ref{chalk}.
%Note that $B_i \subset \eventthree \subset \stagei$. So if $\cap_i \stagei$ occurs then the system does not stabilize.
The stabilization in the third step is more involved and most of the work in this section will come in bounding the probability of $\eventthree$. Other than Proposition \ref{prop:kick}, we will need the `center of mass' calculation in Lemmata \ref{spotted cow}--\ref{bernardo} that guarantees enough particles reaching both endpoints. In the proof of Lemma \ref{jayhawk} we will also carry out a bootstrap argument which, starting from the blocks $\pm \tilde M_{i-1}$, proves the odometer of every block is high.
%This is done in the proof of Lemma \ref{jayhawk}.
Since the definition and analysis of stage $0$ are slightly different from those of stages $i\geq1$, we only treat stage $i\geq1$ in all lemmata of Section \ref{sec:boot} but discuss the modifications for stage $0$ at the end of the section.

%If we get the typical behavior then when we start the carpet/hole phase of the stabilization in every block we see at least $a-1$ sites have particles and all of the sites in the transit regions have particles. 
%If there is a hole in a block then the hole is in the center $a$ positions of the block. 
%There are also big stacks of particles at $-M_{i-1}$ and $M_{i-1}$. 
%Based on this realization we make the following definitions.

%Remember that it causes no loss of generality to assume that $\mu$ is supported on finitely many integers and $\E(\mu)<1$.

\subsection{IDLA steps}
 
We start by bounding the probability that the first step is successful.

%\begin{lemma} \label{new glaurus}
%Let $X(i)$ be i.i.d\ with distribution $\mu$ such that $$\E(\mu)>1-1/(100K).$$
%There exists $c>0$ such that 
%$$\P\left(\eventfourplus \cap \eventfourminus \right)>1-e^{-cM_i}.$$
%\end{lemma}

%\begin{lemma} \label{ultimatum}
%Any configuration resulting from $\eventfourplus \cap \eventfourminus$
%\end{lemma}

%\begin{proof}
%%It causes no loss of generality to assume that $\mu$ is supported on finitely many integers and $\E(\mu)<1$. If not then we can truncate $\mu$ so that t does satisfy these conditions.
%As $S_i^+$ and $S_i^-$ are the sum of independent random variables bounded by $C_1M_i$ we get that  $Var(S_i^+),Var(S_i^+)<C\sqrt{M_i}  M_i$. The other term in the absolute value is $\E(S_i)$. Then the lemma follows from standard estimates.
%\end{proof}

\begin{lemma} \label{rock}
Let $X(i)$ be i.i.d. with distribution $\mu$ satisfying the conditions in Theorem \ref{painting}.
There exists $c>0$ such that for $M_i$ large enough,
%$$\P\left(\sum_{M_{i-1}+1}^{M_i-1}Z_i(j) \leq (1-1/2K)(M_i-M_{i-1})\right)>1-e^{-cM_i}.$$
$$\P\left(\eventone \right)=\P\left(\eventone \ | \ \stageiminusone \right)>1-e^{-cM_i}.$$
\end{lemma}

\begin{proof}
First note that $\eventone$ and $\stageiminusone$ are independent as they were defined on disjoint sets of independent random variables. This justifies the equality.

Next, notice that $S_i^+$ and $S_i^-$ are the sums of less than $M_i$ independent random variables each bounded by $CM_i$ for some $C$, as $\mu$ is finitely supported. Since $M_i$ and $\tilde M_i$ differ by a constant, standard concentration bounds give that
$$\P\left(\eventfourplus \cap \eventfourminus \right)>1-e^{-cM_i}$$
for some $c>0$.

Recall $\{Z_{i,1}(j)\}_{j \in \{M_{i-1},\dots,M_i\}}$ is the sequence generated by running IDLA on 
$\{X(j)\}_{j \in \{M_{i-1},\dots,M_i\}}$ with particles frozen on the boundaries.
If $Z_{i,1}(M_{i-1})=k$, 
 then there exists $y \geq 0$ such that  
$$\sum_{M_{i-1}}^{M_{i-1}+y}X(j) \geq y+k.$$
The random variables $\{X(i)\}$ have mean less than 1 and are bounded and i.i.d.
%So if the previous inequality is true, then the sum exceeds its expected value by a constant times $y$ plus $k$.
So the probability of the previous inequality is decreasing exponentially in $k$ and $y$. Summing up over all $y$ gives that the probability that $Z_{i,1}(M_{i-1})=k$ is exponentially small in $k$.

If 
\begin{equation} 
\label{prissy}
\sum_{M_{i-1}+1}^{M_i-1}Z_{i,1}(j) \leq (1-1/(3K))(M_i-M_{i-1}-1),
\end{equation}
then one of the following events must be true: 
\begin{enumerate}
\item $\sum_{M_{i-1}+1}^{M_i-1}X_i(j) \leq (1/2)(\E(\mu)+1-1/(3K))(M_i-M_{i-1}-1)$;
\item  there exists $k\geq (1/4)(\E(\mu)-(1-1/(3K)))(M_{i}-M_{i-1}-1)$ and $y \geq 0$  such that  
$\sum_{M_{i-1}}^{M_{i-1}+y}X(j) =y+k$;
\item there exists $k\geq (1/4)(\E(\mu)-(1-1/(3K)))(M_{i}-M_{i-1}-1)$ and $y \geq 0$  such that   
$\sum_{M_{i}-y}^{M_{i}}X(j) =y+k$.
\end{enumerate}
By standard estimates on sums of independent bounded random variables, the probability of the first event is decreasing exponentially in $M_i$. The probabilities of the second and third events are decreasing exponentially in $M_i$ by the argument earlier in the lemma. This completes the proof for $\intervalright$. The proof for $\intervalleft$ is similar.
%
%Thus the 
%probability that there are  less than 
%$$ (\E(\mu)+1-1/2K)(1/2)(M_i-M_{i-1})$$ particles in $(M_{i-1},M_i)$ is exponentially small in $M_i$.
%
%If an interval $I=[a,b]$ is a maximally full interval in $Z_i(j)$ (i.e. $Z_i(j)=1$ for all $j \in [a,b]$ and $Z_i(a-1)=Z_i(b+1)=0$) then $\sum_{j \in I} X(j)=b-a+1$. 
\end{proof}

Next we condition on $\stageiminusone \cap \eventone$ and bound the probability that the second step is successful.

\begin{lemma} \label{chalk}
For $\gamma=.02$, there exists $c>0$ such that for $M_i$ large enough,
$$\P(\eventtwo^C \ | \ \stageiminusone \cap \eventone) <e^{-cM_i}.$$
\end{lemma}

\begin{proof}
%Condition on any resulting configuration  
%We have defined $\{Z_{i,1}(j)\}$ to be the resulting configuration from $$\stageiminusone \cap \eventone.$$ 
By the definition of $\eventone$ and $\stageiminusone$ respectively, %conditioned on this event %$\{Z_{i,1}(j)\}$   has
$$\#\{j \in \intervalright\setminus\{M_i\} \ : \ Z_{i,1}(j)=0\}<(\tilde M_i-\tilde M_{i-1})/3,$$ 
%since $\eventone$ occurred
 and 
%$$\#\{j \in \{M_{i-1}-lK,\dots,M_{i-1}-(l-1)K-1\} \ : \ Z_{i,1}(j)=0\}<2$$ for all $l\geq 1$. 
$$\#\{j \in \{ M_{i-1}-(M_i-M_{i-1}),\dots ,M_{i-1}\}  \ : \ Z_{i,1}(j)=0\}\leq  \tilde M_i-\tilde M_{i-1},$$
since $Z_{i,1}(j)=Y_{i-1}(j)$ for $j \in \intervalminus$ and the configuration $Y_{i-1}$ remains valid after the carpet/hole dynamics of stage $i-1$.
%since $\stageiminusone$ occurred.
%for all $l\geq 1$. 

If $\eventtwo$ fails, then there are less than $.2 \tilde M_i \leq .2(1+\gamma) \tilde M_{i-1}$ particles at $M_{i-1}$ (or $-M_{i-1}$ when $i=0$) when the second step ends. We only treat the case where the event is violated at $M_{i-1}$ because the other case would be similar.
On the event $\stageiminusone$ we have at least $.25\tilde M_{i-1}$ particles initially at $M_{i-1}$,
so on $\stageiminusone\cap \eventone \cap \eventtwo^C$, the number of particles released from $M_{i-1}$ during the second step must be at least
% to occur then the number of particles that must be released at $M_{i-1}$ and perform IDLA before $\intervali$ is completely full must be at least 
\begin{eqnarray*}
.25\tilde M_{i-1}-.2(1+\gamma)\tilde M_{i-1}
&>&.045 \tilde M_{i-1}\\
&>&2\gamma \tilde M_{i-1} \\
&=&2((1+\gamma)\tilde M_{i-1}- \tilde M_{i-1})\\
&\geq&2(\tilde M_i- \tilde M_{i-1})
\end{eqnarray*}
for $\gamma = 0.02$.
After $\tilde M_i- \tilde M_{i-1}$ particles have settled to the left of $M_{i-1}$, every site in 
$\{M_{i-1}-(M_i-M_{i-1}),\dots ,M_{i-1} \}$ has one particle. Thus the nearest vacancy to the right of $M_{i-1}$ (if it exists) is at least as close as the nearest vacancy to the left of $M_{i-1}$, and the probability that 
each subsequent particle settles to the right of $M_{i-1}$ is at least 1/2.
The probability that it takes at least $(\tilde M_i-\tilde M_{i-1})$ particles to get 
less than $(\tilde M_i-\tilde M_{i-1})/3$ of them settling to the right of $M_{i-1}$ is exponentially unlikely in 
$\tilde M_i-\tilde M_{i-1}$ and thus in $M_i$.
\end{proof}

\subsection{Center of mass}

In this subsection, we carry out the `center of mass' calculation, which is useful for bounding the probability that the third step is successful.
Define
$$N_i:=\sum_{j \in \intervali}X(j)-[(2M_i+1)-(2\tilde M_i+1)].$$
This is the number of excess particles in $\intervali$ above the level of one hole per block.

\begin{lemma} \label{spotted cow}
If $\stageiminusone \cap \eventone$ occurs, then
\begin{eqnarray*}
\sum_{j \in \intervali}jZ_{i-1,3}(j)& \geq &
-(2\tilde M_{i-1}+1)(a/2)+
.25\tilde M_{i-1}M_{i-1}\\
&&-(N_{i-1}-.25\tilde M_{i-1})M_{i-1}-.02\gamma \tilde M_i M_i.
\end{eqnarray*}\end{lemma}

\begin{proof}
%The sum can be rewritten as the sum over all particles in $[-M_{i-1},M_{i-1}]$ of the locations of those particles after $\stageiminusone$.
We break this sum up into sums over particles in $\intervalminus$, $\intervalleft$ and $\intervalright$.
For the particles in $\intervalminus$, we further partition them into three groups. 
%For every block $-a/2+jK,\dots, a/2+jK$ with no hole (i.e. one particle at every location) after stage $i-1$, we designate the particle at $jK$ to be an `extra particle'.
The first group consists of all the carpet particles in $\intervalminus\setminus\{-M_{i-1}, M_{i-1}\}$, with one particle at every location except for the holes.
%The first group consists of one particle at every location in $\intervalminus\setminus\{-M_{i-1}, M_{i-1}\}$ unless the site is a hole or has an extra particle.
The second group is all the particles at $M_{i-1}$. 
The third group consists of all the remaining particles, which are all of the free particles in $\intervalminus\setminus\{-M_{i-1}, M_{i-1}\}$ plus all the particles at $-M_{i-1}$.

The absolute value of the sum over the first group is at most
$$(2\tilde M_{i-1}+1)(a/2).$$
As $\stageiminusone$ occurs, we get that the sum of the locations over the second group is at least 
$$.25\tilde M_{i-1}M_{i-1}.$$
in absolute value.
By \ref{pro:conserved}, the number of particles in the third group is $N_{i-1}$ minus the number of particles in the second group, which is at least $N_{i-1}-.25\tilde M_{i-1}$. Thus the sum over the third group is at least  
$$(N_{i-1}-.25\tilde M_{i-1})(-M_{i-1}).$$
For the particles in $\intervalleft$ and $\intervalright$, since $\eventone$ occurs, so does $\eventfourplus \cap \eventfourminus$. By the definition of those events, we get that the absolute value of the both sums combined is at most 
$.02\gamma \tilde M_i M_i.$ Combining these estimates proves the lemma.
\end{proof}

%Recall that $\beta \in (0,1)$ was defined in Proposition \ref{prop:kick}.

%Let $L_1,L_2 \in [-(1-\beta) Kn, (1-\beta)Kn]$, $N_1, N_2\geq \beta n$ and $\{a_j\}_{j=-n}^n \in [0,a]$.
%\begin{equation} 
%\eta^{\carpethole} = \eta^{\carpethole} (L,N,a_j)= N_1 \delta_{L_1} +N_2 \delta_{L_2} +\sum_{i = 0}^{nK} \delta_i - \sum_{1 \leq j \leq n} \delta_{K j+a_j}, \end{equation}
%
%Let a $\beta$-\strange 
%distribution be one that
%\begin{enumerate}
%\item dominates $\eta^{\carpethole}$ for some $n$, $N_1,N_2$, $L_1,L_2$, and $\{a_j\}_{1}^{n}$, 
%\item the number of particles
%is at most $\tilde Ln$, 
%\item not too many extra particles in $(L_1,L_2)$
%\item every position not in $(L_1,L_2)$ has a particle and
%\item every particle outside of $(L_1,L_2)$ is active with probability at least 1/3.
%\end{enumerate}

Let $\djokovich$ be the event that there are at most $\delta(2\tilde M_i+1)$ blocks without a hole at the end of stage $i$. The next lemma says that if the majority of free particles exit through the endpoints, the proportion of the particles leaving from each side cannot become more unbalanced without causing a significant shift in the center of mass.

\begin{lemma} \label{freddy krueger}
Let $\gamma = .02$. The following holds for $\delta=4\times10^{-4}$ and sufficiently large $M_i$. On the event $$\stageiminusone \cap \eventone \cap \eventtwo \cap \djokovich \cap \{ Z_{i,3}(M_i)<.25\tilde M_i\},$$
we have
$$ D_i := \sum_{j \in \intervali}jZ_{i-1,3}(j) - jZ_{i,3}(j) >(\gamma/10K)M_i^2. $$
\end{lemma}
\begin{proof}
First we show that under the same hypotheses,
\begin{eqnarray*}
\sum_{j \in \intervali} jZ_{i,3}(j)&<& 
(2 \tilde M_i+1)(a/2) + (.25 \tilde M_i+\delta(2\tilde M_i+1))  M_i\\
&&-\left(N_{i-1}+(4/3)( \tilde M_i-\tilde M_{i-1}) -.25 \tilde M_i -\delta(2\tilde M_i+1)\right)M_i.
\end{eqnarray*}
The proof of this fact proceeds in a similar manner to Lemma \ref{spotted cow}.
%The sum can be rewritten as the sum over all particles in $\intervali$ of the locations of those particles after $\stagei$. 
We decompose the particles into three groups: the carpet particles in $\intervali \setminus \{-\tilde M_i, \tilde M_i\}$, the free particles in $\intervali \setminus \{-\tilde M_i, \tilde M_i\}$ plus the particles at $M_i$, and the particles at $-M_i$.
The sum of the locations of the carpet particles is at most $$(2 \tilde M_i+1)(a/2).$$
Since $\djokovich \cap \{ Z_{i,3}(M_i)<.25\tilde M_i\}$ occurred, the sum of the locations of the second group is at most
$$(.25 \tilde M_i+\delta(2\tilde M_i+1))M_i.$$
%There are at most $.25 \tilde M_i$ particles at $M_i$ and $\beta (2M_i+1)$ blocks with frozen particles. Each of those particles  contribute at most $M_i$ to the sum. Thus the sum of the locations of the extra particles plus the particles at $M_i$ is at most
%$a(2 \tilde M_i+1)$.
By $\eventone$ we have $N_i \geq N_{i-1} +2 (2/3) (\tilde M_i - \tilde M_{i-1})$, so there are at least
$$N_{i-1}+(4/3)( \tilde M_i-\tilde M_{i-1})-(.25 \tilde M_i +\delta(2\tilde M_i+1))$$
many particles at $-M_i$, each contributing $-M_i$ to the sum. Putting these estimates together gives the desired upper bound.

Combining this result with Lemma \ref{spotted cow}, we get
\begin{eqnarray*} 
\lefteqn{ \sum_{j \in \intervali} jZ_{i-1,3}(j)-jZ_{i,3}(j)}\hspace{.25in}&& \\
%&>&.25M_{i-1}\tilde M_{i-1}-(3\tilde M_{i-1})a-(N_{i-1}-.25\tilde M_{i-1})M_{i-1}-.01\gamma M_i \tilde M_i\\
%&&-\bigg((.25+3\beta)\tilde M_i M_i+(3 \tilde M_i)a+(N_{i-1}+.99\gamma\tilde M_{i-1}(2)-.25\tilde M_i-\beta(3\tilde M_i))(-M_i)\bigg)\\
%&>&-(2\tilde M_{i-1}+1)(a-1)/2+ .25M_{i-1}\tilde M_{i-1}\\
%&&-(N_{i-1}-.25\tilde M_{i-1})M_{i-1}-.01\gamma M_i \tilde M_i\\
%&&-\bigg((.25 \tilde M_i+.01(2\tilde M_i+1))  M_i+(2 \tilde M_i+1)(a-1)/2\\
%&&+\big(N_{i-1}+2(.99(\tilde M_i-\tilde M_{i-1})-(.25\tilde M_i+.01(2\tilde M_i+1)\big)(-M_i)\bigg)\\
&>& (4/3)(\tilde M_i - \tilde M_{i-1})M_i - .5(\tilde M_i M_i - \tilde M_{i-1} M_{i-1}) + N_{i-1}(M_i - M_{i-1})\\
&&-2\delta(2\tilde M_i+1)M_i - .02\gamma \tilde M_i M_i - (2\tilde M_i+1)a\\
&>&(1.3-1.1+0-.08-.02)\gamma \tilde M_i M_i\\
%&>& (4/3)\gamma \tilde M_{i-1}M_i - 1.1 \gamma\tilde M_{i-1} M_{i-1} + 0 -.08\gamma \tilde M_i M_i - .02\gamma \tilde M_i M_i \\
&>&(\gamma/10K)M_i^2,\\
\end{eqnarray*}
for $\delta= 0.0004 = 0.02\gamma$ and $M_i$ sufficiently large.
\end{proof}

%We say that the number of steps of stage $i$ is the number of times a particles moves to an adjacent position. We denote this by  $\steps(i)$. 
Lastly, we show that with high probability there is not enough time for the center of mass to reach a displacement of the size implied by Lemma \ref{freddy krueger}. Recall the notation $D_i$ from the last lemma.
%Therefore, if most particles reach the boundary, then typically there will be at least $.25\tilde M_i$ of them on both sides.
\begin{lemma} \label{bernardo}
There exists $c>0$ such that for all large enough $M_i$,
$$\p(|D_i| > M_i^{1.6} \ | \ \stageiminusone) \leq e^{-cM_i^{0.1}}.$$
%$$\P(\steps(i)>M_i^{3.1}) \leq M_i^4e^{-cM_i^{.2}}.$$
\end{lemma}

Before proving Lemma \ref{bernardo}, we also state the following variant of it which will be useful in another instance of this `center of mass' argument. For $m_0', m_1'\in\Z$ such that $-\tilde M_i \leq m_0' \leq m_1' \leq \tilde M_i$, let $\hat m_0 := (m_0'-1)K+a/2$ and $\hat m_1 := (m_1'+1)K-a/2$, and define $\rho_{m_0',m_1'}(j):\intervali \to \Z$ to be the piecewise linear function $\rho_{m_0',m_1'}(j):= j$ when $j \in \{\hat m_0, \dots, \hat m_1\}$ and $\rho_{m_0',m_1'}(j):= \hat m_0$ (resp. $\hat m_1$) when $j < \hat m_0$ (resp. $j > \hat m_1$). Finally, we define
%$$\rho_{m_0,m_1}(j) := \begin{cases}
%\hat m_0, &j \leq \hat m_0,\\
%j, & j \in \{\hat m_0+1, \hat m_1-1\},\\
%\hat m_1, & j \geq \hat m_1. 
%\end{cases}$$
$$D'_{i,m_0',m_1'} := \sum_{j \in \intervali} \rho_{m_0', m_1'}(j) Z_{i,2}(j) -  \rho_{m_0', m_1'}(j) Z_{i,3}(j).$$
Different from $D_i$, the above definition records the change of a restricted version of center of mass in the \textit{third} step of stage $i$. Also, recall that $u_i(z)$ denotes the odometer at $z \in \intervali$ during the third step of stage $i$.

\begin{lemma}\label{semifreddo}
For $-\tilde M_i \leq m_0' \leq m_1' \leq \tilde M_i$, there exists $c>0$ such that for all large enough $M_i$,
$$\p(|D'_{i,m_0',m_1'}| > M_i^{1.6} + \max\{u_i(\hat m_0), u_i(\hat m_1)\} \ | \ \stageiminusone \cap \eventone \cap \eventtwo) \leq e^{-cM_i^{0.1}}.$$	
\end{lemma}

\begin{proof}[Proof of Lemmata \ref{bernardo} and \ref{semifreddo}]
%The proof is similar to that of Lemma \ref{bernardo} with some modifications.
We start with Lemma \ref{semifreddo}. If the inequality in question occurs, then either (1) the carpet/hole dynamics from $Z_{i,2}$ to $Z_{i,3}$ takes more than $M_i^{3.1}$ many topplings, or (2) the restricted center of mass, i.e. the sum of $\rho_{m_0', m_1'}$ function values of all particles' locations, moves by at least $M_i^{1.6} + \max\{u_i(\hat m_0), u_i(\hat m_1)\}$ during the carpet/hole dynamics, which undergoes at most $M_i^{3.1}$ topplings. It suffices to bound the probabilities of both (1) and (2).

On one hand, since $\mu$ is finitely supported, there is a deterministic upper bound on the number of particles inside $\intervali$ that is linear in $M_i$. So if the system takes at least $M_i^{3.1}$ steps, then there exists one particle that moved $CM_i^{2.1}$ many times for some $C>0$. As particles are frozen at $\pm M_i$, it must have done so without hitting these two points. For each particle, this has probability at most $e^{-cM_i^{0.1}}$ for some $c>0$. By a union bound over all particles in $\intervali$, the first event has probability at most $e^{-cM_i^{0.1}}$ for some $c>0$ and large enough $M_i$.

%The identical argument from the proof of Lemma \ref{bernardo} gives the same bound on the first event.
For the second event, let $\tau$ be the total number of topplings taken during the third step of stage $i$, and let $\mathcal L(t) \in \intervali$ and $\sigma(t) \in \{-1,+1\}$ be the starting location and direction of the $t$th toppling respectively. We may write
\begin{align*}
D'_{i,m_0',m_1'} =& \sum_{t=1}^\tau \sigma(t) 1\{\hat m_0 < \mathcal L(t) < \hat m_1\}\\
 &+ \sum_{t=1}^\tau 1\{\mathcal L(t) = \hat m_0, \sigma(t) = 1\} - \sum_{t=1}^\tau 1\{\mathcal L(t) = \hat m_1, \sigma(t)=-1\}.
\end{align*}
The first sum with the upper limit $\tau$ replaced by some time variable $\tau'=0,\dots,\tau$ is a martingale with bounded differences. So by applying the Azuma's inequality and a union bound over times, the probability that $\tau\leq M_i^{3.1}$ and the first sum exceeds $M_i^{1.6}$ is at most $e^{-cM_i^{0.1}}$ for some $c>0$ and large enough $M_i$. Since both sums in the second line above are at most $u_i(\hat m_0)$ and $u_i(\hat m_1)$ respectively, we get the desired bound on the second event. This completes the proof of Lemma \ref{semifreddo}.

For Lemma \ref{bernardo}, note that by an almost identical argument, one could state and prove an analog of Lemma \ref{semifreddo} that applies to the partial stabilization of the whole stage $i$ (from $Z_{i-1,3}$ to $Z_{i,3}$) instead of just the third step ($Z_{i,2}$ to $Z_{i,3}$). So Lemma \ref{bernardo} is the special case where $m'_0 = - \tilde M_i$ and $m'_1 = \tilde M_i$. Note that the odometer term in the inequality disappears because no toppling starts from the boundary. 
\end{proof}

\subsection{Carpet/hole step}
Finally, we bound the probability that the third step is successful. 

\begin{lemma} \label{jayhawk}
Let $\gamma=.02$. There exist some $\alpha,c >0$ such that for $\delta = \beta = 4\times10^{-4}$ and sufficiently large $K, M_i$, we have
$$\P(\eventthree^C \ | \ \stageiminusone \cap \eventone \cap \eventtwo) <e^{-c M_i^\alpha}.$$
\end{lemma}

%This requires some more notation.

\begin{proof}
We pick constants $\gamma, \delta, \beta, K$ and $M_i$ as stated so that Proposition \ref{prop:kick} and all previous lemmata in Section \ref{sec:boot} are true.
If $\eventthree^C$ occurs, then one of these three events during the third step of stage $i$ must happen:
\begin{enumerate}
\item the odometer $u_i(-a/2+m'K)$ at the left endpoint of some block $m'$ is less than $\beta (2\tilde M_i+1)$;
\item the odometer $u_i(-a/2+m'K) \geq \beta (2\tilde M_i+1)$ for every block $m'$ but $\frozen > \delta(2\tilde M_i+1)$;
\item $\frozen \leq \delta(2\tilde M_i+1)$ but there are less than $.25 \tilde M_i$ particles at either $-M_i$ or $M_i$.
%\item $\frozen> .01(2\tilde M_i+1)$ and the configuration after stage 2 is fully valid,
%\item $\frozen> .01(2\tilde M_i+1)$, the configuration after stage 2 is not fully valid and the odometer is less than $\beta n$ at some block or
%\item the odometer at at least one site is zero.
\end{enumerate}

The second event is exponentially unlikely in $M_i$ by the second statement of Proposition \ref{prop:kick}. For the third event which we denote as $\mathcal T$, the `center of mass' calculation, Lemmata \ref{freddy krueger} and \ref{bernardo}, as well as its symmetric version shows that $\p(\eventone \cap \eventtwo \cap \mathcal T \ | \ \stagei)$ is exponentially small in $M_i^{0.1}$. So by Lemmata \ref{rock} and \ref{chalk} we get the desired bound on $\p(\mathcal T \ | \ \stagei \cap \eventone \cap \eventtwo)$. In the remainder of the proof, we focus on the bootstrap argument which bounds the first event.

%A symmetric argument covers the case that there are too few particles at $-M_{i}$. 

%The probability of the second event is bounded by A1 of Proposition \ref{prop:kick}.
%The probability of the third event is bounded by A2 of Proposition \ref{prop:kick}.

As we are conditioning on $\stageiminusone \cap \eventone \cap \eventtwo$, there are a large number (at least $.2\tilde M_i$) of initial particles at $-M_{i-1}$ right before the third step of stage $i$. By the second statement of Lemma \ref{lem:sbe}, with probability exponentially close to one, the odometer $u_i(-M_{i-1}-a)$ at the left endpoint of the block containing $-M_{i-1}$ must be at least $\beta(2\tilde M_i + 1)$. So if there is some block whose left endpoint odometer is less than $\beta (2\tilde M_i+1)$, then there exists some largest interval $I'=\{m'_0, \dots, m'_1\}$ such that $-\tilde M_{i-1} \in I'$ and $u_i(-a/2+m'K) \geq \beta (2\tilde M_i+1)$ for any block $m' \in I'$.

First, we claim it is exponentially unlikely in $M_i$ that $\frozen(m'_0, m'_1) > \delta(2\tilde M_i + 1)$. Indeed, by Proposition \ref{prop:kick} for any interval $I''=(m_0'', m_1'')$ satisfying $|I''| \geq \delta(2\tilde M_i+1)$, the probability that $m_0'' = m_0'$ and $m_1'' = m_1'$ but $\frozen(m_0'', m_1'') > \delta(2\tilde M_i+1)$ is exponentially small in $M_i$. The inequality $\frozen(m_0'', m_1'') < \delta(2\tilde M_i+1)$ holds trivially for other intervals $I''=(m_0'', m_1'')$ such that $|I''| < \delta(2\tilde M_i+1)$, so taking a union bound over all such intervals proves the claim.

Next, we show that with exponentially high probability either $m'_0 = -\tilde M_i$ or $m'_1 = \tilde M_i$. Suppose not, then by definition both $u_i(-a/2+(m'_0-1)K)$ and $u_i(-a/2 + (m'_1+1)K)$ are less than $\beta (2\tilde M_i+1)$. So with the notation $\hat m_0$ and $\hat m_1$ from Lemma \ref{semifreddo}, the number of free particles that stayed inside $(\hat m_0, \hat m_1)$ before the third step of stage $i$ but become at or to the right of $\hat m_1$ afterwards is at most $\beta (2\tilde M_i+1) + 1$ by \ref{pro:end}. The same also holds for those becoming at or to the left of $\hat m_0$.
%Fix some interval containing $-M_{i-1}$. For this interval if we stabilize and freeze upon leaving the interval we have that at most $\beta (2\tilde M_i+1)$ particles leave the interval on either side. Otherwise the odometer of the adjacent blocks would be at least $\beta (2\tilde M_{i}+1)$. 
Since there are at least $.2\tilde M_i$ particles inside $(\hat m_0, \hat m_1)$ before the third step,  at least $.2\tilde M_i - 2\beta (2\tilde M_i+1) - 2 \geq .19\tilde M_i$ particles remain frozen in this interval at the end of the carpet/hole procedure. This is exponentially unlikely in $M_i$ by the above claim.
%On one hand, this means that the interval $I'$ is of length  at least $.19\tilde M_i$. On the other hand, in the above claim we proved that the probability that $|I'| \geq .19 \tilde M_i$ and $\frozen(m'_0, m'_1) \geq .19 \tilde M_i$ is exponentially small. This shows that we have either $m'_0 = -\tilde M_i$ or $m'_1 = \tilde M_i$ with exponentially high probability.

%On the other hand, by Proposition \ref{prop:kick} for any interval $(m_0'', m_1'')$ satisfying $m_1'' - m_0'' + 1 \geq .19\tilde M_i$, the probability that $m_0'' = m_0'$ and $m_1'' = m_1'$ but $\frozen(m_0'', m_1'') \geq .19\tilde M_i$ is exponentially small in $M_i$. Taking a union bound over all such intervals proves the claim.
%Thus we can bound the probability of this by A1 of Proposition \ref{prop:kick}. As there are at most $4\tilde M_i^2$ possible intervals of blocks containing $M_{i-1}$
%and for each block we get a uniform bound of $e^{-c\tilde M_i}$ we get a bound of $4\tilde M_i^2e^{-c\tilde M_i}$ for the probability of the second event and the interval does not contain block $-\tilde M_i$ or $\tilde M_i$.

Thus we may assume, without loss of generality, that $m'_0 = -\tilde M_i$, as the case $m'_1 = \tilde M_i$ would follow from a symmetric argument.
%Since $m'_0 = -\tilde M_i$ and $-\tilde M_{i-1} \in I'$, we have $|I'|\geq\gamma\tilde M_i/2$,
By the claim above we may also assume that $\frozen(m'_0, m'_1) \leq \delta (2\tilde M_i+1)$ which happens with exponentially high probability. Our next goal is to show that if $m'_0 = -\tilde M_i$ and $\frozen(m'_0, m'_1) \leq \delta (2\tilde M_i+1)$, then we have $m'_1 = \tilde M_i$ with exponentially high probability. This would give the bound on the first event, thus proving Lemma \ref{jayhawk}. 

From now on we suppose that $m'_0 = -\tilde M_i$ and $\frozen(m'_0, m'_1) \leq \delta (2\tilde M_i+1)$ but $m'_1 < \tilde M_i$. We will show that this is exponentially unlikely in $M_i^{0.1}$ by carrying out another `center of mass' calculation.
Recall the notations from Lemma \ref{semifreddo}. By assumption, we have $\hat m_0 = - M_i$. Let $V'_{i,2}$ (resp. $V'_{i,3}$) be the set of locations of the carpet particles of $Z_{i,2}$ (resp. $Z_{i,3}$) inside $(\hat m_0, \hat m_1)$. Similar to Lemma \ref{freddy krueger}, $$\left\vert\sum_{j \in \intervali} \rho_{m'_0 ,m'_1}(j) \delta_{V'_{i,2}}(j) -  \rho_{m'_0 ,m'_1}(j) \delta_{V'_{i,3}}(j)\right\vert \leq (m'_1 - m'_0 + 1)a.$$
For the rest of the particles, we consider a new kind of sum with each location shifted by $-\hat m_0$ for simplicity. Let $R_{i,2}$ (resp. $R_{i,3}$) be the number of particles of $Z_{i,2}$ (resp. $Z_{i,3}$) in $\intervali$ at or to the right of $\hat m_1$. On one hand, as there are at least $.2\tilde M_i$ particles at $-M_{i-1}$ initially, we get
$$\sum_{j \in \intervali} (\rho_{m'_0 ,m'_1}(j) - \hat m_0) (Z_{i,2}(j) - \delta_{V'_{i,2}}(j)) \geq (.2\tilde M_i-1) (M_i - M_{i-1}) + R_{i,2}(\hat m_1 - \hat m_0).$$
On the other hand, by assumption we have
$$\sum_{j \in \intervali} (\rho_{m'_0 ,m'_1}(j) - \hat m_0) (Z_{i,3}(j) - \delta_{V'_{i,3}}(j)) \leq (R_{i,3} + \delta(2\tilde M_i+1))(\hat m_1 - \hat m_0).$$
By the definition of $(m'_0, m'_1)$, we have $u_i(\hat m_1) < \beta(2\tilde M_i+1)$, which implies $R_{i,3} - R_{i,2} < \beta(2\tilde M_i+1)+1$ by \ref{pro:end}. Using this fact and the inequality $\hat m_1 - \hat m_0 < 2 M_i$, we combine the above estimates into
\begin{eqnarray*}
D'_{i,m_0',m_1'} &>& (.2\tilde M_i-1) (M_i - M_{i-1}) -2(\beta+\delta)(2\tilde M_i+1)M_i - 2\tilde M_ia\\
&>&(.19 - .16)\gamma\tilde M_i M_i\\
&>& (\gamma/40K) M_i^2,
\end{eqnarray*}
for $\beta = \delta = 0.02\gamma$ and $M_i$ sufficiently large.

Lastly, since $u_i(\hat m_0) = 0$ and $u_i(\hat m_1) <\beta(2\tilde M_i+1)$, Lemma \ref{semifreddo} and a union bound over all intervals $I''=(m''_0, m''_1)$ imply the above event occurs with exponentially small probability in $M_i^{0.1}$. This completes the `center of mass' argument and thus the proof of Lemma \ref{jayhawk}.
\end{proof}

\begin{pfofthm}{\ref{painting}} As mentioned above, we've only proved Lemmata \ref{rock}--\ref{jayhawk} for stage $i \geq 1$. We briefly discuss the counterparts and proofs for stage $i=0$ before putting everything together. Recall the different definitions of $\intervalminus$, $\intervalleft$, $\intervalright$, $\stageiminusone$, $\eventone$ and $\eventtwo$ when $i=0$. If $\stageminusone$ occurs, i.e. there are exactly $L^\ast\geq2$ particles at every location in $\intervalzero$ before stage $0$, then $\eventone$ occurs almost surely as there will be exactly one particle at every site of $\intervalleft \setminus \{-M_0\}$ and $\intervalright \setminus \{M_0\}$. Also, the second step of stage $0$ will be trivial, so by a concentration bound on the IDLA particles in the first step, there will be at least $0.49(L^\ast-1)(2M_0)$ particles at $-a/2$ with probability exponentially close to one. This shows that Lemmata \ref{rock} and \ref{chalk} hold for $i=0$.

For Lemma \ref{spotted cow}, in fact on the event $\stageminusone$ we have $\sum_{j \in \intervalzero} j X(j) = 0.$
In the proof of Lemma \ref{freddy krueger}, by using the fact $N_0 \geq (L^\ast - 1)(2M_i+1)$ and the counterpart of Lemma \ref{spotted cow}, we get
$D_i > M_i^2$. Lemmata \ref{bernardo} and \ref{semifreddo} work for all $i\geq0$. Finally, essentially the same proof works for Lemma \ref{jayhawk} by using the above counterparts of Lemmata \ref{spotted cow} and \ref{freddy krueger} and replacing $-M_{i-1}$ by $-a/2$. In other words, Lemma \ref{jayhawk} also holds for $i=0$.

From the definition of $\stagei$ and Lemmata \ref{rock}, \ref{chalk}, and \ref{jayhawk} we get
\begin{eqnarray}
\lefteqn{\P(\stagei^C\ | \ \stageiminusone)} \hspace{.25in} \text{}&&\\ \nonumber
&=&\P(\eventone^C\ | \ \stageiminusone)\\  \nonumber
&&+\,\P(\eventtwo^C \ | \ \stageiminusone \cap \eventone)\\  \nonumber
&&+\,\P(\eventthree^C \ | \ \stageiminusone \cap \eventone \cap \eventtwo)\\
&<&e^{-cM_i}+e^{-cM_i}+e^{-cM_i^\alpha}
\end{eqnarray}
for some $c, \alpha>0$ and sufficiently large $M_i$. We have that $\stageminusone$ happens with small but positive probability and $M_i$ grows exponentially, so there exists a large enough $M_0$ such that
%the final quantity is summable. Thus  Borel-Cantelli implies that
$$\P(\cap_{i \geq -1} \stagei)>0.$$
Thus by the definition of $\eventthree$ the odometer at site $z=-a/2$ is infinite with positive probability.
By Lemmata \ref{0/1law} and \ref{lem:halftop_abelian} and \ref{pro:legal}, this implies
$$\P( \text{stochastic sandpile stays active}) = 1.$$

\end{pfofthm}

%The support of the random walk step distribution $\{\pm1\}$ generates the group $(\Z, +)$, so with positive probability the odometer at \textit{every} site is infinite. By Lemma \ref{lem:halftop_abelian} and \ref{pro:legal}, this implies that 
%$$\P( \text{stochastic sandpile stays active})>1.$$
%This is a translation-invariant event, and by ergodicity
%$$\P( \text{stochastic sandpile stays active}) = 1.$$
%\end{pfofthm}

%Next note that $\cap_{i \geq 0}  \stagei$ implies that the stochastic sandpile does not stabilize for that 
%initial configuration and set of instructions. These two facts imply that
%$$\P( \text{stochastic sandpile stabilizes})<1.$$
%By the ergodic theorem stabilization is a 0-1 event so
%$$\P( \text{stochastic sandpile stabilizes})=0.$$

%%%%%%%%%%%%%%%%%%%%%%%%%%%%%%%%%%
\section{Carpet processes}
\label{sec:carpets}
%%%%%%%%%%%%%%%%%%%%%%%%%%%%%%%%%%

We now turn to analyzing the dynamics within a single block, which involves keeping track of the parity configuration as it evolves during the carpet/hole procedure. At a high level, our aim is to show that the parity configuration typically has many $1$'s. Sections \ref{sec:carpets} and \ref{sec:ugly}  are devoted to analyzing an auxiliary version of the particle dynamics, which allow us to unravel the complex combinatorics of our half-toppling procedure, by retaining some independence throughout the process. In particular, Section \ref{sec:carpets} provides a formulation of such independence, while Section \ref{sec:ugly} controls the possible impact of the remaining correlation. The outcomes are Lemmata \ref{geo} and \ref{ugly computation}.
%The culmination is Lemma \ref{ugly computation}.
Finally, these two lemmata are used in Section \ref{sec:renewal} to prove Lemma \ref{lem:sbe} by bounding the probability that a block becomes frozen. 

%%%%%%%%%%%%%%%%%%
\subsection{Carpet process and excursions} 
\label{sec:carpet_excur}
%%%%%%%%%%%%%%%%%%

We focus on the dynamics happening inside a single block, with the hot particle at the hole initially. In Section \ref{sec:carpet_excur}, we shall recall some relevant aspects of the dynamics from Section \ref{sec:toppling} and provide the notations that will be used throughout Sections \ref{sec:carpets}, \ref{sec:ugly} and \ref{sec:renewal}.

Recall that the block is made up of a string of $a$ carpet particles, except for a single empty site -- the hole -- and some parity configuration, $\tilde{\omega} \in \tilde{\Omega} = \{0,1\}^{[a]}$. For the remainder of the analysis we shift coordinates so that the leftmost point of the block is position $0$. For any such $\tilde{\omega}$, denote by $L(\tilde{\omega})$ the location of the leftmost 1 of $\tilde{\omega}$, 
\begin{equation}L(\tilde\omega) := \inf\{i: \tilde\omega(i)=1\},\end{equation}
which is the position of hole when the hot particle is inside the hole.
If $\tilde \omega$ is identically zero then we write $L(\tilde \omega)=a+1$. Let $Q: [\mathrm{length}(Q)] \to \Z$ denote the path taken by the hot particle starting at position $L(\tilde \omega)$ and ending on the first return to that site, where $\mathrm{length}(Q)$ is the number of steps taken in the path $Q$ and $Q(u)$ is the position of $Q$ at step $u$. Define $\local_Q$, the local time of $Q$, by 
$$\local_Q(i) := \#\{u \in [\mathrm{length}(Q)-1]:\ Q(u)=i\}.$$ 
Then define $\X$, the parity of the local time of $Q$, by 
$$\X(i):= \local_Q(i)\bmod 2.$$ 

Recall from Section \ref{sec:toppling} that our carpet process inside a single block is a Markov chain $\tilde \omega^t$, one step of which is defined as follows: Let $L^t=L(\tilde \omega^t)\in [0,a]$. 

\begin{itemize} \item (Do an excursion) The hot particle at $L^t$ performs a random walk $Q^t$ on $\Z$ that starts at $L^t$ and ends on its first return to $L^t$. 

\item (Update the parity configuration) For each $i \in [a]$, set 
\begin{equation} \label{rental} \tilde \omega^{t+1}(i)= \tilde \omega^t(i)+\Xt(i)\bmod 2.\end{equation}
\end{itemize}
We will be only interested in the carpet process $\tilde\omega^t$ until the first time the block gets frozen. In Lemma \ref{snack}, we will deal with what happens thereafter by using the renewal property of carpet processes.

To describe the random walk path $Q^t$, we introduce three types of random walk paths, with a right (resp. left) version for each kind, namely: for $L \in [0,a]$, 

\begin{enumerate}[(Type 1), leftmargin=68pt] \item a simple random walk started from $L$ and conditioned to hit $K$ (resp. $-K+a$) without returning to $L$;

\item a simple random walk started at $a$ (resp. $0$) and conditioned to hit $L$ without reaching $K$ (resp. $-K+a$);

\item a right (resp. left) type 1 walk followed by a corresponding version of type 2 walk.

\end{enumerate}

%\begin{itemize} \item[(Type 1)] A simple random walk started from $L$ and conditioned to hit $K$ before hitting $L$ again

%\item[(Type 2)] A random walk started at $-K+a$, reflected at $-K+a$, and stopped on hitting $L$

%\item[(Type 3)] A random walk started at $L$ that reaches $-K+a$ before hitting $L$, and then a walk of type 2.

%\end{itemize}

Most of the time, $Q^t$ is a simple random walk excursion on $\Z$ starting and ending at $L^t$. However, when the hot particle reaches a neighboring block, by \ref{pro:freewhere_afterhot} it respawns from one of the boundary points: $Q^t$ could also be

\begin{enumerate}
\item \textit{`long excursion'}: an emission to the left (resp. right) from $L^t$ followed by a re-arrival from the left (resp. right) boundary, i.e. a path of type 3;
\item \textit{`double-sided path'}: an emission to the left (resp. right) followed by a re-arrival from the right (resp. left) boundary, a union of type 1 and type 2 paths on opposite sides of $L^t$;
\item \textit{`failed re-arrival'}: after the emission, the newly designated hot particle fails to arrive at $L^t$ but makes another emission to a neighboring block.
\end{enumerate}

Among the three corner cases, the `long excursion' is very similar to a random walk excursion in terms of their change of the parities inside the block. In the following, we will simply refer to a long excursion as an \textit{`excursion'}.
%We call $Q^t$ in the second case a \textit{`double-sided path'} and $Q^t$ in the third case a \textit{`failed re-arrival'}.

%This occurs when the hot particle makes an excursion to the right from $L$, hits the right neighbor block before returning to $L$, and then a free particle to the left of $L$ is declared hot and hits $L$ before the left neighbor block. We refer to this case as an `emission to the right.' Since we only use topplings in $\mathcal{F}_i$, emissions to the left always result in the hot particle returning to $L$ from the left (or the process ends before a return), so also we refer to these left emissions as `excursions to the left.'

%%%%%%%%%%%%%%%%%%
\subsection{Random walk parities}
%%%%%%%%%%%%%%%%%%

%We want to analyze how long the carpet process takes to reach the all 0's state, i.e. a frozen block. Ignoring dependencies between particles, each time a random walk path crosses a site, it refreshes the parity there, setting it to 0 or 1 independently of the history of the process.
%In principle, this suggests the all zeros configuration has exponentially small probability.

We will need the following lemma regarding the parity function for a single random walk excursion. Note that the lemma no longer holds at $i=\extreme$ for excursions to the right, or at $i=L-1$ for those to the left -- the parity in such case is determined by the conditioning.

\begin{lemma} \label{loss}
Let $Q$ be a SRW excursion on $\Z$ starting and ending at $L$.
Let $$\range=  \{i: \exists\, u \text{ with }Q(u)=i\}.$$
Let $\extreme$ be its other extreme value of $\range$ (besides $L$).
If $\extreme > L$, then
for any $i\in \range\setminus \{L,\extreme\}$,
\begin{equation} \label{loss_bound} \frac{1}{6}<\P\bigg(\X(i)=0\ | \ \{\X(j')\}_{j' \in[L, i)}, \extreme \bigg)<\frac{5}{6}. \end{equation}
Instead if $\extreme < L$, then
for any $i\in \range\setminus \{L-1, L\}$,
\begin{equation} \label{loss_bound} \frac{1}{6}<\P\bigg(\X(i)=0\ | \ \{\X(j')\}_{j' \in[\extreme, i)}, \extreme \bigg)<\frac{5}{6}. \end{equation}
\end{lemma}

%For any $i\in \range\setminus \{L,\extreme\}$,
%\begin{equation} \label{loss_bound} \frac{1}{6}<\P\bigg(\X(i)=0\ | \ \{\X(j))\}_{j : |L-j| <|L-i|}, \extreme \bigg)<\frac{5}{6}. \end{equation}

\begin{proof}
%Recall that $\mathrm{length}(Q)$ is the length of the excursion and
We only treat the case where the excursion goes to the right from $L$ so that $L\leq Q(u)\leq \extreme$ for all $0\leq u\leq \mathrm{length}(Q)$, as the proof of the other case would be similar.  Fix $i$, a sequence of parities $(x_j)_{j=L}^{i-1}$, and an extreme value $E>i$.  For a deterministic excursion $q$ away from $L$, we let $\tau_0(q)=\tau'_0(q):=0$ and recursively define 
\[ \tau_j(q) := \inf\{ t> \tau'_{j-1}(q) : q(t)=i\} \]
and
\[  \tau'_j(q) := \inf\{ t> \tau_{j}(q) : q(t)\in\{ i-1,i+2\}\}.\]
Note that between $\tau_j(q)$ and $\tau'_j(q)$, the excursion $q$ might bounces back and forth between $i$ and $i+1$. 

We will count the number of visits to $i$ in these intervals conditioning on the entire rest of the path. Let $B= \cup_j (\tau_j(q),\tau'_j(q))$ and define $\Gamma(0)=0$ and $\Gamma(j) = \inf\{ j>\Gamma(j-1): j\notin B\}$.  We then define $Reduce_q(j) = q(\Gamma(j))$, which is the result of removing the bounces of $q$ back and forth between $i$ and $i+1$.
Let $\tilde q$ be any deterministic excursion away from $L$ such that $\text{Parity}_{\tilde q}(j)=x_j$ for $L\leq j\leq i-1$ and $max(\tilde q)=E$. Let $\hat q= Reduce_{\tilde q}$ and $\tilde k = \max\{ k : \tau'_k(\tilde q)<\infty\}$. Consider $\tilde Q$ being distributed like $Q$ conditioned on $Reduce_Q=\hat q$.  For $1\leq j\leq \tilde k$, let $N_j$ be the number of visits of $\tilde Q$ to $i$ in $[\tau_j(\tilde Q), \tau'_j(\tilde Q))$.

When $E \geq i+2$,
a direct computation shows that $N_1,\dots, N_{\tilde k}$ are independent $\textrm{Geometric}(3/4)$ random variables.
%depending on whether $q(\tau'_j(q)) = i-1$ or $i+2$.
Let
\[ N= \sum_{j=1}^{\tilde k} \mathbf{1}\{N_j \textrm{ is odd}\}.\]
Then $N$ has a $\textrm{Binomial}(4/5)$ distribution and a direct computation shows that $1/5 \leq \P(N\textrm{ is even}) \leq 4/5$. Since $\parity_{\tilde Q}(i) = \parity(N)$, the result follows by unwinding the conditioning.

When $E=i+1$,  we always have $q(\tau'_j(q)) = i-1$ and there must exist some $k_0 \in[1,\tilde k]$ with $N_{k_0} > 1$. So conditioned on each $N_k$ being either equal to or greater than one with at least one inequality $N_{k_0}>1$ for some $k_0$, we have $N_1, \dots, N_{\tilde k}$ distributed either as 1 almost surely or as $\text{Geometric}(3/4)$ random variables. The result then follows from a similar argument.
\end{proof}

In Section \ref{sec:carpet_excur}, we introduced the random walk paths of Type 1--3, with a right (resp. left) version for each kind. We also require a similar result for these paths.

\begin{lemma} \label{loss_ouch}
Define $\extreme$ as in Lemma \ref{loss}.
For a right version walk $Q$ of any type and any $i \in [L+1, a] \setminus\{\extreme\}$, 
 $$\frac{1}{6}<\P\bigg(\X(i)=0\ | \ \{\X(j')\}_{j' \in [L,i)}, \extreme \bigg)<\frac{5}{6}.$$
For a left version walk $Q$ of any type and any $i \in [0, L-2]$, 
$$\frac{1}{6}<\P\bigg(\X(i)=0\ | \ \{\X(j')\}_{j' \in [-K+a, i)}, \extreme \bigg)<\frac{5}{6}.$$
\end{lemma}

\begin{proof} The proof follows similarly to that of Lemma \ref{loss}. \end{proof}

%\begin{lemma} \label{loss_ouch} For a walk $Q$ of type 1 and any $i \in [L, a]$, 
% $$\frac{1}{6}<\P\bigg(\X(i)=0\ | \ \{\X(j'))\}_{j' \in [L,i)} \bigg)<\frac{5}{6}.$$
 
%For any walk $Q$ of type 2 or 3 and any $i \in [0, L-1]$, 
%$$\frac{1}{6}<\P\bigg(\X(i)=0\ | \ \{\X(j'))\}_{j' \in [-K+a, i)} \bigg)<\frac{5}{6}.$$
%\end{lemma}

%\begin{proof} For $Q$ of type 1, the proof follows identically to that of Lemma \ref{loss}. For $Q$ of type 2 or 3, note that started from the final time $T$ that $Q$ hits site $-K+a$, the distribution of the path is the same as in type 1. Thus by the result for type 1, the parity of the number of visits to $i$ after time $T$ satisfies the desired bound; and since the bound is symmetric with respect to swapping 0 and 1, it doesn't matter what the parity of the number of visits at $i$ up to time $T$. \end{proof}

%Let $L^t=L(\tilde \omega^t)\in [0,a]$.
We shall give a description of the cases where Lemma \ref{loss} or \ref{loss_ouch} fails when applied to the random walk path $Q^t$ defined in \ref{sec:carpet_excur}. Following the comment above Lemma \ref{loss}, we consider the event that the random variable $Z \in \{0,1\}$ is \textit{uniquely determined} by the $\sigma$-algebra $\mathcal R$, that is,
\begin{equation}
\p\left(Z = 1 \mid \mathcal R\right) \in \{0, 1\}.
\end{equation}
Let $\mathcal T(Q)$ denote whether the random walk path $Q$ is an excursion, a double-sided path, or a failed re-arrival. 
Let $$\range(Q) := \{i: \exists\, u \text{ with } Q(u) = i\}.$$
%Recall that $$\range(Q^t) := \{i: \exists\, u \text{ with } Q^t(u) = i\}.$$
%Also let $$\rangel(Q^t) := \range(Q^t) \cap (-\infty, L^t) \quad\text{and}\quad \ranger(Q^t) : = \range(Q^t) \cap (L^t, \infty).$$

\begin{lemma}\label{unholy}
Consider the $\sigma$-algebra $\mathcal R_{t,i}$ generated by $$\mathcal T(Q^t), L^t, \range(Q^t) \text{ and } \{\Xt(j')\}_{j' < i}.$$
Suppose $Q^t$ is not a failed re-arrival.
Almost surely, the random variable $\Xt(i)$, for some $i\in (\range(Q^t) \setminus \{L^t\}) \cap [0,a]$, is uniquely determined by $\mathcal R_{t,i}$ if and only if one of the following is true:
\begin{itemize}
%\item $i = L^t$;
\item $Q^t$ is an excursion to the left or a double-sided path, and $i = L^t-1$;
\item $Q^t$ is an excursion to the right or a double-sided path, and $i = max(\range(Q^t))$.
\end{itemize}
Otherwise, a.s.
\begin{equation}
\p(\Xt(i) = 1 \mid \mathcal R_{t,i}) \in (1/6,5/6).
\end{equation}

\end{lemma}

%\begin{lemma}\label{gershwin}
%Let $0\leq j \leq i \leq a$. Almost surely, the random variable $\Xt(i)$ is uniquely determined by $L^t$, $\range(Q^t)$ and $\{\Xt(j')\}_{j' < j}$ if and only if one of the following is true:
%\begin{itemize}
%\item $i = L^t$.
%\item $Q^t$ is an excursion to the left or a double-sided path and $|\rangel(Q^t)\setminus (-\infty, j)| = 1$. The singleton of such set has to be 
%\item $i = L^t-1$;
%\item $Q^t$ is an excursion to the right or a double-sided path, and $i = max(\range(Q^t)) \leq a$.
%\end{itemize}
%\end{lemma}

\begin{proof}
This is a consequence of the discussion in this section.
\end{proof}

%%%%%%%%%%%%%%%%%%
\subsection{Auxiliary carpet process $\omega^t$}
\label{sec:aux}
%%%%%%%%%%%%%%%%%%

%We will study the sequence of random variables $\{L(\tilde \omega^t)\}_t.$ We aim to show that if 
%$a$ is large then the expected value of the smallest $t$ such that $L(\tilde \omega^t)=a+1$ will be exponentially large in $a$.

We start with the following heuristic. 
Suppose we had the following two facts:
\begin{enumerate}
\item there exist $\epsilon, D>0$ such that if $L(\tilde \omega^t)>\epsilon a$ then
$$\E\left( L( \tilde \omega^{t+1})-L(\tilde \omega^t) \mid \tilde \omega^t \right)<-D;$$
\item there exists $G > 0$ such that for all $a$ and all  $\tilde \omega^t$,
$$\E\left( L( \tilde \omega^{t+1})-L(\tilde \omega^t) \mid \tilde \omega^t \right)<G.$$
\end{enumerate}
Since the process $\tilde \omega^t$ is a Markov process, the above two facts, together with some regularity conditions, would be sufficient to show that the expected value of the smallest $t$ with $L(\tilde \omega^t)=a+1$ is growing exponentially in $a$, that is, the block does not get frozen for a long time.
Unfortunately neither of the facts above is true.

To make this heuristic rigorous we will introduce an auxiliary carpet process $\{\omega^t\}_{t\geq0}$ and a related filtration $\{\reveal^t\}_{t\geq0}$, where we do not reveal the full information about the state of $\tilde \omega^t$.
%instead we reveal just enough information to determine the starting point of the next excursion.
%The process $\omega^t$
%is denoted by $\{\tilde \omega^t\}$ and it 
%is a hidden Markov process
It turns out that $$\E\left( L( \omega^{t+1})-L( \omega^t) \mid \reveal^t \right)$$ will have a more uniform drift, which allows us to make the counterparts of the above heuristic rigorous. 
\\
%Instead we only reveal $L(\tilde \omega^0)$, and partial information about subsequent excursions. Roughly speaking, we reveal just enough information after each excursion to determine the starting point $L$ of the next excursion. 

%\begin{enumerate}
%\item $\omega$ contains at least one 1, unless $\omega = (0, 0, \ldots, 0)$,
%\item $\omega(j) = 0$ for $j < L(\omega),$
%\end{enumerate}
%where
% $$L(\omega) := \inf\{i: \omega(i)=1\}.$$

%This process is not a Markov chain but since it is a function of one it still involves lots of independence. 
%The definition of $\reveal^t$ is chosen to reveal (close to) the minimum amount of information needed to determine $L(\tilde \omega^t).$ Via the ?'s, the sequence $\omega^t$ will keep track of how much information in $\tilde \omega^t$ has been revealed by $\reveal^t$.

%\begin{definition}\label{ell_defs} Set the following notations:

%\begin{itemize} 
%\item For any $Q$ and for each $i \in [a]$ we recall the definition of the zero one valued random variables $\X(i)$
%$$\X(i)= \#(u:\ Q(u)=i)\bmod 2.$$ \item For an excursion $Q$ to the left or right, let $\ell=\ell(Q)$ be the maximum distance away from $L$ reached by the excursion $Q$. For an emission to the right, let $\ell(Q) = K-L$. Note $\ell \geq 1$. 
%\item For an excursion $Q$ to the left, let $\hat \ell$ be the maximum distance away from $L$ with $\hat \ell < \ell$ and such that that the local time of the excursion at $L-\hat \ell$ is odd, i.e.\ $\X(L-\hat \ell)=1$ and  $\X(L-\hat \hat \ell)=0$ for all $\hat \hat \ell,$ $\ell > \hat \hat \ell > \hat \ell$. 

%\end{itemize}

%\end{definition}

We inductively define an increasing family of $\sigma$-algebras $(\reveal^{t})_{t \geq 0}$ and the auxiliary carpet process $(\omega^{t})_{t\geq0}$. %together with a family of $0/1$-valued random variables $\hidden(s, i, t)$.
Recall the notations from Section \ref{sec:carpet_excur} and Lemma \ref{unholy}. Every $\sigma$-algebra $\reveal^{t}$ will be generated by $\{\mathcal R_{s,\, i(s,t)}\}_{s < t}$ for some random $i(s, t)$.
At each time $t$, the process $\omega^t$ takes values in the space consisting of all
$$\omega \in \Omega \subset \{0,1,?\}^{[a]}$$ such that $\omega(j) = 0$ for $j < L(\omega),$
where
 $$L(\omega) := \inf\{i: \omega(i)=1\}.$$
In particular, if $\omega$ contains no 1, then $\omega = (0, 0, \ldots, 0)$.
The variable $\omega^{t}$ will be defined as a function of the paths $\{Q^{s}\}_{s < t}$ and measurable in $\reveal^{t}$.

To keep track of the information in $\reveal^t$, we also define a family of $0/1$-valued random variables, $\hidden(s,i,t)$, that are indexed by triples $$(s,i,t) \in \Z_{\geq 0} \times [a] \times \Z_{>0} \text{ with } s < t.$$
Our definition will ensure that a.s. $\hidden(s,i,t) = 1$ if and only if $\Xs(i)$ is not uniquely determined by $\reveal^t$, that is,
%$i \in \range(Q^s)$ and
$$\p\left(\Xs(i) = 1\mid \reveal^t\right) \notin \{0, 1\}.$$

%The inductive definition of these objects is split into cases, depending on the random walk path $Q^t$ that occurred at time $t$.

%Recall that $$\range(Q^t) := \{i: \exists\, u \text{ with } Q^t(u) = i\}.$$
%Also let $$\rangel(Q^t) := \range(Q^t) \cap (-\infty, L^t) \quad\text{and}\quad \ranger(Q^t) : = \range(Q^t) \cap (L^t, \infty).$$

Set $\omega^0 = \tilde \omega^0$ and $\reveal^0 =\sigma(\tilde \omega^0)$. Suppose we have defined the $\sigma$-algebra $\reveal^t$, the state $\omega^t$ and the family of random variables $\hidden(\cdot, \cdot, t)$ after the $t$-th random walk path $Q^{t-1}$. We start by including $\reveal^t$ and $\mathcal T(Q^t)$ in $\reveal^{t+1}$, i.e. all the previously revealed plus whether $Q^t$ is an excursion, a double-sided path, or a failed re-arrival. For any $s<t$, set $$\hidden(s, i, t+1)= 0\text{ if }\hidden(s, i, t) = 0.$$
If $Q^t$ is an excursion or a double-sided path, then $\omega^{t+1}$, $\reveal^{t+1}$ and $\hidden(\cdot, \cdot, t+1)$ are further defined via the three-step procedure:

\begin{enumerate}
	\item Refresh the range of $Q^t$ (other than $L^t$) with ?'s as follows to get $\omega_\ast^{t+1}\in \{0,1,?\}^{[a]}$:
	
	If $\max(\range(Q^t)) \neq L^t+1$ and $\min(\range(Q^t)) \neq L^t-1$, then
	\begin{equation}
		\omega_\ast^{t+1}(i) = \left\{\begin{array}{lr}
			?, & i \in \range(Q^t)\setminus\{L^t\}\\
			0, & i = L^t\\
 			\omega^t(i), &\text{otherwise}.	
 		\end{array}\right.
	\end{equation}
	Include the random variable $\range(Q^t)$ in $\reveal^{t+1}$ and set $\hidden(t, i, t+1) = 0$ for all $i \in \range(Q^t)^c \cup \{L^t\} $.
	%Note that $\Xt(L^t)=1 \in \reveal^{t+1}$ trivially.
	
	If $\max(\range(Q^t)) = L^t+1$, then we define $\omega_\ast^{t+1}$ similarly but change the symbol at $L^t+1$ to 
	\begin{equation}
		\omega_\ast^{t+1}(L^t+1)= \left\{\begin{array}{lr}
			?, & \text{if }\omega^t(L^t+1) = \,?\\
			(\omega^t(L^t+1) + 1) \bmod 2, &\text{otherwise}.	
 		\end{array}\right.
	\end{equation}
	Additionally, set $\hidden(t, L^t+1, t+1) = 0$.
	%In this case, $\Xt(L^t+1) = 1 \in \reveal^{t+1}$ trivially.
	If $\min(\range(Q^t)) = L^t-1$, then we do the similar change at $L^t-1$.
	
	\item Find the leftmost one to get $\omega^{t+1} \in \Omega$:
	\begin{equation}
		\omega^{t+1}(i) = \left\{\begin{array}{lr}
			0, & i < L^{t+1}\\
			1, & i = L^{t+1}\\
 			\omega_\ast^{t+1}(i), &i \geq L^{t+1}+2.	
 		\end{array}\right. 
	\end{equation}
	We leave the definition of $\omega^{t+1}(L^{t+1}+1)$ to the last step. Include in $\reveal^{t+1}$ all the random variables $\Xs(i)$ for $s\leq t$ and $i \leq L^{t+1}$ which have not been in $\reveal^t$. Set $\hidden(s, i, t) = 0$ for any such $s, i$.
	 
	\item Inspect the bit at $L^{t+1}+1$ to see if it is uniquely determined by $\reveal^{t+1}$. More specifically, for any $s \leq t$ such that $\hidden(s, L^{t+1}+1, t+1)$ has not been set to zero, set $\hidden(s, L^{t+1}+1, t+1)=0$ if one of the two cases from Lemma \ref{unholy} holds with $Q^s$ and $i=L^{t+1}+1$. If after this, we have $\hidden(s, L^{t+1}+1, t+1)=0$ for all $s \leq t$, then we define $\omega^{t+1}(L^{t+1}+1)= \tilde\omega^{t+1}(L^{t+1}+1)$; otherwise, let $\omega^{t+1}(L^{t+1}+1)=\,?$. Finally, we set all the undefined $\hidden(\cdot,\cdot,t+1)$ to one.
\end{enumerate}

If $Q^t$ is a failed re-arrival, then we simply leave $\omega^{t'}$, $\reveal^{t'}$ and $\hidden(\cdot,\cdot,t')$ undefined for $t'\geq t+1$.
This completes the inductive definition of all three items.

The last definition implies that all results about the auxiliary carpet process $\omega^t$,
%starting from Section \ref{sec:carpets} up until Lemma \ref{smith},
including but not limited to the key Lemmata \ref{geo}, \ref{ugly computation} and \ref{smith} below, only apply until the first occurrence of a failed re-arrival.
%bounds the probability that the corresponding event occurs before any failed re-arrival.
However, since the auxiliary process $\omega^0 = \tilde\omega^0$ may start from an arbitrary configuration, we could restart the definition of $\{\omega^{T+t}\}_{t\geq0}$ with $\omega^T = \tilde\omega^T$, at some stopping time $T$ close enough to the time in question. This is done in Lemma \ref{snack}.

\subsection{Properties of the process $\omega^t$}
%%%%%%%%%%%%%%%%%%

%First, for brevity we define two families of $0,1$ valued random variables that are indexed by triples
%$(t,i,s) \subset \Z_+ \times [a] \times \Z_+$.
%We call these $\hidden(t,i,s)$ and  $\revealed(t,i,s)$.
%If $t \leq s$ then $$\hidden(t,i,s) + \revealed(t,i,s)=\visited(t,i)$$ and  if $\visited(t,i)=0$ then
%$$\hidden(t,i,s) = \revealed(t,i,s)=0.$$
%These are defined by:

%$$\visited(t,i) = 1\{i \in \text{Range}(Q^t)\},$$

%and 
%$$\revealed(t,i,s) = 1 \text{ iff } \visited(t,i) = 1 \text{ and } \tilde \omega^{s}(i) \in \reveal^s. $$
 
We now state several consequences of these definitions before proving Lemma \ref{geo}, which summarizes our progress in Section \ref{sec:carpets}.
For $s<t$, let $$i_h(s,t) = \inf\{i: \hidden(s,i,t) = 1\}.$$
Also let $$\hidden_L(s,t) := \{i: \hidden(s,i,t)=1\} \cap (-\infty, L^t)$$
and 
$$\hidden_R(s,t) := \{i: \hidden(s,i,t)=1\} \cap (L^t, \infty).$$

%Recall that $$\range(Q^t) := \{i: \exists\, u \text{ with } Q^t(u) = i\}.$$
%Also let $$\rangel(Q^t) := \range(Q^t) \cap (-\infty, L^t) \quad\text{and}\quad \ranger(Q^t) : = \range(Q^t) \cap (L^t, \infty).$$

\begin{lemma}\label{graduation} The following are true.
\begin{enumerate}
\item $\omega^t(i)=\tilde\omega^t(i)$ if $\omega^t(i)\neq\, ?$. \label{omega:0/1}
\item $L(\omega^t)= L(\tilde \omega^t)$. \label{omega:L}
\item $\omega^t(i) = \, ?$ if and only if there exists some $s<t$ such that $\hidden(s,i,t)=1.$ \label{omega:?}
\end{enumerate}
\end{lemma}

\begin{proof}
All items follow from the definition of $\omega_\ast^t$ and $\omega^t$ by induction on $t$. 
\end{proof}
 
\begin{lemma} \label{party}
The following are true.
\begin{enumerate}
\item $\omega^t$ is measurable in $\revealtit.$
\item $\revealtit$ is generated by $\{\mathcal R_{t,\, i_h(s,t)}\}_{s<t}$ where $R_{t,\, i}$ is defined in Lemma \ref{unholy}. \label{revealR}
\item For any $s<t$, both $\hidden_L(s,t)$ and $\hidden_R(s,t)$ are either empty or a connected set of size at least two. \label{hidden2}
%$$|\hidden_L(s,t)| \neq 1 \text{ and }|\hidden_R(s,t)| \neq 1.$$
%\item Suppose $\{i, i+1\} \subseteq \hidden_{L}(s_0,t)$ or $\hidden_R(s_0,t)$ for some $s_0$, then a.s. $$\p\left(\tilde\omega^t(i) = 1 \mid \revealtit, \{\mathcal R_{s,i}\}_{s<t}\right) \in (1/6,5/6).$$
\end{enumerate} 
\end{lemma}
 
\begin{proof}
The proof of each item goes by induction. In particular, for item \eqref{omega:0/1} and \eqref{omega:L} we use Lemma \ref{unholy} in the `inspect' step and trivially in the two corner cases of the `refresh' step.
%The definition of $\omega^t$ involved certain random variables from $\{Q^{t'}\}_{t'=1}^{t-1}$. All of those random variables are included in $\reveal^t$.
Item \eqref{omega:?} is guaranteed by the definition of $\hidden$, both from the two corner cases in the `refresh' step and from the cases of Lemma \ref{unholy} in the `inspect' step.
\end{proof}

For the purpose of Lemma \ref{geo}, it is more convenient to study the intermediate state $\omega_\ast^t$ defined in the `refresh' step of the procedure. We consider the corresponding $\reveal_\ast^t$ and $\hidden_\ast(\cdot, \cdot, t)$ associated with $\omega_\ast^t$. Let $\reveal_\ast^{t+1}$ (resp. $\hidden_\ast(\cdot, \cdot, t+1)$\,) be $\reveal^{t+1}$ (resp. $\hidden(\cdot,\cdot,t+1)\,$) at the end of the `refresh' step without performing the changes in the subsequent two steps. All undefined $\hidden_\ast(\cdot, \cdot, t+1)$ are set to one. We also define $i_{h,\ast}(s,t)$, $\hidden_{L,\ast}(s,t)$ and $\hidden_{R,\ast}(s,t)$ by replacing $\hidden(s,i,t)$ in the original definitions with $\hidden_\ast(s,i,t)$. The next lemma follows as a by-product of the inductive proof of the last two lemmata.

\begin{lemma}\label{foggy}
Except for Lemma \ref{graduation} \eqref{omega:L}, Lemmata \ref{graduation} and \ref{party} still hold if we replace $\omega^t$, $\revealtit$ and $\hidden(\cdot,\cdot,t)$ with the corresponding $\omega_\ast^t$, $\reveal_\ast^{\,t}$ and $\hidden_\ast(\cdot,\cdot,t)$. \qed
\end{lemma}

\begin{lemma}\label{cost_ast}
If $\omega_\ast^t(i)=\,?$, then a.s. $$\p\left(\tilde\omega^t(i) = 1 \mid \revealtit_\ast, \{\mathcal R_{s,i-1}\}_{s<t}\right) \in (1/6,5/6).$$
\end{lemma}

\begin{proof}
In the following, when we refer to the items in Lemmata \ref{graduation} and \ref{party}, we actually mean their counterparts in Lemma \ref{foggy}.
Since $\omega_\ast^t(i)=\,?$, by Lemma \ref{graduation} \eqref{omega:?} there exists some $s_0<t$ with $\hidden_\ast(s_0,i,t) = 1$. By Lemma \ref{party} \eqref{hidden2}, we must have
\begin{equation}\label{supply chain}
\{j, j+1\} \subseteq \hidden_{L,\ast}(s_0,t) \cup \hidden_{R,\ast}(s_0,t)	
\end{equation}
holds for either (1) $j=i$, or (2) $j=i-1$ with $i = \max(\hidden_{L,\ast}(s_0,t))$ or $i=\max(\hidden_{R,\ast}(s_0,t))$.

In the first case, neither of the events in Lemma \ref{unholy} occurs with $Q^{s_0}$ at $i$. So by combining Lemma \ref{party} \eqref{revealR}, the fact that $i_{h,\ast}(s_0,t) \leq i$, independence and Lemma \ref{unholy}, we obtain that a.s.
$$\p\left(\Xszero(i) = 1 \mid \reveal_\ast^t, \{\mathcal R_{s,i}\}_{s<t} \right) \in (1/6,5/6).$$
Since $\tilde\omega^t(i) = \tilde\omega^0(i)+ \sum_{s<t} \Xs(i)$, this proves Lemma \ref{cost_ast} in the first case.

For the second case, the display from the first case can be shown similarly if we replace both $i$'s with $i-1$'s. The assumption also implies $\Xszero(i)$ is uniquely determined by $\mathcal R_{s_0,i}$, so a.s.
$$\p\left(\Xszero(i) = 1 \mid \reveal_\ast^t, \{\mathcal R_{s,i-1}\}_{s<t} \right) \in (1/6,5/6).$$
Lemma \ref{cost_ast} then follows.
\end{proof}

For $\omega \in \{0,1,?\}^{[a]}$, define
$$N(\omega) := \#\{i\in[0,a]: \omega(i)\neq 0\}.$$ 
One key part of our analysis is to show that when $L(\omega^t)$ is large, $N(\omega^t)$ behaves like a random walk biased to increase, see Lemma \ref{nice computation}. Lemma \ref{geo} says that $N(\omega^t)$ is unlikely to decrease dramatically by erasing many ?s.
%This will ensure that it takes a long time until $N(\omega^t)=0$.
  
 \begin{lemma} \label{geo}
For any $t,k \in \N$, a.s.
$$\p\left(N(\omega_\ast^{t})-N(\omega^{t}) \geq k \ | \ \revealtit_\ast \right) \leq (5/6)^{(k-1)/2}.$$ 
 \end{lemma}

%$$\p\left(N(\omega^{t})-N(\omega^{t+1}) \geq k \ | \ \revealtit_\ast \right) \leq (5/6)^{(k-3)/2}.$$
\begin{proof}
%By the definition of the `refresh' step, we have $N(\omega^t) - N(\omega_\ast^{t+1}) \leq 2$. So it suffices to bound $$\p\left(N(\omega_\ast^{t+1})-N(\omega^{t+1}) \geq k \ | \ \reveal^t_\ast \right) \leq (5/6)^{(k-1)/2}.$$
Let $i_n$ be the location of the $n$th leftmost ? in $\omega_\ast^{t}$ (if it exists). Since we might lose one more ? in the `inspect' step, the above conditional probability is at most
$$\p\left(\tilde\omega^t(i_1) = 0, \tilde\omega^t(i_2)=0,\dots,\tilde\omega^t(i_{k-1}) = 0 \ | \ \reveal^t_\ast \right).$$
%and thus by
%$$\p\left(\tilde\omega(i_1) = 0, \tilde\omega(i_3)=0,\dots,\tilde\omega(i_{2\lceil(k-1)/2\rceil-1}) = 0 \ | \ \reveal^t_\ast \right).$$
The above expression is at most $(5/6)^{\lceil(k-1)/2\rceil}$ by repeated conditioning and use of Lemma \ref{cost_ast} at every other $i_n$. This proves the desired estimate.
\end{proof}

\section{Zeros of auxiliary  carpet process}\label{sec:ugly}
%%%%%%%%%%%%%%%%%%%%%%%%%%%%%%%

%\subsection{Key lemma} 

With Lemma \ref{geo}, we will show that when $L(\omega^t)$ is large, the number of nonzero symbols $N(\omega^t)$ has a bias to increase. However, this does not rule out the possibility that $N(\omega^t)$ becomes small when $L(\omega^t)$ is small.
%which could lead to a huge leap to the right going from $L(\omega^t)$ to $L(\omega^{t+1})$.
The main goal of Section \ref{sec:ugly} is to show that with a nice initial configuration, this is unlikely to occur for an exponentially long time.

To state Lemma \ref{ugly computation}, the main result of this section, we fix a few notations. First we introduce several special states in the space $\Omega$, defined in Section \ref{sec:aux}.
Let $$\text{ Base}=01??\cdots??$$
$$\;\,\text{ Exit}=0000\cdots00.$$
The `Base' state starts with $01$ followed by $a-1$ many ?, whereas the `Exit' state is the all-zero state which occurs exactly when the block becomes frozen. Define
$$\tau_{\text{Exit}} :=\inf\{t\geq0: \omega^t=\text{Exit}\}.$$
For $\epsilon = 1/200$, define an `$\epsilon$-Base' state to be any configuration of the form $$00\cdots01??\cdots??$$
starting with \textit{at most} $\epsilon a$ many 0's from the left. In particular, the Base state is an $\epsilon$-Base state.

%
%and
%
%$$\tau_{\text{Emit}} = \inf\{t': \ell(Q_t) \geq K \text{ or } \ell(Q_t) \leq -K+a\},$$

Since $\omega^t$ is not a Markov chain with respect to its natural filtration, we consider the carpet processes $\{\tilde\omega^t\}_{t\geq t_0}$ and $\{\omega^t\}_{t\geq t_0}$ that start from some negative $t_0$ instead of zero, and use the shorthand notation $\overline{\p}_{\omega_0}(A) \leq x$ to mean that
$$\sup_{R\, \in\, \reveal^0} \p\left(A \ | \ \omega^0 = \omega_0, R \right) \leq x.$$
Write $\underline{\p}$ similarly for the infimum. In fact, all results in Section \ref{sec:ugly} are true in a stronger sense where the supremum/infimum is over all $R \in \sigma(\{Q^s\}_{s < 0})$, but the above notation has the advantage of working for both Sections \ref{sec:ugly} and \ref{sec:renewal}. We also write $\overline{\p}_{\epsilon\text{-Base}}$ in the case that the statement holds for any initial $\epsilon$-Base state $\omega_0$.

%We are ready to state the main result of this section.

\begin{lemma} \label{ugly computation}
Consider the event 
$$\key = \left\{L(\omega^t) \leq \epsilon a \text{ and } N(\omega^t) \leq \epsilon a \text{ for some } 0\leq t< \tau_{\text{Exit}} \right\}.$$
For $\epsilon = 1/200$ and sufficiently large $a$, 
$$\overline{\p}_{\epsilon\text{-Base}}(B) \leq \exp(-a/100).$$
\end{lemma}

The rest of Section \ref{sec:ugly} is devoted to the proof of Lemma \ref{ugly computation}. In order to produce so many zeros in $\omega^t$, the auxiliary carpet process needs to follow certain scheme -- it has to work from right to left and generate the zeros without refreshing the pre-existing zeros to its right. The proof then goes by identifying such schemes and bounding each scheme's probability, via estimates mimicking those for birth-death chains.

%%%%%%%%%%%%%%
\subsection{History of zeros in $\omega^t$}
%%%%%%%%%%%%%%

We start by studying how zeros are generated in the auxiliary carpet process. Write
$$\visited(t,i) = 1\{i \in \range(Q^t)\}.$$
Recall that the process $\omega^t$ starts from some negative time $t_0$. For any time $t \geq 0$, define 
 $$\rightzeros^t := \{i>L(\omega^t):\ \omega^t(i)=0\}.$$
For any time $t\geq0$ and $i \in [L(\omega^t),a]$, define 
$$\last^t(i) := \sup \{t' \in [0,t) : \visited(t',i)=1 \} \vee -1,$$
which is the last \textit{nonnegative} time $t'$ before $t$ when $i$ gets visited by a random walk path. 

\begin{lemma} \label{fromage mini}
For any $s$ and $i$ with $L(\omega^{s+1}) < i \le L(\omega^s)$, we have $\visited(s,i)=1$
\end{lemma} 
\begin{proof}
$Q^s$ cannot be an excursion to the right, since necessarily $L(\omega^{s+1}) > L(\omega^s)$ in that case. If $Q^s$ is a double-side path, then $\visited(s,i) = 1$ for all $i$. So assume $Q^s$ is an excursion to the left. Then
$L(\omega^{s+1}) \geq \min(\range(Q^s))$ and thus $\min(\range(Q^s)) < i \le L(\omega^{s})$. This implies $\visited(s,i)=1$.
\end{proof} 
 
\begin{lemma} \label{quejo}
For any $t\geq0$, $i \in \rightzeros^t$ and $ t' \in [\last^t(i)+1, t]$, $$ L(\omega^{t'})<i.$$
\end{lemma} 
 
 \begin{proof}
Suppose that $L(\omega^{t'})\geq i$. Since $i \in \rightzeros^t$, we have $L(\omega^t) < i \le L(\omega^{t'})$. Thus there exists some $s$ such that $t' \leq s <t$ with  $L(\omega^{s+1}) < i \le L(\omega^{s})$. But by Lemma \ref{fromage mini} we must have a time $s \in (\last^t(i),t)$ with $\visited(s,i)=1$, which contradicts the definition of $\last^t(i)$. Therefore $L(\omega^{t'})< i$.
 \end{proof}
 
\begin{lemma}\label{boutique}
Suppose $L(\omega^s) \notin \{i,i+1\} \subseteq \range(Q^s)$, and $L(\omega^{t'}) < i$ for any $t' \in [s+1,t]$, then $$\omega^t(i) = \omega^t(i+1) =\, ?.$$
%In particular, the conclusion holds when $L(\omega^s) \notin \{i,i+1\} \subseteq \range(Q^s)$ for $s = \last^t(i)$.
%$$\hidden(s,i,t) = \hidden(s,i+1,t) = 1.$$
\end{lemma}

\begin{proof}
Recall the definition of the procedure in Section \ref{sec:aux}. By the first assumption, $\hidden(s,i',s+1)$ has not been set to zero at the end of the `refresh' step for $i'\in\{i,i+1\}$. By the second assumption, $\hidden(s,i',t')$ will remain to be one for $i'\in\{i,i+1\}$ and $t' \in [s+1,t]$. So $\hidden(s,i,t) = \hidden(s,i+1,t) = 1$ and the lemma follows by Lemma \ref{graduation} \eqref{omega:?}.
\end{proof}

\begin{lemma} \label{pao}
Let $i_{\max} := \max\{i > L(\omega^t):\omega^t(i)=0\}.$
We have
\begin{enumerate}
\item the function $\last^t(i)$ is decreasing on the set $[L(\omega^t),a]$; \label{monotone}
\item for any $i,j \in \rightzeros^t$ with $j < i$ and $\last^t(j)=\last^t(i)\neq\last^t(i_{\max})$, we have $i=j+1$ and at time $\last^t(i)$, there was an 
excursion to its left $Q^s=Q^{\last^t(i)}$ starting from $i$ such that $\visited(s,i-2)=1$ and $\parity_{Q^s}(i-1)=0$; \label{adjacent}
\item there do not exist three distinct values $i,j,k \in \rightzeros^t$ with 
   $\last^t(i) = \last^t(j) = \last^t(k) \neq \last^t(i_{\max})$. \label{nothree}
\item if $\omega^0$ is an $\epsilon$-Base state, then there do not exist four distinct values in $\rightzeros^t$ with the same function value $\last^t(\cdot)$.
\label{nofour}
   \end{enumerate}
\end{lemma}
 
\begin{proof}
Since the random walk always starts from the leftmost one, a straightforward induction on the time $t$ proves item \eqref{monotone}. Item \eqref{nothree} follows directly from item \eqref{adjacent}, so it remains to show items \eqref{adjacent} and \eqref{nofour}.

For item \eqref{adjacent}, write $s = \last^t(j) = \last^t(i)$. By item \eqref{monotone}, we have $s > \last^t(i_{\max}) \geq -1$, so $s \geq 0$ and $\visited(s,j) = \visited(s,i)=1$. Thus $\{j, j+1\} \subseteq \range(Q^s)$. We claim that $L(\omega^s) \in \{j, j+1\}$. Suppose not, then Lemmata \ref{quejo} and \ref{boutique} combined imply that $\omega^t(j)=\,?$, which is a contradiction. This proves the claim.

We consider all the possible cases of $Q^s$ satisfying $L(\omega^s) \in \{j, j+1\}$. First, $Q^s$ cannot be a double-sided path: otherwise, we would have $\last^t(i) = \last^t(i_{\max})$. Secondly, $Q^s$ cannot be an excursion to the right: otherwise, from $L(\omega^s) \in \{j, j+1\}$ we must have $L(\omega^{s+1}) > j$, which contradicts Lemma \ref{quejo}. Lastly, if $Q^s$ is an excursion to the left, then we must have $L(\omega^s) = j+1$ to guarantee $\visited(s, j+1) = 1$. This implies $i=j+1$. Since no random walk after time $s$ visits $j$, it follows that $\parity_{Q^{s}}(j) = \tilde\omega^s(j) = \omega^s(j)=0$, which in turn implies that $\visited(s,i-2)=1$. This proves item \eqref{adjacent}.

For item \eqref{nofour}, we first show that for any $s \geq 0$, $|\rightzeros^t(s)| \leq 3$, where
$$\rightzeros^t(s) := \{i \in \rightzeros^t: \last^t(i) = s\}.$$
%there do not exist four distinct values in $\rightzeros^t$ with the function value $\last^t(\cdot)$ equal to $s$.
Suppose not, then out of any four such elements we can find $i \in \rightzeros^t(s)$ such that $L(\omega^s) \notin \{i, i+1\} \subseteq \range(Q^s)$, which leads to a similar contradiction to that in the proof of item \eqref{adjacent}. This proves the bound for $s\geq0$.

It remains to check the case where $s=-1$. We show that $|\rightzeros^t(-1)| \leq 1$ again by contradiction. Suppose not and there exist $i, j \in \rightzeros^t(-1)$ with $i<j$. Since $\last^t(j)=-1$, by item \eqref{monotone} we have $j > L(\omega^0)$ and thus $\omega^0(j) =\,?$ due to the definition of an $\epsilon$-Base state. Moreover, by Lemma \ref{quejo} we get $L(\omega^{t'}) < i$ for any $t' \in [0,t]$, so it follows from the definition of $\omega^t$ that $\omega^{t'}(j) =\,?$ for any $t' \in [0,t]$. This contradicts with the fact that $\omega^t(j)=0$, which proves item \eqref{nofour}.
%By item \eqref{monotone}, if $\rightzeros^t(-1)$ is nonempty, then $i_{\max}$ must belong to it. We claim that in this case, $\rightzeros^t(-1) = \{i_{\max}\}$. Suppose not and there exists another $i \in \rightzeros^t(-1)$ 
\end{proof}

%{\bf Claim 1:} $L(\omega^{\last^t(i)+1}) < j$.

%This is a direct consequence of Lemma \ref{quejo} and the fact that $\last^t(i) = \last^t(j)$.

%{\bf Claim 2:} It is impossible that $\omega^{\last^t(j)+1}(j) = \omega^{\last^t(j)+1}(j+1) = ?$

%Suppose not. Since Lemma \ref{quejo} says that $L(\omega^{t'}) < j$ for any $ t' \in [\last^t(j)+1, t]$, equation \eqref{rasta} implies $$\hidden(\last^t(j),j,t')=\hidden(\last^t(j),j+1,t')=1$$ for any such $ t'$. This implies $\omega^t (j) = ? \neq 0$, which is a contradiction.
%\\

%By examining all possibilities in equations \eqref{redemption}, \eqref{song}, \eqref{rasta} and \eqref{pasta}, one finds that in order for both Claim 1 and 2 to be satisfied, the excursion $Q^{\last^t(i)}$ has to be one starting from $L(Q^{\last^t(i)})=i$ to its left with $\hat \ell \geq 2$. Also we must have $j=i-1$. It follows that $\parity_{Q^{\last^t(i)}}(i-1) = 0$ because $\omega^t(i-1) = 0$ and no excursion after $\last^t(i)$ visits $i-1$.

%\subsection{Good Sequences}

%Motivated by Lemma \ref{pao} \eqref{adjacent} and \eqref{nothree}, we introduce the notion of a \textit{good} sequence, which captures how zeros may be generated in $\omega^t$, as well as some more notation.
We introduce some more notation.
For $p<a$ write $[p,a] =[p,a] \cap \Z.$
We call a sequence $x \in \{0,*,\qm\}^{[p,a]}$ {\bf good}
if for all $j$ such that $x(j)=*$ we have  $x(j+1)=0$. The good sequences capture different ways in which zeros may be generated in $\omega^t$, with adjacent $*$ and $0$ representing a pair of zeros generated as described in Lemma \ref{pao} \eqref{adjacent}.

%We classify the zeros in a good sequence depending on whether they are generated in pairs and their distance from the next zero to the right. These factors lead to different bounds.
Given a good $x$ and $p < a$, we partition the zeros in $x$ as follows:
$$\setzero:=\{j \in [p,a]:x(j)=0\}=\setone \cup \settwo \cup \setthree \cup \{\tilde M(x)\},$$
where
\begin{eqnarray*}
\tilde M(x) &:=& \max \{j: \ x(j)=0\}\\
\tilde N(x) &:=& \max \{j < \tilde M(x): \ x(j)=0\}
\end{eqnarray*}
and
\begin{eqnarray*}
\setone &:=& \{ j \in [p,a]: x(j)=0, x(j+1) \neq\, \qm \text{ and }x(j-1)\neq*\} \ \setminus \{\tilde N(x)\}\\
\settwo &:=& \{ j \in [p,a]: x(j)=0, x(j+1) \neq\, \qm \text{ and }x(j-1)=*\} \ \setminus \{\tilde N(x)\}\\
\setthree &:=& \{ j \in [p,a]: x(j) =0, x(j+1) =\,\qm\}\cup\{\tilde N(x)\}\setminus\{\tilde M(x)\}.
\end{eqnarray*}

%Let $$\tilde M(x):=\max \{j: \ x(j)=0\}$$
%and $$\tilde N(x):=\max \{j < \tilde M(x): \ x(j)=0\}.$$
%Let $$\setone:=\{ j \in [p,a]: x(j)=0, x(j+1) \neq\, \qm \text{ and }x(j-1)\neq*\} \ \setminus \{\tilde N(x)\}.$$
%$$\settwo:=\{ j \in [p,a]: x(j)=0, x(j+1) \neq\, \qm \text{ and }x(j-1)=*\} \ \setminus \{\tilde N(x)\}.$$
%$$\setthree:=\{ j \in [p,a]: x(j) =0, x(j+1) =\,\qm\}\cup\{\tilde N(x)\}\setminus\{\tilde M(x)\}.$$

For $j \in \setzero \setminus \{ \tilde M(x) \}$, let $k(j)$ be such that
$$j+k(j)=\inf\{j'>j: j' \in \setzero\}$$
and let $r(j)$ be such that
$$j+r(j)=\inf\{j'>j: x(j')=0 \text{ or }*\},$$
except when $j=\tilde N(x)$, we have $r(\tilde N(x)) = \tilde M(x) + 1 - \tilde N(x)$ instead.
Note that in a good sequence, for $j \neq \tilde N(x)$ we always have $r(j)=k(j) \text{ or } k(j)-1$.

Finally, for $r \geq 2$ divide $\setthree$ into pieces 
$$\setthreek:=\setthree\cap\{j:\ r(j)=r\}.$$

\subsection{Counters, stopping times and events}\label{count von count}

%\subsection{Definition of the counters and the stopping times} 

In this section, we outline the proof of Lemma \ref{ugly computation}.
Suppose we're given a good sequence $x$. To bound the probability of the right-to-left dynamics generating $x$, we will inductively define a counter $y^s(i)$ at every vertex $i \in \setzero$ starting from $\tilde M(x)$.
%We will do this using a sequence $z \in \N^{\setzero}$ and some random variables that we call counters.
The goal of the counters is twofold. On one hand, given a sequence $z \in \N^{\setzero}$ the counters can be used to identify a sequence of stopping times $\tilde T(i)$ for every $i \in \setzero$. We will tailor our definitions so that $\tilde T(i) = \last^t(i) + 1$ with the right choice of $z$ (in most cases), see Lemma \ref{stopping time}. On the other hand, the counters also give us bounds on the probability of the key events $A_{x,z}$, see Lemma \ref{healthy}.

We start with the definition of $y^s(i)$ and $\tilde T(i)$ at $i = \tilde M(x)$.

\begin{itemize}
\item For $i =\tilde M(x)$, the counter $y^0(i)$ starts at zero. For any $s \geq 0$, the counter $y^{s+1}(i)=y^s(i)+1$ increases by one if at time $s$ the random walk starts from $L(\omega^s)\geq \tilde M(x)$ and either $s=0$ or the previous random walk path $Q^{s-1}$ was from $L(\omega^{s-1}) < \tilde M(x)$; otherwise, the counter stays put and we have $y^{s+1}(i)=y^s(i)$. Given $z(i)$, define the stopping time $\tilde T(\tilde M(x))$ to be the smallest $s$ such that $y^s(i)=z(i)$ and the next path starts at $L(\omega^s) < \tilde M(x)$ if such $s$ exists; in this case, we say $\tilde T(\tilde M(x))$ is well-defined. 
\end{itemize}

In order to define other counters we will work inductively. Suppose that the stopping time $\tilde T(i+k(i))$ is well-defined, we shall define $y^s(i)$ for all $s \geq \tilde T(i+k(i))$.
\begin{itemize}
\item For $i \in \setone \cup \settwo$, then initially at $s=\tilde T(i+k(i))$, we set the counter $y^s(i)=\tilde\omega^s(i)-1$. For any $s \geq \tilde T(i+k(i))$, let $$y^{s+1}(i)=y^s(i)+\local_{Q^s}(i).$$ In words, the counter increases by one every time a particle moves from $i$.
\item For $i \in \setthree$, then the counter $y^{\tilde T(i+k(i))}(i)$ starts at zero and for $s \geq \tilde T(i+k(i))$, we have
$$y^{s+1}(i)=\begin{cases}
y^s(i)+1,&  L(\omega^s)\in [i,i+r(i))\\
y^s(i)+\local_{Q^s}(i),&  L(\omega^s)< i.
\end{cases}
$$ In words, the counter increases by one every time a random walk path starts at a location belonging to $[i,i+r(i))$ or every time a particle moves from $i$ in a path that starts to the left of $i$.
\end{itemize}

Given $z(i)$, let $\tilde T(i)$ be the smallest $s$ such that $y^s(i)=z(i)$ if such $s$ exists; in this case, we say $\tilde T(i)$ is well-defined. This completes the definition of the counter $y^s(i)$ and stopping time $\tilde T(i)$ at every $i \in \setzero$.

%\subsection{The Key Events}
Finally, we define the key events $A_{x,z}$ in the analysis of counters. For a good sequence $x$ and $z \in \N^{\setzero}$, define $A_{x,z}$ to be the event that

\begin{enumerate}
\item the stopping times $\tilde T(i)$ are well-defined for all $i \in \setzero$; \label{well-defined}
\item $\tilde T(\tilde M(x)) < \tau_{\text{Exit}}$; \label{no exit}
\item  for any $i \in \setzero \setminus \{\tilde M(x)\}$, none of  $Q^s$ visits $i+r(i)$ during $\tilde T(i+k(i)) \le s < \tilde T(i)$;\label{no visit}
\item for $i \in \settwo$, just before $\tilde T(i)$ the particle made an excursion to the left $Q^s = Q^{\tilde T(i)-1}$ starting from $i$ where $\visited(s,i-2)=1$ and $\parity_{Q^s}(i-1)=0$. \label{two steps}
\end{enumerate}

%\begin{lemma} \label{sprechen}
%If $A_{x.z}$ occurs then for any $p \in \setone \cup \settwo \cup \setthree \cup \tilde M(x)$  no excursion visits $p$ after $\tilde T(p)$.
%\end{lemma}

%\begin{proof}
%At the stopping time $\tilde T(p)$ the counter $y(p)$ is at $z(p)$. If $p \in \setone \cup \settwo \cup \setthree$ then any visit to $p$ causes an increase in $y(p)$ so $y(p)>z(p)$ and the realization is not in $A_{x,z}$. For $p=\tilde M(x)$ and $p \in\setthree$ a visit to $p$ causes a change in the sequence \{$\omega^s(p)\}^{s \in \N}$. That change forces an increase in the counter $y(p)$ and $y(p)>z(p)$ and $A_{x,z}$ does not occur.
%\end{proof}

\begin{lemma} \label{union}
If the event $\key$ in Lemma \ref{ugly computation} occurs, then there exists 
\begin{itemize}
\item $p\leq\epsilon a$, 
\item a good
$x \in \{0,*,\qm\}^{[p,a]}$ and 
\item $z \in \N^{\setzero}$ 
\end{itemize}
such that 
\begin{itemize}
\item the event $A_{x,z}$ occurs and 
\item all the $z(i)$'s are odd for $i \in \setone \cup \settwo$.
\end{itemize}
\end{lemma}

\begin{lemma} \label{healthy}
Fix any $p\leq a$.
Fix any good sequence $x \in \{0,*,\qm\}^{[p,  a]}$.
Fix any sequence $z \in \N^{\setzero}$. We have
\begin{eqnarray*}\label{spacey}
\overline{\p}_{\epsilon\text{-Base}}( A_{x,z}) \leq\left(1-p_{\frac{1}{2}}^{\,-\tilde M(x)+a+1}\right)^{z(\tilde M(x))}
\prod_{i \in        \setone}\left(\frac{1}{2}\right)^{z(i)}
\prod_{i \in  \settwo} \frac{1}{6}\left(\frac{1}{2}\right)^{z(i)-1}
%\\
\prod_{r>1}\prod_{i \in \setthreek}\left(1-\frac{1}{2r}\right)^{z(i)},
\end{eqnarray*}
where $p_{\frac{1}{2}} = \frac{1}{2} - \frac{1}{a^4}$.
%and $p_{\frac{1}{6}} = \frac{1}{6} + \frac{1}{a^4}$.
\end{lemma}

In Section \ref{sec:ugly_proof}, we will first prove the above two lemmata, and then combine them to prove Lemma \ref{ugly computation} using a union bound.

%%%%%%%%%%%%%%%%
\subsection{Analysis of zero generation in $\omega^t$}  \label{sec:ugly_proof}
%%%%%%%%%%%%%%%%

%The proof of Lemma \ref{ugly computation} consists of three steps. 
%In Lemma \ref{union} we show that $B$ is contained in the union of the events $A_{x,z}$.
%In Lemma \ref{healthy} we bound the probability of the events $A_{x,z}$.
%Finally we combine these two results using the union bound to prove Lemma \ref{ugly computation}.

%\subsection{Proof of Lemma \ref{union}}

We start by proving Lemma \ref{union}.
Assume the event $\key$ from Lemma \ref{ugly computation} occurs, that is, $L(\omega^t) \leq \epsilon a \text{ and } N(\omega^t) \leq \epsilon a \text{ for some } 0\leq t< \tau_{\text{Exit}}$.
%Then there exists $t<\tau_{\text{Exit}}$ such that $L(\omega^{t}) \leq \epsilon a$ and $N(\omega^t) \leq \epsilon a$.
We define the corresponding $p$, $x$ and $z$ as follows. Recall $i_{\max} = \max\{i>L(\omega^t): \omega^t(i)=0\}.$
\begin{itemize}
\item Let $p:=L(\omega^t) \leq \epsilon a$. 
\item Set $\tilde M := \min\{i \in \rightzeros^t: \last^t(i) = \last^t(i_{\max})\}.$
\item Define a sequence $x$ in $\{0,*,\qm\}^{[p,a]}$ by first setting $x(\tilde M):=0$ and $x(i) := \,\qm$ for $i > \tilde M$. For $i < \tilde M$, let
\[x(i):=\begin{cases}
\qm & \text{if } \omega^t(i)=\,\qm,\\
* & \text{if } \omega^t(i)=  \omega^t(i+1)=0 \ \text{and} \ \last^t(i) = \last^t(i+1), \\
0 & \text{otherwise}.
\end{cases}
\]
%\[x(i):=\begin{cases}
%\qm & \text{if } \omega^t(i) \in \{?,1 \},\\
%* & \text{if } \omega^t(i)=  \omega^t(i+1)=0 \ \text{and} \ \last^t(i) = \last^t(i+1), \\
%0 & \text{otherwise.}
%\end{cases}
%\]
Note that $x$ is good due to Lemma \ref{pao} \eqref{nothree}.
\item For $i = \tilde M(x)$ let $z(\tilde M(x)) := y^t(\tilde M(x))$. Note that $\tilde T(\tilde M(x))$ is well-defined by our definition of $z(\tilde M(x))$ and the fact that $L(\omega^t) < \tilde M(x)$. To define $z(i)$ for $i \in \setzero \setminus \{\tilde M(x)\}$ we will work inductively. Suppose that $z(i+k(i))$ is given and $\tilde T(i+k(i))$ is well-defined. Then we let $z(i) := y^t(i)$. Again $\tilde T(i)$ is well-defined by our choice of $z(i)$. This completes the definition of $z$.%and also checks item \ref{well-defined} in the definition of $A_{x,z}$. 
\end{itemize}

\begin{lemma}\label{stopping time}
Let $p$, $x$ and $z$ be defined as above. For $i = \tilde M(x)$, we have $\last^t(i+1) + 1 \le \tilde T(i) \le \last^t(i) + 1$. For any $i \in \setzero \setminus \{\tilde M(x)\}$, we have $\tilde T(i) = \last^t(i) + 1$.
\end{lemma}

\begin{proof}
	We shall prove Lemma \ref{stopping time} by induction on $i$. We start with the base case where $i = \tilde M(x)$. If $z(i) = y^t(i) = 0$, then $\tilde T(i) = 0$ and the upper bound on $\tilde T(i)$ is trivial. If $y^t(i) > 0$, then by the definition of $\tilde T(\tilde M(x))$ we have $L(\omega^{\tilde T(i)}) < i \le L(\omega^{\tilde T(i)-1})$. Lemma \ref{fromage mini} implies $\visited(\tilde T(i)-1, i) = 1$ and thus $\tilde T(i)-1 \le \last^t(i)$. This proves the upper bound.
	
	To show $\tilde T(i) \geq \last^t(i+1) + 1$ for $i = \tilde M(x)$, we argue by contradiction. Suppose for some $s \geq \tilde T(\tilde M(x))$, the path $Q^{s}$ visits $\tilde M(x)+1$. Since $s \geq \tilde T(\tilde M(x))$, by definition $L(\omega^{t'}) < \tilde M(x)$ for any $t'\in[s,t]$, so the conditions of Lemma \ref{boutique} are satisfied.
	%equation \eqref{rasta} implies $$\hidden(t', \tilde M(x), s)=\hidden(t', \tilde M(x)+1, s)=1$$ for any such $s$.
	This implies $\omega^t (\tilde M(x)) = \,?$, which is a contradiction. This completes the proof of the base case.
	
	Now suppose $\tilde T(i+k(i)) \le \last^t(i+k(i)) + 1$ holds for some $i \in \setzero \setminus \{\tilde M(x)\}$, we will prove $\tilde T(i) = \last^t(i) + 1$. First, note that $\last^t(i) \neq \last^t(i+k(i))$; otherwise, Lemma \ref{pao} \eqref{adjacent} would imply $k(i) = 1$ and $x(i) = *$, which contradicts $i \in \setzero$. So by Lemma \ref{pao} \eqref{monotone} and the induction hypothesis we have $\last^t(i) > \last^t(i+k(i)) \geq \tilde T(i+k(i)) - 1$. Thus $\last^t(i) \geq \tilde T(i+k(i))$ and it makes sense to talk about $y^{\last^t(i)}$.
	
	In order to prove $\tilde T(i) = \last^t(i) + 1$, it suffices to check
	\begin{equation} \label{increase}
		y^{\last^t(i)+1} > y^{\last^t(i)}
	\end{equation}
	and for any $t' \in [\last^t(i)+1, t)$,
	\begin{equation}
		\label{constant}
		y^{t'+1}(i) = y^{t'}(i).
	\end{equation}
	Since $Q^{\last^t(i)}$ visits $i$, the inequality \eqref{increase} follows from the definition of counter directly if $i \in \setone \cup \settwo$ or $i \in \setthree$ and $L(\omega^{\last^t(i)}) < i+r(i)$. The remaining case that $i \in \setthree$ and $L(\omega^{\last^t(i)}) \geq i+r(i)$ cannot happen due to the fact that $\last^t(i) > \last^t(i+r(i))$ and Lemma \ref{quejo}.
	For \eqref{constant}, $Q^{t'}$ does not visit $i$ for $t' \in [\last^t(i)+1, t)$, so it follows that $y^{t'+1}(i) = y^{t'}(i)$ for any such $t'$. 
	%Since Lemma \ref{quejo} implies that $L(\omega^{t'}) < i$, we have $y^{t'+1}(i) = y^{t'}(i) + \local_{Q^{t'}}(i) = y^{t'}(i)$, where in the second equality we used the fact that the excursion $Q^{t'}$ does not visit $i$ for $t' > \last^t(i)$.
\end{proof}

\begin{proof}[Proof of Lemma \ref{union}]
We finish the proof by checking all requirements on $x$ and $z$. We've checked that $x$ is a good sequence.
%First, the fact that $x$ is good follows from Lemma \ref{pao} \eqref{nothree} directly. 

We check that $z(i)=y^t(i) \geq 0$ for all $i \in \setone \cup \settwo$. In the proof of Lemma \ref{stopping time}, we've shown $\last^t(i) \geq \tilde T(i+k(i))$,
%for any $i \in \setzero \setminus \{\tilde M(x)\}$. Thus for $i \in \setone \cup \settwo$ where $k(i)=1$, we conclude 
so $y^t(i) \geq y^{\last^t(i)+1} \geq y^{\last^t(i)} + 1 \geq y^{\tilde T(i+k(i))} + 1 \geq 0$.

To see why $z(i)=y^t(i)$ is odd for $i \in \setone \cup \settwo$, note that by definition $y^{\tilde T(i+1)}(i)=\tilde\omega^{\tilde T(i+1)}(i)-1$. Also $y^t(i)-y^{\tilde T(i+1)}(i)$ has the same parity as $\tilde \omega^t(i) - \tilde\omega^{\tilde T(i+1)}(i)$. Combining these with $\tilde\omega^t(i) = \omega^t(i) = 0$ proves $y^t(i)$ is odd.

Finally, we show $A_{x,z}$ occurs by checking every item of its definition: we've checked item \eqref{well-defined} in the definition of $z$ above; item \eqref{no exit} holds because $t<\tau_{\text{Exit}}$; item \eqref{no visit} follows from the lower bound on $\tilde T(i)$ in Lemma \ref{stopping time}; item \ref{two steps} follows from Lemma \ref{pao} \eqref{adjacent} and Lemma \ref{stopping time}.
\end{proof}

Next we prove Lemma \ref{healthy}.

%\subsection{Proof of Lemma \ref{healthy}}

\begin{proof}[Proof of Lemma \ref{healthy}]
We will estimate the probability inductively. We start with $p=\tilde M(x)$  and then progressively lower it until we get the full result.

For $p=\tilde M(x)$, $\left.x\right|_{[\tilde M(x),a]}$ and $\left.z\right|_{[\tilde M(x),a]}$, define the $j$-th iteration of the counter $y^s(\tilde M(x))$ to be the set of paths $Q^s$ such that $y^{s+1}(\tilde M(x)) = j$ and $L(\omega^s) \geq \tilde M(x)$, for any $j \in [1, z(\tilde M(x)]$. For each iteration of the counter, we may sample an infinite sequence of paths and reveal as many of them as needed. In each iteration we cannot have the first $a-\tilde M(x)+1$ paths all being excursions to their right; otherwise, the leftmost one in the configuration would exceed $a$, which contradicts the definition of $A_{x,z}$ item \ref{no exit}. Since the probability of any path reaching a neighboring block is at most $2/a^4$ (see the proof of Lemma \ref{snack}), the probability of each path being an excursion to the right (including a long excursion to the right, but excluding the double-sided path and the failed re-arrival) is at least $1/2-1/a^4$. Thus we get the following upper bound on the probability of $A_{x,z}$
$$\left(1-(1/2-1/a^4)^{\,- \tilde M(x)+a+1}\right)^{z(\tilde M(x))}.$$

Suppose we have established the upper bound $U\!B_{x,z,p}$ for $p$, $\left.x\right|_{[p,a]}$ and $\left.z\right|_{[p,a]}$. 
Let $p'=\max\{j<p:x(j)=0\}.$ We will establish the bound $U\!B_{x,z,p''}$ for $p''$, $\left.x\right|_{[p'',a]}$ and $\left.z\right|_{[p'',a]}$, where $p'' = p'$ if $p' \in \setone \cup \setthree$ and $p'' = p'-1$ if $p' \in  \settwo$.

If $p' \in \setone$, i.e. $r(p')=1$ and $x(p'-1)\neq*$, then by the definition of $A_{x,z}$ item \ref{no visit} we know that every movement of a particle from $p'$ in the time interval $[\tilde T(p), \tilde T(p'))$ must be to the left. By the definition of the counter, there are $$y^{\tilde T(p')} - y^{\tilde T(p)} = z(p') - (\tilde \omega^s(i)-1) \geq z(p')$$ many such movements. Thus we get an upper bound of $$U\!B_{x,z,p}\,2^{-z(p')}.$$

If $p' \in \settwo$, i.e. $r(p')=1$ and $x(p'-1)=*$, then the same analysis as in the last case $p
' \in \setone$ implies that there are at least $z(p')$ many movements from $p'$ during $[\tilde T(p), \tilde T(p'))$, all of which are to the left. Moreover, by the definition of $A_{x,z}$ item \ref{two steps}, at the last such movement from $p'$ the particle makes an excursion to the left (including a long excursion to the left, but excluding the double-sided path and failed re-arrival) that visits $p'-1$ an even number of times. This happens with probability at most $1/6$. In fact, for a \textit{usual} random walk excursion starting from zero, the probability of going to the left and visiting site $-1$ even times is exactly $1/6$. For our purpose, the random walk respawns at either the left or the right endpoint of the current block after reaching a neighboring block. In either case, a straightforward coupling argument gives the $1/6$ upper bound. Therefore, we obtain a slightly improved bound of $$U\!B_{x,z,p}\,(1/6)(1/2)^{z(p')-1}.$$

If $p' \in \setthreek$, i.e. $r(p') = r$ with $r \geq 2$, then by the definition of $A_{x,z}$ item \ref{no visit}, we know that every time a random walk path starts at a location in $[p', p'+r)$ during $[\tilde T(p), \tilde T(p'))$, the first step either moves to the left or moves to the right without hitting $p'+r$. This has probability at most $1-1/2r$. Also, every time during $[\tilde T(p), \tilde T(p'))$ the particles moves from $p'$ in a path that starts to the left of $p'$, the particle either moves to the left or moves to the right and does not hit $p'+r$ before returning to $p'$, which also has probability at most $1-1/2r$. Combining these with the definition of the counter, we get the bound $$U\!B_{x,z,p}(1-1/2r)^{z(p')}.$$
Putting these together gives us the lemma.
\end{proof}

Finally, we complete the argument by proving Lemma \ref{ugly computation}.

\begin{corollary} \label{cr7}
If $\key$ occurs and $p, x$ are defined as in Lemma \ref{union}, then $x(p)=\,\qm$ and
\begin{equation} \label{lisbon}
\tilde M(x) \geq a - N(\omega^t)-1 \geq a(1-\epsilon)-1,
\end{equation}
\begin{equation} \label{manchester}
|\setone| + 2|\settwo| + 2|\setthree| \geq a - L(\omega^t) - N(\omega^t) - 2 \geq a(1-2\epsilon) - 2,
\end{equation}
\begin{equation} \label{turin}
|\setthree| \leq \sum\nolimits_r (r-1) |\setthreek| \le N(\omega^t) \le a\epsilon,
\end{equation}
\begin{equation}\label{superstition}
	|\setqmit| := |\{j \in [p,a]: x(j) = \qm\}| \leq N(\omega^t)+2 \leq \epsilon a+2.
\end{equation}
\end{corollary}

\begin{proof}
For \eqref{lisbon}, there are only $\qm$'s to the right of $\tilde M(x)$: at most $N(\omega^t)-1$ of them correspond to $?$'s and $1$'s in $\omega^t$, and at most two of them come from zeros in $\omega^t$ by the choice of $\tilde M$ in the definition of $x$ and Lemma \ref{pao} \eqref{monotone}\eqref{nofour}. For \eqref{turin}, notice the convention that $\tilde N(x) \in \setthree$ and $r(\tilde N(x)) = \tilde M(x)+1 - \tilde N(x)$.
\end{proof}

%\subsection{Proof of Lemma \ref{ugly computation}}

\begin{proof}[Proof of Lemma \ref{ugly computation}]
By Lemma \ref{union}, it suffices to bound
$$\sum_{x \text{ good}}\sum_{z }\p(A_{x,z}),$$
where the sum is taken over all good $x$'s satisfying the bounds in Corollary \ref{cr7} and all $z \in \N^\setzero$ satisfying the parity constraint in Lemma \ref{union}.

%Recall that in the proof of Lemma \ref{union}, we find $t < \tau_{\text{Exit}}$ such that $L(\omega^t) \le \epsilon a$ and $N(\omega^t) \le \epsilon a$ and use $\omega^t$ to define a good $x$.

%It turns out that the property of $\omega^t$ says more about what good sequence $x$ could be. Remember from \eqref{lisbon} in Corollary \eqref{cr7}
%$$\tilde M(x) \geq a - N(\omega^t) \geq a(1-\epsilon).$$
%Now we sum up the remaining terms in the product in the upper bound in Lemma \ref{healthy} over all appropriate choices of $\hat s$, $\setfive$ and $z$.
By \eqref{lisbon}, the leading term of the product in Lemma \ref{healthy} is at most $$(1-2.1^{-2-a\epsilon})^{z(\tilde M(x))}$$
for sufficiently large $a$.
Summing up over all possible values of $z(\tilde M(x))$ gives us
\begin{equation} \label{house of}
\frac{1}{1-(1-2.1^{-2-a\epsilon})}\leq 5 \cdot 2.1^{a\epsilon}.
\end{equation}
From \eqref{manchester} and \eqref{turin}, we get 
%$$|\setone| + 2|\settwo| + |\setthree| = a - L(\omega^t) - N(\omega^t) - 1 \geq a(1-2\epsilon) - 1.$$
%Since $|\setthree| \le 1+ N(\omega^t) \le 1 + a\epsilon$, we get
\begin{equation}\label{setone bound}
	|\setone| \geq a(1 - 4\epsilon) - 2|\settwo| - 2.
\end{equation}
%Also choose $a$ large enough such that
%\begin{equation}\label{motown}
%	p_{\frac{1}{6}} = 1/6 + 1/a^4 \leq 0.334\cdot1/2.
%\end{equation}
Finally, from \eqref{turin} we have
%\begin{equation*} \label{small setthree}
%\sum_r (r-1) |\setthreek| \le 2 + N(\omega^t) \le 2+a\epsilon,
%\end{equation*}
%where the extra 2 comes from the convention of $r(\tilde N(x))$, so we have
%$$\sum_r r |\setthreek| \le \sum_r 2(r-1) |\setthreek| \leq 2a\epsilon,$$
%which implies
\begin{equation} \label{cards}\prod_r (2r)^{|\setthreek|} \leq 2^{\sum_r r |\setthreek|} \leq 2^{\sum_r 2(r-1) |\setthreek|} \leq 2^{2a\epsilon}.
\end{equation}

Fix $x$ and write $m := |\settwo|$.
By Lemma \ref{healthy} and the distributive property, we have
\begin{eqnarray*}
\lefteqn{ \sum_{z} \p(A_{x,z})} &&\\
&\leq &
5\cdot 2.1^{a\epsilon} 
\prod_{i \in \setone}\left( \sum_{z \geq 1}2^{1-2z}\right)
\prod_{i \in \settwo}\left( \sum_{z \geq 1}1/3\cdot2^{1-2z}\right)
\prod_r\prod_{i \in \setthreek}\left(\sum_{z \geq 0} \left( 1-\frac{1}{2r}\right)^{z}\right)\\
&\leq & 5\cdot  2.1^{a\epsilon} (2/3)^{a(1-4\epsilon)-2m-2}(2/9)^m
\prod_r (2r)^{|\setthreek |}\\
&\leq &15 \cdot 2.1^{a\epsilon} (2/3)^{a(1-4\epsilon)}(1/2)^{m}2^{2a\epsilon}\\
&\leq &15 \cdot 50^{\epsilon a} (2/3)^a (1/2)^{m},
\end{eqnarray*}
where we've used the fact that all $z(i)$'s are odd for $i \in \setone \cup \settwo$ and
equations \eqref{house of}, \eqref{setone bound} and \eqref{cards}.

Note that the bound we just developed only depends on $m=|\settwo|$. For $m \in \N$ let $$\mathcal W(m) := \{x: x \text{ is good, } x(p)=\,\qm, |\settwo|=m \text{ and } |\setqm| \leq \epsilon a +2\}.$$
Write $$S_{n, k} = \sum_{i=0}^k \binom{n}{i}.$$
%We let $$S(x)=\#\{p' \in \settwo\}$$ and develop a bound  $\p(A_{x,z}) \leq C(m)$ that only depends on $S(x)=m$
%Observe that
%\begin{equation}\label{binom}
%	W(m) \le a \binom{a}{a\epsilon}\binom{a-m}{m}.
%\end{equation}
Also note that by the binomial theorem, for $a \geq 2m$ and $\lambda\geq0$ we have
\begin{equation} \label{madrid}
\binom{a-m}{m} (1/2)^m\lambda^{a-2m} \le (\lambda+1/2)^{a-m}.
\end{equation}
Then by \eqref{superstition} we have
\begin{eqnarray*}
\sum_{x \text{ good}}\sum_{z }\p(A_{x,z}) &=&\sum_m \sum_{x \in \mathcal W(m)}\sum_{z}\p(A_{x,z})\\
&\leq & \sum_m a S_{a, \epsilon a+1} \binom{a-m}{m} 2^{\epsilon a} \cdot 15 \cdot 50^{\epsilon a} (2/3)^a (1/2)^{m}\\
&=& 10^{2\epsilon a} S_{a, \epsilon a} \cdot 15a^2(2/3)^a \sum_m\binom{a-m}{m} (1/2)^m\\
&\leq & 10^{2\epsilon a} S_{a, \epsilon a} \cdot 15a^2(2/3)^a \sum_m (7/4)^{a-m}(4/5)^{a-2m}\\
&\leq & 10^{2\epsilon a} S_{a, \epsilon a} \cdot 150a^2(14/15)^a,
\end{eqnarray*}
where in the second last inequality, we use \eqref{madrid} by picking the suitable $\lambda=5/4$.

We complete the proof of Lemma \ref{ugly computation} by taking a small enough $\epsilon$. Indeed, recall the Chernoff bound 
$$\log S_{a, \epsilon a} \leq H(\epsilon)a$$
where $H(\cdot)$ is the binary entropy function $H(x) = -x \log x - (1-x) \log(1-x)$. For $\epsilon = 1/200$,
$$\log(14/15) + 2\epsilon \log 10 + H(\epsilon) < -\frac{1}{100},$$
so the above probability is upper bounded by $\exp(-a/100)$ for all $a$ sufficiently large. 
\end{proof}

%%%%%%%%%%%%%%%%%%%%%%%%%%
%%%%%%%%%%%%%%%%%%%%%%%%%%
\section{Single block estimate} \label{sec:renewal}
%%%%%%%%%%%%%%%%%%%%%%%%%%
%%%%%%%%%%%%%%%%%%%%%%%%%%

Using the results of the previous sections, which show that the carpet process $\omega$ is well behaved in terms of the number of ?'s and 1s, we return to the main task of bounding the probability that a particle is frozen and proving Lemma \ref{lem:sbe}.  

Define
%$$R(\omega)= \inf\{x>L(\omega): \omega(x)=0\}.$$
$$A =\{L(\omega) \leq \epsilon a  \text{ and } N(\omega)> \epsilon a \}\subset \Omega.$$
%$$B_k =\{R(\omega)-L(\omega) \geq k\}\subset \Omega.$$
Recall
$$\text{Base}=01??\cdots??$$
$$\text{ Exit}=0000\cdots00$$
$$\tau_{\text{Exit}} =\inf\{t\geq 0: \omega^{t}=\text{Exit}\}$$
as well as the definition of an $\epsilon$-Base state.

%the law of the Markov chain $\omega^t$ started from state $\omega^0 = \omega_0$. 

Our first lemma is obtained by analyzing the bias in the process $N(\omega^t)$. It says that starting with a large enough $N(\omega_0)$, the process is much more likely to return to a state $L(\omega) \leq \epsilon a$ than become frozen.

\begin{lemma} \label{nice computation}
For any state $\omega_0 \in A$ and sufficiently large $a$,
$$\overline{\p}_{\omega_0}(\tau_{\text{Exit}}< a^3 \wedge \inf\{t>0: L(\omega^t)\leq \epsilon a\}) \leq e^{-\frac{1}{3} 10^{-6} \sqrt{a}}.$$
\end{lemma}

\begin{proof}
%This is just a calculation on a biased random walk.

Consider the value of $$\Delta_t N \vcentcolon = N(\omega^{t+1}) - N(\omega^t) = (N(\omega^{t+1}_\ast) - N(\omega^t)) + (N(\omega^{t+1}) - N(\omega^{t+1}_\ast)).$$
%and recall the notation from Definition \ref{ell_defs}.
Let $\ell(Q^t) := L^t - \min(\range(Q^t))$ be the maximum distance reached by the random walk path $Q^t$ to its left. By the definition of the `refresh' step, the first term
$$N(\omega_\ast^{t+1}) - N(\omega^t) \geq (\ell(Q^t) \wedge L^t) -2.$$
When $L^t > \epsilon a$ and $\ell(Q^t)>0$, this is at least $(\ell(Q^t) \wedge \epsilon a) - 2$; otherwise, we have the lower bound $-2$.
Using Lemma \ref{geo}, the second term
$$N(\omega^{t+1}) - N(\omega^{t+1}_\ast)\succ -\text{Geo}(p) - 1,$$
where $p = 1-\sqrt{5/6}$ and $\succ$ is stochastic domination.

%In all cases, we may use Lemma \ref{geo} to bound $\Delta_t N$ from below. We have in particular by Lemma \ref{geo} that when $Q^t$ is to the right,
%$$(\Delta_t N) \succ -\text{Geo}(p) - 3,$$
%where $p = 1-\sqrt{5/6}$ and $\succ$ is stochastic domination.
%(The $-1$ is due to the parity at site $L(\omega^t)$, which always switches from $1$ to $0$, while the Geometric random variable counts the number of $?$'s that are revealed to be $0$'s.) 
%$$\E[(\Delta_t N)1\{Q^t \text{ is to the right}\}] \geq -\sum_{q \geq 0} (5/6)^{q/2 - 1} \geq -14.$$

%If the excursion $Q^t$ is to the left and $L(\omega^t) \geq \epsilon a$, then
%\begin{equation} \Delta_t N \geq (\hat \ell \wedge L(\omega^t)) - 2 \geq (\hat{\ell}(Q) \wedge \epsilon a) - 2, \end{equation}
%where $Q$ is a simple random walk excursion conditioned to have maximum $\hat{\ell}(Q) < a^2$. By Lemma \ref{loss}, for any $1 < q < a^2$, 
%$$\p(\hat{\ell}(Q) > q) \geq \p(\text{Extreme}(Q) > 2q) (1-(5/6)^{q}).$$

%so
%$$\E[\hat{\ell}(Q) \vee \epsilon a] \geq \sum_{q \geq 1}^{\epsilon a / 2} (1-(5/6)^q) \frac{1}{2q} \geq \frac{1}{30} \log(\epsilon a / 2) - 2 > 20.$$

%Note that while $\omega^t > \epsilon a$, these bounds on $\Delta N_t$ are uniform over the value of $N_t$. So

Combining both estimates, we obtain that for $t < \inf \{s > 0: L(\omega^s) \leq \epsilon a\}$, $N(\omega^t) - N(\omega^0)$ stochastically dominates a sum of the form 
$$S_t = \sum_{i=1}^{B_{t-1}} Z_i - \sum_{i=1}^{t} (Y_i + 3),$$
where $B_{t-1} \sim $ Binomial$(t-1, 1/2)$, $Y_i \sim $ Geo$(p)$ are iid, and $Z_i$ are iid positive variables with tail $\p(Z = q) \geq \frac{1}{2q^2}$ for $q \in [1, \epsilon a]$. Our aim is to show that for $t < a^3$,  $S_t \geq - \epsilon a$ with high probability, which implies that with the same probability, for $t < a^3 \wedge \inf \{s > 0: L(\omega^s) \leq \epsilon a\}$, $N(\omega^t) > 0$, i.e. $\tau_{\text{Exit}}$ has not occurred.

To prove this, first observe that throwing out the positive sum in $S_t$ and using a tail bound for a sum of Geometric random variables,

\begin{equation} \p(S_{\epsilon a / 100} < - \epsilon a / 2) \leq \p(\sum_{i=1}^{\epsilon a / 100} Y_i > \epsilon a/3) \leq \exp(-\epsilon a / 100). \end{equation}

%(Such a bound can be obtained in many ways: for example, exponentiating and using Markov's inequality, and the formula for the MGF of a Geometric variable, one obtains, for any $\lambda < - \log(1-p)$, 

%$$\p(\sum_{i=1}^{\epsilon a / 100} Y_i > \epsilon a / 2) \leq \exp(-\epsilon a \lambda/2) \left(\frac{p}{1-q e^\lambda} \right)^{\epsilon a / 100}.$$

%Now take $\lambda = 1/20$ and check some decimals.)

To see that $N(\omega^t)$ stays away from 0 at later times, we carry out the following concentration bound for the sum of $Z_i$. Observe that by standard tail bounds for Binomials (e.g. a Chernoff bound), for any $q = 1, 2, \ldots, a^{1/4}$ and any $t > \epsilon a / 100$, 

\begin{align} \p\left(\#\{i \leq B_{t-1}: Z_i = q\} \leq \frac{1}{12} t q^{-2}\right) &\leq \p(\text{Bin}(t-1,\, \frac{1}{2} \cdot \frac{1}{2} q^{-2}) \leq \frac{1}{12} t q^{-2}) \\
&\leq \exp(- t q^{-2} / 72) \\
&\leq \exp(- \epsilon \sqrt{a} / 7200). \end{align}
By a union bound over all such $q$, none of these events occurs with probability at least $1 - \exp(-\frac{\sqrt{a}}{2*10^6})$ if $a$ is sufficiently large, and if none of them occurs, we have 

$$\sum_{i=1}^{B_{t-1}} Z_i \geq \sum_{q=1}^{a^{1/4}}  \frac{1}{12} t q^{-1} \geq \frac{1}{50} t \log a.$$

For the negative part of $S_t$, again using a similar tail bound for a sum of Geometrics, for any $t > \epsilon a / 100$, 
$$\p\left(\sum_{i=1}^{t} (Y_i + 3) \geq t(3 + 10 p^{-1})\right) \leq \exp(-a/10^4).$$
Combining these bounds, since we can take sufficiently large $a \geq e^{6000}$ so that $\frac{1}{50} \log(a) > 3 + 10p^{-1}$, we obtain for any $t > \epsilon a / 100$, 
$$\p(S_t < -\epsilon a) \leq \exp(-\frac{1}{2} 10^{-6} \sqrt{a}).$$
Taking a union bound over all values $t \leq a^3$ gives the result.

%$$\overline{\p}_{\omega_0}(N_t = 0 \text{ for some } t \leq a^3 \vee \inf \{s > 0: L(\omega^s) \leq \epsilon a\}) \leq \exp(-c \sqrt{a}).$$

%Thus, as long as $L(\omega^t) > \epsilon a$, excursions to the left yield a large positive value of $\Delta_t N$ -- something like $\log \epsilon a$ on average -- while excursions to the right yield a relatively small negative value by Lemma \ref{geo} -- something bounded in expectation, with geometric tail. So when $L(\omega^t)\geq \epsilon a$ the process $N(\omega^t)$ stochastically dominates a random walk with a positive drift. It follows that with exponentially high probability, $N$ will stay away from 0 for long enough to reach $L(\omega) < \epsilon a$, for example returning to Base $\in A$ via Lemma \ref{LR 3 zoning}. 

% If the excursion moves to the right, then 
%$$\Delta_t N = -1 + \sum_{L^t<i\leq a} (\tilde\omega^{t+1}(i) - \tilde\omega^{t}(i)).$$ 
\end{proof}

The next lemma gives a lower bound on the frequency at which the carpet process $\omega^t$ refreshes by returning to the Base state. 

\begin{lemma} \label{LR 3 zoning}
The following hold for $a$ sufficiently large. If $\omega_0$ satisfies $L(\omega_0) \leq a$, 
\begin{equation}\underline{\p}_{\omega_0}(\omega^1 = \text{Base} \text{ or }\omega^2 = \text{Base}) \geq \frac{1}{600 a^2}. \end{equation} 
Additionally, 
\begin{equation} \overline{\p}_{\omega_0}(\tau_{\text{Base}} > a^3) \leq \exp(-a/2000), \end{equation}
where $\tau_{\text{Base}} =  \inf\{t > 0: \omega^t = \text{Base}\}.$

\end{lemma}

\begin{proof}
The second claim follows immediately by repeated application of the  first claim, 
%by the Markov property for $\omega^t$, 
bounding the probability of hitting the Base state independently for every successive pair of times:

$$\overline{\p}_{\omega_0}(\tau_{\text{Base}} > a^3) \leq (1-\frac{1}{600a^2})^{a^3/2} \leq \exp(- a/2000).$$

One way to return to the base state is by first taking an excursion to the left beyond 0 and not matching the parity at 1 -- this makes $L(\omega^1)=1$ -- and then, at the next step, taking an excursion to the right beyond $a$ and not matching the parity at 2. Recall that the probability for a simple random walk excursion to reach distance at least $\ell'$ is $\frac{1}{\ell'}$, so the probability for an excursion to reach distance $\ell' \in [a, a^2)$ is at least $\frac{1}{2a}$. Note that by the choice of $K$, an excursion that reaches maximum distance $\ell' \in [a, a^2)$ visits every point in the block but does not reach any neighboring block. Thus if $L(\omega_0) \geq 2$, then by Lemma \ref{loss} 
%\begin{equation} \underline{\p}_{\omega_0}(\omega^1(1) = 1 | \ell(Q^0) \in [a, a^2), Q^0 \text{ is to the left}) \geq \frac{1}{6}. \end{equation}
%Similarly, 
%\begin{equation} \underline{\p}_{\omega_0}(\omega^2(2) = 1 | \ell(Q^1) \in [a, a^2), Q^1 \text{ is to the right}) \geq \frac{1}{6}. \end{equation}
%Thus 
$$\underline{\p}_{\omega_0}(\omega^2 = \text{Base}) \geq \left(\frac{1}{2} \cdot \frac{1}{2a} \cdot \frac{1}{6}\right)^2 \geq \frac{1}{600a^2}.$$
If $L(\omega_0) = 1$, only the second step is necessary. 
%When $N(\omega)=0$ then you first take a short excursion to the left and then a slightly longer excursion to the right, not matching too many unrevealed digits. Finally you take a long excursion to the left.
\end{proof} 

%Before moving forward we first prove a buffed up version of Lemma \ref{ugly computation}, which allows for emissions to occur:
%
%\begin{lemma} \label{ugly computation2}
%There exist $\epsilon, c>0$ such that for all sufficiently large $a$ and from  $\omega^0 =\text{Base}$, the probability of the event
%$$\key =L(\omega^t) \leq \epsilon a \text{ and } N(\omega^t) \leq \epsilon a \text{ for some } t< \tau_{\text{Exit}} $$
%is at most $e^{-ca}$.
%\end{lemma}

Now we are ready to state Lemma \ref{smith} which is the culmination of the previous two sections. It shows that in a typical scenario, the carpet process keeps refreshing itself by returning to the Base rather than become frozen.

\begin{lemma} \label{smith} For $a$ sufficiently large, 
$$\overline{\p}_{\epsilon\text{-Base}}(\tau_{\text{Exit}} < \tau_{Base}) \leq e^{-\frac{1}{4} 10^{-6} \sqrt{a}}$$
%There exists $c>0$ such that for any $a$ and from $\omega_0 =\text{Base}$ the probability that you exit before you return to the base state $(\tau_{\text{Exit}}<\inf\{t>0: \omega^t=\text{Base}\})$ is $\leq 2e^{-ca}$.
\end{lemma}

\begin{proof}
By Lemma \ref{LR 3 zoning}, the process hits the Base in time at most $a^3$ with exponentially high probability. So it suffices to bound the event that $\tau_{\text{Exit}} < a^3$.

If $\tau_{\text{Exit}} < a^3$ occurs, then either
\begin{enumerate}
\item there exists a time $t<\tau_{\text{Exit}}$ with $L(\omega^t) \leq \epsilon a$ and $N(\omega^t) \leq \epsilon a$ or 
\item there exists a time $t < \tau_{\text{Exit}} < a^3$ with $\omega^t \in A$ and $L(\omega^s) > \epsilon a$ for any $s$ such that $t < s < \tau_{\text{Exit}}$. 
\end{enumerate}
By Lemma \ref{ugly computation} the first event occurs with exponentially small probability in $a$.
%If $(1)$ does not occur, then at any time $t<\tau_{\text{Exit}}$ with $L(\omega^t) \leq \epsilon a$, we have $\omega^t \in A$.
By Lemma \ref{nice computation} and a union bound over all $t < a^3$, the probability of the second event is exponentially small in $\sqrt{a}$. Combining these estimates, we obtain that the process reaches the Exit after $a^3$ steps with exponentially high probability.
%Finally, by Lemma \ref{LR 3 zoning}, the process hits the Base in time at most $a^3$ with exponentially high probability. The result follows by a union bound. 
\end{proof}

\subsection{Proof of Lemma \ref{lem:sbe}}

Following \cite{HRR20}, we first re-parameterize the process by the number of \textit{attempted emissions} from block $i$, instead of the number of additional input particles from the right. Each attempted emission, as a portion of the carpet/hole toppling procedure described in Section \ref{sec:carpet_rules}, is defined as the evolution until the hot particle reaches a new block or until the hot particle becomes frozen.

Denote by Left$(k)$ and $F(k)$ the number of particles emitted to the left from block $i$ and the number of frozen particles in block $i$ right after the $k$th attempted emission respectively (a slight abuse of notation). Let $e(0)=0$ and $e(k)$ denote the number of steps taken by the auxiliary carpet process $\omega^t$ until the completion of the $k$th attempted emission. Note that $e(k)$ counts steps of the chain $\omega^t$, not individual topplings in the carpet/hole procedure. 

The following lemma gives a conclusion in the same spirit as Lemma \ref{smith}, but for a process not necessarily started from the Base or $\epsilon$-Base state.

%Call the $j$th time a free particle in block $i$ makes an excursion to a neighboring block the $j$th \textbf{emission} for $j = 0, 1, 2, \ldots$ and let $e(j)$ denote the number of steps taking by the hidden Markov chain $\omega^t$ between emissions, with $e(0) = 0$. Note $e(j)$ counts steps of the chain $\omega$, not individual topplings in the carpet/hole procedure. 

%Let $\mathcal{G}_j$ denote the $\sigma$ field generated by all stack instructions used by the carpet/hole procedure up to time $e(j)$ (including $\mathcal{F}_{i-1}$). 

%I commented out the conditioning here. We have it in the original paper, I couldn't figure out why though, perhaps there is a subtle difference I'm missing, or I found a simpler argument. Our bounds are uniform over the starting location, things work nicely because even starting from a bad config, we hit Base even before we have a chance to exit. perhaps this is the divergence with the other paper.

\begin{lemma} \label{snack} The following holds for $a$ sufficiently large. If the initial carpet $(\eta, \omega)$ is valid, for any $k \geq 2$, almost surely,
\begin{equation} \p_{(\eta, \omega)}(F(k) = 1 | \mathcal{F}_{i-1}) \leq 8a^{-1}. \end{equation}
\end{lemma}
%The same bound holds for any $k \geq 0$ if the initial carpet is fully valid.

%\begin{lemma} \label{porto} For any $j \geq 1$, all $a$ sufficiently large and any $\omega_0$,
%$$\overline{\p}_{\omega_0}(\omega^{e(j)} \notin A) \leq 5a^{-1}.$$
%\end{lemma}

%To prove the single block estimate, we need to re-index by a slightly different counter: attempted emissions. Let $\sigma(1) = 0$, and for $k = 2, 3, \ldots$, let $\sigma(k)$ denote the number of steps taken by the hidden Markov chain $\omega^t$ the $(k-1)$st time a free particle was emitted from block $i-1$ and then declared hot at site $0$ in block $i$. %Let $\mathcal{H}_k$ denote the $\sigma$-field generated by all stack instructions up to time $\sigma(k)$. 
%Also, denote by Left$(k)$ and $F(k)$ the number of particles emitted to the left from block $i$ and the number of frozen particles in block $i$ after the $k$th attempted emission, respectively (a slight abuse of notation). The following lemma gives a conclusion similar to that of Lemma \ref{porto}, but at the attempted emission times instead of the emission times. 

\begin{proof}
We consider three cases.

\textbf{Case I}: $F(k-2) = 0$ and $F(k-1) = 0$.
Note that the probability of each attempted emission being followed by a failed re-arrival is at most $a/K=1/a^{3}$. In the following, we also assume that $Q^{e(k-2)-1}$ and $Q^{e(k-1)-1}$ are not failed re-arrivals.

Under those assumptions, during the $k-1$th and $k$th attempted emissions, the carpet process keeps performing random walk excursions from the hole, until one such excursion reaches a neighboring block.
Recall that the probability that a random walk excursion reaches maximum distance at least $l$ is $\frac{1}{l}$, and the distance from any point in block $i$ to a neighboring block is at most $K$ and at least $K/2$. Thus each excursion of the hot particle in block $i$ has probability at least $a^{-4}$ and at most $2 a^{-4}$ of reaching a neighboring block, uniformly over the initial position of the excursion (inside block $i$). Elementary estimates give that
\begin{equation} \label{diversey}
	\overline{\p}_{\omega_0}(e(k-1) - e(k-2) < a^3) \leq {a^3}\cdot 2a^{-4} \le 2a^{-1}
	%1-(1-2a^{-4})^{a^3}
\end{equation}
and
\begin{align}
	\overline{\p}_{\omega_0}(e(k) - e(k-2) > 2a^5) &\leq (1-a^{-4})^{2a^5} + 2a(1-a^{-4})^{2a^5-1}\\
	& \le a^{-1}.
\end{align}
%Note that \eqref{diversey} is no longer true in Case II.

%\begin{align} \overline{\p}_{\omega^0}(e(k-1) - e(k-2) > a^3 \text{ and } e(k) - e(k-2) < 2a^5) &\leq 1-(1-2a^{-4})^{a^3} + (1-a^{-4})^{a^{5}} \\
%&\leq 4a^{-1}. \end{align}

Following the comment in Section \ref{sec:aux}, we consider the auxiliary carpet process started at $e(k-2)$, and due to the assumption on $Q^{e(k-1)-1}$ all results since Section \ref{sec:carpets} apply until time $e(k)$.
By Lemma \ref{LR 3 zoning}, for any $k \geq 2$ we have
\begin{equation} \overline{\p}_{\omega_0}(\omega^t \neq \text{Base, } t \in [e(k-2), e(k-2) + a^3]) \leq \exp(-a/2000). \end{equation} 
By Lemma \ref{smith} and the previous line, since there are at most $2a^5$ returns to the Base state up to time $2a^5$, for any $k \geq 2$, 
\begin{equation} \overline{\p}_{\omega_0}(\omega^t = \text{Exit for some } t \in [e(k-2)+a^3, e(k-2) + 2a^5]) \leq 2a^5 \exp(-\frac{1}{4}10^{-6} \sqrt{a}). \end{equation}
Combining all these estimates and taking $a$ sufficiently large gives an upper bound $4a^{-1}$ in the first case.
%The probability that the $j$th emission happens within the first $a^3$ steps is at most $1/a$. 
%The probability that the $j$th emission happens after the first $a^5$ steps is less than $e^{-a/2}$. 
%The probability that the process has not reached the base state after $a^3$ steps is exponentially small. 
%If the process reaches the base state then with probability at least $1-a^5e^{-ca} $ it doesn't leave $A$ until it has returned to the base state at least $a^5$ times.
%Thus the probability that the $j$th emission happens with $\omega^t \not \in A$ is at most $2/a$. 

\textbf{Case II}: $F(k-2) = 1$. By \ref{pro:frozen_run}, we have $e(k-1)=e(k-2)+1$, cf. \eqref{diversey}, and in the $(k-1)$th attempted emission, we topple the hot particle started from $0$ or $a$ until it reaches a neighboring block. Since we've chosen $K = a^4$, the probability that the hot particle reaches a neighboring block without visiting the entire block $[0,a]$ is at most $a/K=1/a^{3}$: it must arrive at one endpoint of $[0,a]$ first, and then reach a neighboring block without touching the other endpoint. By Lemma \ref{loss_ouch} for type 1 paths,
\begin{equation}
\overline{\p}_{\omega_0}(\omega^{e(k-1)} \text{ is not } \epsilon\text{-Base} \mid [0,a] \subset \text{Range}(Q^{e(k-1)-1})) \leq (5/6)^{\epsilon a} \leq \exp(-a/1200).
\end{equation}
Moreover, a similar argument using an upper tail bound on $e(k)-e(k-1)$ and Lemma \ref{smith} as in Case I shows that
\begin{equation}
\overline{\p}_{\omega_0}(F(k)=1 \mid \omega^{e(k-1)} \text{ is }\epsilon\text{-Base}) \leq a^5\exp{\frac{1}{4}10^{-6}\sqrt{a}}.	
\end{equation}
Combining all these, we obtain an upper bound $2a^{-1}$ for sufficiently large $a$ in the second case.

\textbf{Case III}: $F(k-1) = 1$. By \ref{pro:frozen_run} and Lemma \ref{loss_ouch}, with probability at least $1-2/a^3$ we have $\omega^{e(k)}$ is not an all-zero configuration and thus $F(k)=0$.

Taking a union bound over the three cases completes the proof.
\end{proof}

\begin{proof}[Proof of Lemma \ref{lem:sbe}] We use Lemma \ref{snack} to prove the single block estimate. We will focus on the first statement because the second one will be shown in \eqref{depot} as an intermediate step. Fix an $l \geq 0$, and block $i \in \{0, \ldots, n-1\}$. Our aim is to uniformly bound the expression
\begin{equation} \E\left[\sum_{s \geq 2} e^{c F_i^i(s)} 1\{L_i^i(s) = l\} \middle| \mathcal{F}_{i-1}\right].\end{equation}
%for either $m_0 = 0$ or $2$.
Note that each particle added at site $iK+a$ causes at least one attempted emission in block $i$ (possibly more if additional particles arrive from the left of block $i$ as a result). Thus the sum is upper bounded by
\begin{equation} \widetilde{\E}\left[\sum_{k \geq 2} e^{c F(k)} 1\{\text{Left}(k) = l\}\right], \end{equation}
where we use $\widetilde{\E}$ for the conditional expectation w.r.t. $\mathcal F_{i-1}$. Set
\begin{equation} \tau_l = \inf\{k \geq 0: \text{Left}(k) = l\}, \end{equation}
and re-write the latter sum (over $k \geq 2$) as 
\begin{equation}\label{lugless}
 \widetilde{\E}\left[\sum_{k = \tau_l \vee 2}^{\tau_{l+1}-1} e^{c F(k)}\right]
= \widetilde{\E}\left[ \sum_{k \geq 0} 1\{\tau_{l+1} - (\tau_l\vee2) > k\} e^{c F((\tau_l\vee2) + k)} \right]. 
 \end{equation}

We claim that for any $k, l \geq 0$,
\begin{equation}\label{depot} \widetilde{\p}(\tau_{l+1} - \tau_l \geq k) \leq (2/3)^{\lfloor k/2 \rfloor}\end{equation}
and
\begin{equation}\label{fedex}
	\widetilde{\E}[e^{cF((\tau_l\vee2)+k)}] \leq 1 + 100e^ca^{-1}(k+\log(a)).
\end{equation}

For \eqref{depot}, note that either the $k+1$st attempted emission is a successful emission, or it is frozen in which case the $k+2$nd attempted emission is emitted to a neighboring block with probability 1 by \ref{pro:frozen_run}. By gambler's ruin, the probability to be emitted to the right in either case is at most $\frac{K}{2K-a} \leq 2/3$. Thus
\begin{equation} \widetilde{\p}(\text{Left}(k+2) = \text{Left}(k)) \leq 2/3. \end{equation}
This proves the first claim.

For \eqref{fedex}, when $l=0$, we have $\tau_l=0$ and the claim follows directly from Lemma \ref{snack}. For $l \geq 2$, note that $\tau_l \geq \tau_{l-2} + 2$. By \eqref{depot},
\begin{equation}
\widetilde{\p}(\tau_l - \tau_{l-2} > k') \leq 2(\frac{2}{3})^{\lfloor k'/4\rfloor}	
\end{equation}
Moreover, Lemma \ref{snack} implies that
\begin{equation}
\widetilde{\p}(F(j) = 1 \text{ for some } j \in [\tau_{l-2}+2, \tau_{l-2}+k'+k]) \leq 8(k+k')a^{-1}.
\end{equation}
Combining these two estimates into a union bound by setting $k'=\lceil 12 \log(a) \rceil$ yields the result. When $l=1$, a similar argument works by replacing $\tau_{l-2}$ with 0. This proves \eqref{fedex}.

Now we turn back to the task of bounding the sum \eqref{lugless}. Putting the above bounds \eqref{depot} and \eqref{fedex} together, we get
\begin{equation} \widetilde{\E}\left[1\{\tau_{l+1} - (\tau_l\vee2) > k\} e^{c F((\tau_l\vee2) + k)}\right] \leq \min \{1 + 100e^ca^{-1}(k+\log(a)), (2/3)^{\lfloor k/2 \rfloor} e^c \}. \end{equation}
Splitting the sum \eqref{lugless} at $k = \lceil 6 \log(a) \rceil$, and using both bounds, the original sum is bounded by 
\begin{equation} c' = 7\log(a) + 2500 e^c \log^2(a) a^{-1} \end{equation}
Finally, we pick $a = e^{1.1c}$ and large enough $c \geq 4\times10^4$ so that the condition $\delta c = 0.0004c > \log c'$ holds, which proves Lemma \ref{lem:sbe} for large enough $a$.

It remains to obtain the explicit values for $a$ and $K$. Note that throughout Sections \ref{sec:ugly} and \ref{sec:renewal}, we've taken $a$ sufficiently large multiple times. Two sharpest constraints, however, are $a \geq e^{6000}$ in the proof of Lemma \ref{nice computation} to achieve the bias of the random walk, and $a \geq e^{4.4 \times 10^4}$ above for Lemma \ref{lem:sbe} to work under a very small $\delta=.0004$. Thus, we conclude that Lemma \ref{lem:sbe} holds for $a\geq e^{4.4 \times 10^4}$ and $K = a^4 \geq e^{1.76\times 10^5}.$
\end{proof}

\bibliographystyle{plain}
\bibliography{leo}
\end{document}